\documentclass[11pt,reqno]{amsart}

\usepackage[T1]{fontenc}
\usepackage[utf8]{inputenc}
\usepackage{lmodern}
\usepackage{amsmath,amssymb,amsthm,mathtools,mathrsfs}
\usepackage{esint}
\usepackage{enumitem}
\usepackage{aliascnt}
\usepackage{microtype}
\usepackage[colorlinks=true,linkcolor=blue,citecolor=blue,urlcolor=blue]{hyperref}
\usepackage[nameinlink,capitalize,noabbrev]{cleveref}
\usepackage[a4paper,margin=1.15in]{geometry}
\numberwithin{equation}{section}

\newtheorem{theorem}{Theorem}[section]

\newaliascnt{proposition}{theorem}
\newtheorem{proposition}[proposition]{Proposition}
\aliascntresetthe{proposition}

\newaliascnt{lemma}{theorem}
\newtheorem{lemma}[lemma]{Lemma}
\aliascntresetthe{lemma}

\newaliascnt{corollary}{theorem}
\newtheorem{corollary}[corollary]{Corollary}
\aliascntresetthe{corollary}

\theoremstyle{definition}
\newaliascnt{definition}{theorem}
\newtheorem{definition}[definition]{Definition}
\aliascntresetthe{definition}

\crefname{theorem}{Theorem}{Theorems}
\Crefname{theorem}{Theorem}{Theorems}
\crefname{proposition}{Proposition}{Propositions}
\Crefname{proposition}{Proposition}{Propositions}
\crefname{lemma}{Lemma}{Lemmas}
\Crefname{lemma}{Lemma}{Lemmas}
\crefname{corollary}{Corollary}{Corollaries}
\Crefname{corollary}{Corollary}{Corollaries}
\crefname{definition}{Definition}{Definitions}
\Crefname{definition}{Definition}{Definitions}

\newcommand{\R}{\mathbb R}
\newcommand{\C}{\mathbb C}
\newcommand{\LipLip}{\operatorname{Lip}_{\mathrm c}}
\newcommand{\capw}{\operatorname{cap}_{2,\omega}}
\newcommand{\capwt}{\operatorname{cap}_{2,w}}
\newcommand{\supp}{\operatorname{supp}}
\newcommand{\one}{\mathbf 1}
\newcommand{\dist}{\operatorname{dist}}
\newcommand{\loc}{\mathrm{loc}}
\newcommand{\avgint}{\fint}
\newcommand{\Hpi}{\mathcal H_{\pi}}
\newcommand{\Hm}{\mathcal H_m}
\newcommand{\Vpi}{\mathcal V_{\pi}}
\newcommand{\Vm}{\mathcal V_m}
\newcommand{\PhiPi}{\Phi_{\pi}}
\newcommand{\cH}{\mathcal H}

\title[Capacitary domination and Fefferman--Phong inequalities]{Capacitary domination and Fefferman--Phong inequalities for measure potentials on $2$-PI spaces}
\author{Tan Duc Do}
\thanks{Faculty of Applied Sciences, Ho Chi Minh City University of Industry and Trade, Ho Chi Minh City, Vietnam. Corresponding author: \href{mailto:tanducdo.math@gmail.com}{tanducdo.math@gmail.com}.}
\date{}


\hypersetup{
  pdftitle={Capacitary domination and Fefferman--Phong inequalities for measure potentials on 2-PI spaces},
  pdfauthor={Tan Duc Do},
  pdfsubject={Capacitary domination and Fefferman--Phong inequalities for singular measure potentials on 2-PI spaces},
  pdfkeywords={mixed Poincare inequality, Fefferman--Phong inequality, measure potential, variational capacity, 2-PI space, critical radius}
}

\begin{document}

\begin{abstract}
Let $(X,d,\omega)$ be an unbounded complete doubling metric measure space supporting a weak $(1,2)$-Poincar\'e inequality, and let $\pi$ be a positive Radon measure. For a ball $B=B(x,R)$, suppose that for some $\delta>0$ and $C_{\mathrm{gr}}>0$,
\[
 \pi(B(y,r))\le C_{\mathrm{gr}}
 \left(\frac rR\right)^\delta\frac{\omega(B(y,r))}{r^2},
 \qquad y\in B,\quad 0<r\le R.
\]
Let $\operatorname{cap}_{2,\omega}(K;\Omega)$ denote the relative variational $2$-capacity of $K$ in an open set $\Omega$. We prove that, for every fixed $\Lambda_0>2$, there is a constant $C>0$ such that
\[
 \pi(K)\le C\operatorname{cap}_{2,\omega}(K;\Lambda_0B),
 \qquad K\subset B\ \text{compact},
\]
with the same outer ball for all compact $K$. No doubling or lower-Ahlfors regularity is imposed on $\pi$. This fixed-domain capacitary domination yields representative-independent $L^2(d\pi)$ traces and, under a lower normalization $\pi(B)\ge c_0\omega(B)/R^2$ with $c_0>0$, the mixed $d\omega\,d\pi$ Poincar\'e estimate. Summation over bounded-overlap normalized covers gives two-sided global Fefferman--Phong inequalities and identifies the corresponding homogeneous energy spaces.

In $\mathbb R^d$ with a Muckenhoupt $A_2$ weight $w$, scale decay and controlled enlargement of $\pi$ generate the required critical-radius cover for every sufficiently large fixed threshold $A$; no additional lower reverse-doubling hypothesis on $w$ is needed. The resulting energy comparison gives equality of the measure-potential and critical-multiplier form domains, a smooth form core, a nonnegative self-adjoint realization, resolvent estimates, and local critical-multiplier bounds. The framework also covers reverse-H\"older function potentials, Carnot groups, and lower-dimensional singular measures, and provides a representative-aware capacity derivation of an earlier generalized Schr\"odinger energy estimate.
\end{abstract}

\maketitle

\section{Introduction and main results}

\subsection{Motivation}

Let $(X,d,\omega)$ be a metric measure space and let $\pi$ be a positive Radon measure. We study an energy-comparison problem for Schr\"odinger-type forms with possibly singular potentials. There are two natural ways to measure the size of the potential contribution. One is the measure-potential energy determined directly by $\pi$; the other is a multiplier energy determined by a positive function $m$, whose reciprocal represents a local length scale. The question is when these two descriptions are interchangeable, up to the ambient first-order Sobolev energy.

The reciprocal scale $m^{-1}$ should be thought of as the radius on which the potential has the same order as the first-order energy. In the abstract globalization argument this scale is encoded by a covering with ball radii comparable to $m^{-1}$, while in the weighted Euclidean application it is generated by a critical-radius construction based on normalized potential mass. Comparisons of this type belong to the Fefferman--Phong uncertainty-principle theory; see \cite{Fefferman1983}. Critical-radius methods for Schr\"odinger operators with reverse-H\"older or measure potentials were developed in influential work of Shen \cite{Shen1995,Shen1999} and subsequently adapted to degenerate and generalized settings; see, among others, \cite{KurataSugano,WuYan,BuiDoTrong2021,DoTruong2023}.

The singularity of $\pi$ is the main point that distinguishes the present problem from the usual function-potential setting. A positive Radon measure may be concentrated on an $\omega$-null set, so its potential energy cannot be determined from the $\omega$-almost-everywhere class of a Sobolev function alone. One must first identify a representative that is meaningful for $\pi$, and then control that representative by the ambient Sobolev geometry. This leads naturally to capacity.

The local mechanism of the paper is therefore
\begin{equation*}
 \begin{aligned}
 \text{improved decay of }\pi\text{ on smaller balls}
 &\Longrightarrow \text{codimensional content control}
 \\
 &\Longrightarrow \text{capacity control in one fixed outer ball}
 \\
 &\Longrightarrow \text{trace and mixed Poincar\'e estimates}.
 \end{aligned}
\end{equation*}
Once a local scale can be organized into a bounded-overlap covering adapted to that scale, these local estimates can be summed to obtain the global Fefferman--Phong comparison. The rest of the introduction makes this scheme precise and states the main results.

\subsection{Metric Sobolev setting and target inequalities}

We now fix the analytic framework needed for the precise statements. Throughout, $A\lesssim B$ means $A\le CB$ for a constant $C>0$ independent of the displayed variables; $A\gtrsim B$ and $A\simeq B$ have the corresponding meanings. Unsubscripted constants denoted by $C$ may change from line to line and depend only on the structural data explicitly indicated in context.

Let $(X,d,\omega)$ be an unbounded complete metric measure space. The measure $\omega$ is assumed Borel regular and doubling: there exists $C_D\ge1$ such that
\begin{equation*}
 0<\omega(B(x,2r))\le C_D\,\omega(B(x,r))<\infty,
 \qquad x\in X,\ r>0.
\end{equation*}
In particular, $\omega$ has full support. For a ball $B=B(x,r)$ we write $r_B:=r$ and, whenever $0<\omega(B)<\infty$,
\[
 \fint_B f\,d\omega:=\frac1{\omega(B)}\int_B f\,d\omega,
 \qquad
 u_{B,\omega}:=\fint_Bu\,d\omega.
\]

For $1\le p<\infty$, the $p$-modulus of a family $\Gamma$ of rectifiable curves is
\[
 \operatorname{Mod}_p(\Gamma):=\inf_\rho\int_X\rho^p\,d\omega,
\]
where the infimum is over nonnegative Borel functions $\rho$ satisfying $\int_\gamma\rho\,ds\ge1$ for every $\gamma\in\Gamma$. A nonnegative Borel function $g$ is a $p$-weak upper gradient of $u$ if
\[
 |u(\gamma(0))-u(\gamma(\ell_\gamma))|\le\int_\gamma g\,ds
\]
for all rectifiable curves outside a family of $p$-modulus zero. The Newtonian space $N^{1,2}(X)$ consists of the functions $u\in L^2(X,d\omega)$ admitting a $2$-weak upper gradient in $L^2(X,d\omega)$; $N^{1,2}_{\loc}(X)$ denotes the corresponding local space, and $g_u$ denotes the minimal $2$-weak upper gradient.

We assume that $X$ supports a weak $(1,2)$-Poincar\'e inequality: there are $C_P>0$ and $\lambda_P\ge1$ such that
\begin{equation*}
 \fint_B|u-u_{B,\omega}|\,d\omega
 \le C_Pr_B\left(\fint_{\lambda_PB}g^2\,d\omega\right)^{1/2}
\end{equation*}
for every ball $B$, every locally integrable $u$, and every $2$-weak upper gradient $g$ of $u$. A metric measure space satisfying the doubling and weak $(1,2)$-Poincar\'e conditions will be called a $2$-PI space.

For an arbitrary set $E\subset X$, define the global Sobolev $2$-capacity by
\[
 \operatorname{Cap}_{2,\omega}(E)
 :=\inf_{\substack{G\supset E\\G\ \mathrm{open}}}
 \inf\left\{\int_X(|v|^2+g_v^2)\,d\omega:
 v\in N^{1,2}(X),\ v\ge1\ \omega\text{-a.e. on }G\right\}.
\]
A statement holds \emph{quasi-everywhere} (q.e.) if it fails only on a set of $\operatorname{Cap}_{2,\omega}$-capacity zero. Under the complete $2$-PI assumptions, Newtonian functions admit $2$-quasicontinuous representatives; we denote such a representative of $u$ by $\widetilde u$. For an open set $\Omega\subset X$, $N_0^{1,2}(\Omega)$ denotes the Newtonian zero-boundary space, equivalently the functions whose quasicontinuous representatives vanish q.e. on $X\setminus\Omega$.

Keith--Zhong open-endedness and the metric Sobolev--Poincar\'e improvement imply the quadratic estimate
\begin{equation*}
 \int_B |u-u_{B,\omega}|^2\,d\omega
 \lesssim r^2\int_{\lambda B}g_u^2\,d\omega,
 \qquad B=B(x,r),
\end{equation*}
for a structural dilation $\lambda\ge1$.

For compact $K\subset\Omega$ with $\Omega$ open, we use the relative variational capacity
\begin{equation*}
 \capw(K;\Omega)
 :=\inf\left\{
 \int_\Omega g_v^2\,d\omega:
 v\in N_0^{1,2}(\Omega),\ \widetilde v\ge1\ \text{q.e. on }K
 \right\}.
\end{equation*}

Let now $\pi$ be a positive Radon measure and let $m:X\to(0,\infty)$ be a positive function representing the reciprocal of a local length scale associated with $\pi$. The two potential energies are
\[
 \int_X |\widetilde u|^2\,d\pi
 \qquad\text{and}\qquad
 \int_X m(x)^2|u(x)|^2\,d\omega(x),
\]
and our target is the two-sided comparison
\begin{align}
 \int_X m^2|u|^2\,d\omega
 &\lesssim \int_X g_u^2\,d\omega+\int_X|\widetilde u|^2\,d\pi,
 \label{eq:intro-fp-forward}\\
 \int_X|\widetilde u|^2\,d\pi
 &\lesssim \int_X g_u^2\,d\omega+\int_Xm^2|u|^2\,d\omega.
 \label{eq:intro-fp-reverse}
\end{align}
For a ball $B=B(x,r)$, the mixed oscillation that links the two energies is
\begin{equation}
 \mathcal O_{\omega,\pi}(u;B)
 :=\int_B\int_B |u(x)-\widetilde u(y)|^2\,d\omega(x)\,d\pi(y).
 \label{eq:intro-mixed-osc}
\end{equation}
Unlike the usual $d\omega(x)d\omega(y)$ oscillation, this quantity records the values of the quasicontinuous representative on sets that may be $\omega$-null but carry $\pi$-mass.

The local argument below proves fixed-domain capacitary control of $\pi$ on compact subsets of a ball. This makes the $L^2(d\pi)$ trace representative-independent and yields a mean-zero trace estimate. On balls satisfying the critical normalization
\[
 \pi(B)\simeq \frac{\omega(B)}{r^2},
\]
the trace estimate controls \eqref{eq:intro-mixed-osc} up to the ordinary Poincar\'e energy. The global inequalities \eqref{eq:intro-fp-forward}--\eqref{eq:intro-fp-reverse} then follow by summing the resulting local estimates over a bounded-overlap covering whose radii represent the scale $m^{-1}$.

\subsection{Content, capacity, and mixed oscillation}

Let $B_0=B(x_0,R)$ and let $\delta>0$. For compact $K\subset B_0$, define the restricted codimensional content
\begin{equation}
 \cH_{\delta,\omega}(K;B_0)
 :=\inf_{\{B_i\}}
 \sum_i\left(\frac{r_i}{R}\right)^\delta
 \frac{\omega(B_i)}{r_i^2},
 \label{eq:intro-content}
\end{equation}
where the infimum is taken over countable coverings $K\subset\bigcup_i B_i$ with centers in $B_0$ and radii $0<r_i\le R$. The factor $\omega(B_i)/r_i^2$ is the natural codimension-two scale for quadratic capacity, while $(r_i/R)^\delta$ supplies a positive gain at smaller radii. We say that $\omega$ is $Q$-Ahlfors regular on a specified range of scales if $c r^Q\le\omega(B(x,r))\le C r^Q$ there for fixed constants $0<c\le C<\infty$. Thus, if $\omega$ is $Q_{\mathrm{reg}}$-Ahlfors regular with $Q_{\mathrm{reg}}-2+\delta>0$, then \eqref{eq:intro-content} is comparable to $R^{-\delta}$ times the $(Q_{\mathrm{reg}}-2+\delta)$-dimensional Hausdorff content. Here $\mathcal H^s_\infty(E)$ denotes the $s$-dimensional Hausdorff content, namely the infimum of $\sum_i r_i^s$ over countable ball coverings $E\subset\bigcup_i B(x_i,r_i)$.

The point of the next theorem is that this content is controlled by relative $2$-capacity using one and the same outer ball. More precisely, the outer domain is $\Lambda_0B_0$ for every compact set $K\subset B_0$; it is not allowed to depend on $K$.

\begin{theorem}
\label{thm:intro-content-capacity}
Let $(X,d,\omega)$ be an unbounded complete $2$-PI space, let $\delta>0$, and fix $\Lambda_0>2$. Then there exists $C_{\delta,\Lambda_0}>0$ such that
\begin{equation*}
 \cH_{\delta,\omega}(K;B_0)
 \le C_{\delta,\Lambda_0}\,\capw(K;\Lambda_0B_0)
\end{equation*}
for every ball $B_0\subset X$ and every compact set $K\subset B_0$.
\end{theorem}

The content in \eqref{eq:intro-content} is a measure-adapted codimensional Hausdorff content, and the proof of \cref{thm:intro-content-capacity} uses a boxing argument. To place this result in context, the relevant comparisons with the closest capacity--Hausdorff estimates in metric spaces are as follows.
\begin{enumerate}[label=\textup{(\roman*)},leftmargin=2.2em]
\item Kinnunen--Korte--Shanmugalingam--Tuominen \cite[Theorem~3.6 and Remark~3.5]{KinnunenKorteShanTuominen} prove the two-sided equivalence between $1$-capacity and codimension-one Hausdorff content. For $p>1$, their Remark~3.5 records the elementary direction in which $p$-capacity is controlled by codimension-$p$ Hausdorff content, whereas the content-by-capacity direction used here is not supplied by that argument.
\item Costea \cite[Theorem~4.4 and Remark~4.5]{Costea2009} obtains relative content-by-capacity estimates for a fixed exponent $Q_C$ and $1<p\le Q_C$, under the global lower mass bound
\[
 \omega(B(x,t))\ge C^{-1}t^{Q_C},\qquad x\in X,\ t>0.
\]
For $p=2$ and a power gauge, this gives the familiar fixed-outer-ball estimate in the polynomial-growth setting.
\item In the present theorem the covering cost is the intrinsic quantity
\[
 \left(\frac{r_i}{R}\right)^\delta\frac{\omega(B_i)}{r_i^2},
\]
and no prescribed polynomial lower mass bound or Ahlfors regularity is imposed. The same preassigned outer ball $\Lambda_0B_0$ is used for every compact $K\subset B_0$.
\end{enumerate}
Related neighbourhood-capacity and density comparisons can be found in \cite{Lehrback2012,CantoIhnatsyevaLehrbackVahakangas2025}. The contribution needed below is precisely the third feature above: it converts one-sided ball growth of a possibly singular measure directly into uniform compact-set capacity domination. The later bounded-overlap summation is standard once this local estimate is available.

Suppose now that, for some $C_{\mathrm{gr}}>0$, $\pi$ satisfies, on $B_0=B(x_0,R)$,
\begin{equation}
 \pi(B(y,r))
 \le C_{\mathrm{gr}}
 \left(\frac rR\right)^\delta
 \frac{\omega(B(y,r))}{r^2},
 \qquad y\in B_0,\quad 0<r\le R.
 \label{eq:intro-growth}
\end{equation}
Then coverings immediately give $\pi(K)\lesssim\cH_{\delta,\omega}(K;B_0)$, and \cref{thm:intro-content-capacity} yields capacity control of every compact subset. The capacitary trace theorem proved below then gives the local $L^2(d\pi)$ estimate.

The next main result converts the one-sided growth condition into fixed-domain capacitary domination and the local trace estimate needed for the mixed Poincar\'e inequality.

\begin{theorem}
\label{thm:intro-growth-trace}
Assume \eqref{eq:intro-growth} and fix $\Lambda_0>2$ as above. Then
\begin{equation*}
 \pi(K)\le C\capw(K;\Lambda_0B_0)
\end{equation*}
for every compact $K\subset B_0$. Moreover, there is a structural dilation $\Lambda\ge\Lambda_0$ such that, for every $u\in N^{1,2}_{\loc}(X)$,
\begin{equation*}
 \int_{B_0}|\widetilde u-u_{B_0,\omega}|^2\,d\pi
 \le C\int_{\Lambda B_0}g_u^2\,d\omega.
\end{equation*}
If additionally, for some $c_0>0$, the lower critical normalization
\begin{equation*}
 \pi(B_0)\ge c_0\frac{\omega(B_0)}{R^2}
\end{equation*}
holds, then, for every such $u$,
\begin{equation*}
 \mathcal O_{\omega,\pi}(u;B_0)
 \le CR^2\pi(B_0)\int_{\Lambda B_0}g_u^2\,d\omega.
\end{equation*}
\end{theorem}

The local trace conclusion is also related to two-weighted Poincar\'e and embedding theorems. The maximal-function and truncation theory of \cite{BjornKalamajska2022,KinnunenKorteLehrbackVahakangas} overlaps with it under additional regularity of the two measures. In particular, when $Q_A>2$, $\omega$ is $Q_A$-Ahlfors regular, and $\pi$ satisfies the target-measure hypotheses of \cite[Theorem~4.1]{BjornKalamajska2022}, the growth condition \eqref{eq:intro-growth} verifies their balance condition for some $q>2$ whenever $\delta>0$. The computation is given in \cref{sec:content-capacity}. No two-sided/Ahlfors volume condition on $\omega$, and no doubling or lower-dimension condition on $\pi$, is assumed in the main content--capacity theorem.

\subsection{Normalized covers and globalization}

The local estimates above become global once $X$ is covered by balls on which the potential measure is normalized at the same scale as the first-order energy. More precisely, let
\[
 \mathscr B=\{B_j=B(x_j,r_j)\}_{j\ge1}
\]
be a countable family of balls that covers $X$ pointwise,
\begin{equation}
 X=\bigcup_{j\ge1}B_j,
 \label{eq:intro-covering}
\end{equation}
and assume that there are fixed constants $\Lambda\ge1$ and $N_\Lambda<\infty$ such that
\begin{equation*}
 \sum_j\one_{\Lambda B_j}\le N_\Lambda,
\end{equation*}
where $\one_E$ denotes the indicator of a set $E$.
On $B_j$, the radius $r_j$ represents the relevant local length scale. We therefore require the multiplier introduced above to satisfy
\begin{equation*}
 m(x)\simeq r_j^{-1},\qquad x\in B_j.
\end{equation*}
The potential is assumed to have the corresponding lower normalization and fixed-domain capacitary control,
\begin{equation}
 \pi(B_j)\gtrsim\frac{\omega(B_j)}{r_j^2},
 \qquad
 \pi(K)\lesssim\capw(K;\Lambda_0B_j)
 \label{eq:intro-cover-cap}
\end{equation}
uniformly for compact $K\subset B_j$. The second condition also gives the matching upper bound $\pi(B_j)\lesssim \omega(B_j)/r_j^2$. Thus, on each selected ball, the three quantities $r_j^{-2}$, $m^2$, and the normalized mass $\pi(B_j)/\omega(B_j)$ are comparable in the sense needed by the local estimates. We call a pair $(\mathscr B,m)$ satisfying the pointwise covering, bounded-overlap, multiplier-comparability, lower-normalization, and capacity conditions above a \emph{normalized capacitary cover}.

The globalization theorem simply adds the local inequalities over such a cover. Bounded overlap controls the gradient terms, while the pointwise covering property is needed for the potential term when $\pi$ is singular. We write $\LipLip(X)$ for the compactly supported Lipschitz functions on $X$.

\begin{theorem}
\label{thm:intro-global-fp}
If $(\mathscr B,m)$ is a normalized capacitary cover in the sense just described, then the inequalities \eqref{eq:intro-fp-forward}--\eqref{eq:intro-fp-reverse} hold for every $u\in N^{1,2}_{\loc}(X)$, with all terms interpreted in $[0,\infty]$. Consequently, the completions of $\LipLip(X)$ in the two energy norms
\begin{align*}
 \|u\|_{\Hpi}^2
 &:=\int_Xg_u^2\,d\omega+\int_X|\widetilde u|^2\,d\pi,\\
 \|u\|_{\Hm}^2
 &:=\int_Xg_u^2\,d\omega+\int_Xm^2|u|^2\,d\omega
\end{align*}
are canonically identified, and every element of the homogeneous completion has a canonical representative in $N^{1,2}_{\loc}(X)$.
\end{theorem}

The bounded-overlap summation itself is standard; the substantive issue is to construct balls for which the fixed-domain capacity condition in \eqref{eq:intro-cover-cap} holds uniformly. Notice also the distinction between the two covering requirements. The union in \eqref{eq:intro-covering} must cover every point, because an $\omega$-null omitted set may carry $\pi$-mass. By contrast, for a family of open balls in a space where $\omega$ has full support, pointwise and $\omega$-almost-everywhere bounded multiplicity are equivalent.

\subsection{Euclidean weighted verification}

We next explain how the abstract local scale arises in the weighted Euclidean Schr\"odinger setting. Let $d\omega=dw=w(x)\,dx$ on $\R^d$, $d\ge3$. For a measurable set $E$, write $w(E):=\int_E w(x)\,dx$. For $1<p<\infty$, the Muckenhoupt class $A_p(\R^d)$ consists of the weights $w>0$ a.e. such that
\[
 [w]_{A_p}:=\sup_B
 \left(\fint_B w\,dx\right)
 \left(\fint_B w^{-1/(p-1)}\,dx\right)^{p-1}<\infty,
\]
where the supremum is over Euclidean balls. The class $A_1$ is defined by $\fint_Bw\,dx\le C\,\operatorname*{ess\,inf}_B w$ uniformly in $B$. We assume only
\begin{equation}
 w\in A_2(\R^d).
 \label{eq:intro-A2-assumption}
\end{equation}
We write
\[
 L_w^2:=L^2(\R^d,dw),
 \qquad
 W_w^{1,2}:=\{u\in L_w^2:\nabla u\in L_w^2(\R^d;\C^d)\},
\]
and $W_{w,\loc}^{1,2}(\R^d)$ for the corresponding local weighted Sobolev space. The stronger polynomial reverse-doubling hypothesis with exponent $\beta>2$, imposed in the earlier generalized Schr\"odinger framework \cite{BuiDoTrong2021}, is not needed for the energy-level verification proved here.
Let $\pi\not\equiv0$ be a positive Radon measure. For $x\in\R^d$ and $r>0$, define the dimensionless quantity
\begin{equation*}
 \PhiPi(x,r):=\frac{r^2\pi(B(x,r))}{w(B(x,r))}.
\end{equation*}
Thus $\PhiPi(x,r)$ compares the mass of the potential measure on $B(x,r)$ with the natural first-order scale $w(B(x,r))/r^2$. Assume that for some $\delta>0$ and constants $C_0,C_M>0$,
\begin{equation}
 \begin{aligned}
 \PhiPi(x,r)
 &\le C_0\left(\frac rR\right)^\delta\PhiPi(x,R),
 &&0<r<R,\\
 \pi(B(x,2r))
 &\le C_M\left[\pi(B(x,r))+\frac{w(B(x,r))}{r^2}\right],
 &&r>0.
 \end{aligned}
 \label{eq:intro-enlargement}
\end{equation}
These are the two measure assumptions used in the generalized degenerate Schr\"odinger framework of \cite{BuiDoTrong2021}; that paper additionally assumes a reverse-doubling exponent $\beta>2$ for the ambient weight. Set
\[
 C_*:=4C_M,
 \qquad A_*:=2C_*=8C_M.
\]
Fix $A\ge A_*$. The associated critical radius is
\begin{equation*}
 \rho_A(x):=\sup\{r>0:\PhiPi(x,r)\le A\},
 \qquad m_A(x):=\rho_A(x)^{-1}.
\end{equation*}
In other words, $\rho_A(x)$ is the supremal admissible scale at which the normalized potential mass remains below the fixed threshold $A$, and $m_A$ is its reciprocal. The theorem below shows that this scale is finite, quantitatively stable under changes of the base point, and normalized exactly as required in the abstract theory.

\begin{theorem}
\label{thm:intro-euclidean}
Under \eqref{eq:intro-A2-assumption} and \eqref{eq:intro-enlargement}, for every $x\in\R^d$ one has $0<\rho_A(x)<\infty$ and
\begin{equation*}
 \pi(B(x,\rho_A(x)))
 \simeq\frac{w(B(x,\rho_A(x)))}{\rho_A(x)^2}.
\end{equation*}
The critical radius satisfies quantitative global comparison estimates and generates a bounded-overlap covering by critical balls. On every selected ball $B_j=B(x_j,r_j)$,
\begin{equation*}
 \pi(B(y,s))
 \lesssim\left(\frac{s}{r_j}\right)^\delta
 \frac{w(B(y,s))}{s^2},
 \qquad y\in B_j,\quad 0<s\le r_j.
\end{equation*}
Therefore, for every $u\in W_{w,\loc}^{1,2}(\R^d)$, with all global terms interpreted in $[0,\infty]$,
\begin{align*}
 \int_{\R^d}m_A^2|u|^2\,dw
 &\lesssim\int_{\R^d}|\nabla u|^2\,dw+
 \int_{\R^d}|\widetilde u|^2\,d\pi,\\
 \int_{\R^d}|\widetilde u|^2\,d\pi
 &\lesssim\int_{\R^d}|\nabla u|^2\,dw+
 \int_{\R^d}m_A^2|u|^2\,dw.
\end{align*}
Different fixed admissible thresholds generate uniformly comparable critical radii.
\end{theorem}

\subsection{Further verification classes and sharpness}

The local theory also applies beyond critical-radius measure potentials. For $q>1$, we write $V\in RH_q(d\omega)$ if
\[
 \left(\fint_B V^q\,d\omega\right)^{1/q}
 \le C\fint_B V\,d\omega
\]
uniformly over balls $B$. For $d\pi=V\,d\omega$ with $V$ in such a reverse-H\"older class, H\"older and Sobolev--Poincar\'e give the trace estimate directly; see \cref{prop:reverse-holder-direct}. The metric result yields the corresponding horizontal-gradient estimate on Carnot groups and applies to lower-dimensional Ahlfors-regular measures, including genuinely singular examples; see \cref{cor:Carnot-mixed,prop:singular-local-example,prop:singular-global-example}.

The structural hypothesis is the quadratic PI property rather than membership in $A_2$. Here a weight is called $2$-admissible when its weighted measure is doubling and supports a weak $(1,2)$-Poincar\'e inequality. There are $2$-admissible weights outside $A_2$, whereas bare $A_p$ membership with $p>2$ need not imply a quadratic Poincar\'e inequality. Finally, the gain $\delta>0$ is sharp in the full complete $2$-PI class: at $\delta=0$ a line segment in $\R^3$ has positive critical one-dimensional content but zero relative $2$-capacity; see \cref{prop:endpoint-failure}.

\subsection{Operator application and relation to earlier work}

The Euclidean estimates also identify the natural form domain of a degenerate Schr\"odinger operator with a measure potential. Let $\mathbf A(x)$ be a real symmetric matrix satisfying
\begin{equation*}
 \alpha_{\mathbf A} w(x)|\xi|^2
 \le \mathbf A(x)\xi\cdot\xi
 \le \beta_{\mathbf A} w(x)|\xi|^2,
 \qquad 0<\alpha_{\mathbf A}\le\beta_{\mathbf A}<\infty.
\end{equation*}
The measure potential is incorporated through the quadratic form
\begin{equation*}
 \mathfrak a_\pi(u,v)
 =\int_{\R^d}\mathbf A\nabla u\cdot\overline{\nabla v}\,dx
 +\int_{\R^d}\widetilde u\,\overline{\widetilde v}\,d\pi.
\end{equation*}
Theorem~\ref{thm:intro-euclidean} shows that the two natural descriptions of its energy domain agree:
\begin{equation*}
 \{u\in W_w^{1,2}:\widetilde u\in L^2(d\pi)\}
 =
 \{u\in W_w^{1,2}:m_Au\in L_w^2\},
\end{equation*}
with equivalent norms. This yields a smooth form core and a nonnegative self-adjoint realization; the resolvent, homogeneous Lax--Milgram, Caccioppoli, and local multiplier consequences are stated in \cref{prop:resolvent-energy,thm:homogeneous-lax-milgram,prop:Caccioppoli,thm:local-critical-multiplier}.

\Cref{sec:relation-earlier} applies the mixed estimate to a published generalized Schr\"odinger energy argument and closes with a finite-scale comparison with the $A_1$ mechanism of \cite{DoTruong2023}. These are applications of the capacity theory, not part of the abstract proof.

\subsection{Novelty and technical difficulties}

The main contribution is the local pair \cref{thm:intro-content-capacity,thm:intro-growth-trace}: an intrinsic content-by-capacity estimate on complete $2$-PI spaces and its conversion of one-sided ball growth into fixed-domain compact-set domination. The content itself is a standard type of metric Hausdorff content; the point is that no prescribed polynomial lower-volume or Ahlfors-regularity assumption is needed and the outer domain is fixed uniformly for all compact subsets. The representative-independent trace and mixed Poincar\'e estimates follow from this setwise control. The globalization in \cref{sec:normalized-covers} is a standard bounded-overlap summation once these local estimates are available.

There are several points at which the singular target measure makes the argument less routine. Newtonian functions are defined only up to capacity-zero sets, so the potential integral requires proof that $\pi$ does not charge those sets. The growth assumption is stated only for balls, whereas the trace theorem requires uniform control of arbitrary compact subsets relative to one fixed outer domain. The mixed oscillation involves $d\omega(x)d\pi(y)$ rather than the usual $d\omega(x)d\omega(y)$ product measure. In the global step the balls must cover every point, not merely almost every point with respect to $\omega$, because the omitted set could carry singular $\pi$-mass. Finally, ordinary weighted Sobolev density does not by itself give convergence in the measure-potential form norm; the reverse Fefferman--Phong estimate provides the additional control needed for the form theory.

\subsection{Organization}

\Cref{sec:preliminaries} collects the quadratic Poincar\'e estimate, reverse volume growth, relative capacity, and the trace-capacity principle. \Cref{sec:content-capacity} proves the fixed-domain content--capacity estimate, converts measure growth into capacity domination, and explains the endpoint obstruction. \Cref{sec:normalized-covers} formulates the normalized-cover hypotheses and proves the global energy comparison announced in \cref{thm:intro-global-fp}. \Cref{sec:verification} verifies these hypotheses in the weighted Euclidean setting and treats reverse-H\"older, Carnot, and singular examples. \Cref{sec:operators} develops the form and operator consequences. Finally, \cref{sec:relation-earlier} relates the mixed-measure argument to \cite{BuiDoTrong2021} and records a finite-scale comparison with \cite{DoTruong2023}.

For clarity, the introductory statements are summaries of results proved later rather than separate assertions with independent proofs. Specifically, \cref{thm:intro-content-capacity} is \cref{thm:content-capacity}; \cref{thm:intro-growth-trace} follows from \cref{thm:growth-capacity,cor:growth-mixed}; \cref{thm:intro-global-fp} is proved by \cref{thm:global-fp,cor:energy-equivalence,prop:abstract-energy-realization}; and \cref{thm:intro-euclidean} is assembled from the critical-radius results of \cref{sec:verification}, culminating in \cref{thm:euclidean-verification}.
\section{PI spaces, variational capacity, and trace inequalities}
\label{sec:preliminaries}

Throughout this section, $(X,d,\omega)$ is an unbounded complete metric measure space. The purpose of the section is to isolate the three analytic ingredients used later: a quadratic Poincar\'e estimate for the ambient measure, a fixed-domain variational capacity, and a trace principle that transfers capacitary control to the possibly singular measure $\pi$. The unboundedness assumption is needed for the large-scale reverse-volume argument and for the large comparison ball in the subsequent content--capacity proof; bounded spaces would require a separate large-scale normalization and are not treated here. We assume that $\omega$ is a Borel regular measure with full support and is doubling: there exists $C_D\ge1$ such that
\begin{equation}
 0<\omega(B(x,2r))\le C_D\,\omega(B(x,r))<\infty
 \label{eq:doubling}
\end{equation}
for all $x\in X$ and $r>0$.

For a ball $B=B(x,r)$, we write $r_B=r$ and
\[
 u_{B,\omega}:=\avgint_Bu\,d\omega.
\]
A nonnegative Borel function $g$ is an upper gradient of $u$ if
\[
 |u(\gamma(0))-u(\gamma(\ell_\gamma))|
 \le\int_\gamma g\,ds
\]
for every rectifiable curve $\gamma$ for which both endpoint values are defined. For $1\le p<\infty$ and a family $\Gamma$ of rectifiable curves, its $p$-modulus is
\[
 \operatorname{Mod}_p(\Gamma)
 :=\inf_\rho\int_X\rho^p\,d\omega,
\]
where the infimum is over all nonnegative Borel functions $\rho$ satisfying $\int_\gamma\rho\,ds\ge1$ for every $\gamma\in\Gamma$. A function $g$ is a $2$-weak upper gradient of $u$ if the upper-gradient inequality is allowed to fail only on a curve family of $2$-modulus zero. The Newtonian space $N^{1,2}(X)$ consists of $L^2(X,d\omega)$ functions admitting a $2$-weak upper gradient in $L^2(X,d\omega)$. Every such function has a minimal $2$-weak upper gradient, denoted $g_u$. We use the local and zero-boundary spaces $N^{1,2}_{\loc}(X)$ and $N_0^{1,2}(\Omega)$ in the standard sense; see \cite{Shanmugalingam2000,BjornBjornBook,HKST}.

We use the following standard formulation.

\begin{definition}
\label{def:2pi}
A metric measure space $(X,d,\omega)$ is a \emph{$2$-PI space} if $\omega$ is doubling and there exist $C_P>0$ and $\lambda_P\ge1$ such that
\begin{equation*}
 \avgint_B |u-u_{B,\omega}|\,d\omega
 \le C_Pr_B
 \left(\avgint_{\lambda_PB}g^2\,d\omega\right)^{1/2}
\end{equation*}
for every ball $B$, every locally integrable $u$, and every $2$-weak upper gradient $g$ of $u$.
\end{definition}

A complete doubling space is proper (closed bounded sets are compact), and a complete PI space is quantitatively quasiconvex (there is $C_q\ge1$ such that any two points can be joined by a rectifiable curve of length at most $C_q$ times their distance); see \cite{HKST,Korte2007}. We will use these facts and Lipschitz cutoffs without further comment.

\subsection{Quadratic oscillation and reverse doubling}

The proofs below require an $L^2$ oscillation estimate. It is automatic under the preceding $2$-PI hypothesis, but not by a direct application of Jensen's inequality.

\begin{lemma}
\label{lem:quadratic-pi}
Let $(X,d,\omega)$ be a complete $2$-PI space. Then there exist $C_2>0$ and $\lambda_2\ge1$, depending only on the doubling and Poincar\'e data, such that
\begin{equation}
 \left(\avgint_B|u-u_{B,\omega}|^2\,d\omega\right)^{1/2}
 \le C_2r_B
 \left(\avgint_{\lambda_2B}g_u^2\,d\omega\right)^{1/2}
 \label{eq:quadratic-pi-normalized}
\end{equation}
for every ball $B$ and every $u\in N^{1,2}_{\loc}(X)$. Consequently, there exists $C_2'>0$, depending only on the same data, such that
\begin{equation}
 \int_B|u-u_{B,\omega}|^2\,d\omega
 \le C_2' r_B^2\int_{\lambda_2B}g_u^2\,d\omega.
 \label{eq:quadratic-pi}
\end{equation}
\end{lemma}

\begin{proof}
Fix a finite doubling dimension $Q_D>2$ and a constant $c_{Q_D}>0$ such that
\begin{equation*}
 \frac{\omega(B(y,r))}{\omega(B(x,R))}
 \ge c_{Q_D}\left(\frac rR\right)^{Q_D}
\end{equation*}
whenever $y\in B(x,R)$ and $0<r\le R$. Such an exponent follows by iterating \eqref{eq:doubling}; enlarging an admissible exponent preserves the estimate.

By the Keith--Zhong theorem \cite{KeithZhong}, there exists $\varepsilon>0$ such that $X$ supports a weak $(1,q)$-Poincar\'e inequality for every $2-\varepsilon<q<2$. On each fixed bounded dilation of $B$, a curve family of zero $2$-modulus also has zero $q$-modulus for $q<2$: apply H\"older's inequality to admissible densities on that finite-measure dilation. Hence the minimal $2$-weak upper gradient $g_u$ may be used in the self-improved inequality below. Choose $q$ in this interval so that
\begin{equation*}
 q>\frac{2Q_D}{Q_D+2}.
\end{equation*}
Then the Sobolev exponent $q^*=Q_Dq/(Q_D-q)$ satisfies $q^*>2$. The metric Sobolev--Poincar\'e theorem \cite{HajlaszKoskela,HKST} gives
\begin{equation*}
 \left(\avgint_B|u-u_{B,\omega}|^{q^*}\,d\omega\right)^{1/q^*}
 \le Cr_B
 \left(\avgint_{\lambda_2B}g_u^q\,d\omega\right)^{1/q}.
\end{equation*}
Since $q^*>2$ and $q<2$, monotonicity of normalized integral means gives
\[
 \left(\avgint_B|u-u_{B,\omega}|^2\,d\omega\right)^{1/2}
 \le
 \left(\avgint_B|u-u_{B,\omega}|^{q^*}\,d\omega\right)^{1/q^*}
\]
and
\[
 \left(\avgint_{\lambda_2B}g_u^q\,d\omega\right)^{1/q}
 \le
 \left(\avgint_{\lambda_2B}g_u^2\,d\omega\right)^{1/2}.
\]
Substitution gives \eqref{eq:quadratic-pi-normalized}. Multiplying by $\omega(B)$ and using doubling to compare the fixed dilates gives \eqref{eq:quadratic-pi}.
\end{proof}

All later arguments use \eqref{eq:quadratic-pi}. Thus one could replace the standard $2$-PI hypothesis by doubling together with a direct quadratic Poincar\'e inequality. \Cref{lem:quadratic-pi} shows that no additional assumption is needed under the standard convention.
The content--capacity argument in the next section also needs a large ball on which a compactly supported test function has small average. In the present unbounded PI setting, the needed reverse volume growth is automatic.

\begin{lemma}
\label{lem:reverse-doubling}
Let $(X,d,\omega)$ be an unbounded complete $2$-PI space. There exist $\sigma>0$ and $C_{\mathrm{RD}}\ge1$ such that
\begin{equation}
 \frac{\omega(B(x,r))}{\omega(B(x,R))}
 \le C_{\mathrm{RD}}\left(\frac rR\right)^\sigma
 \label{eq:metric-reverse-doubling}
\end{equation}
for all $x\in X$ and $0<r\le R$.
\end{lemma}

\begin{proof}
A complete PI space is quantitatively quasiconvex. Fix $x\in X$ and $r>0$. Since $X$ is unbounded, choose $z$ with $d(x,z)>2r$. Quasiconvexity provides a rectifiable curve $\gamma$ joining $x$ to $z$, and continuity of $\gamma$ yields a point $y\in\gamma$ with $d(x,y)=2r$. Then
\[
 B(y,r/2)\cap B(x,r)=\varnothing,
 \qquad B(y,r/2)\subset B(x,4r).
\]
Moreover $B(x,r)\subset B(y,3r)$, and doubling gives
\[
 \omega(B(y,r/2))\ge C_D^{-3}\omega(B(x,r)).
\]
Consequently
\begin{equation*}
 \omega(B(x,4r))
 \ge \vartheta\,\omega(B(x,r)),
 \qquad \vartheta:=1+C_D^{-3}>1.
\end{equation*}
Iteration gives $\omega(B(x,4^kr))\ge\vartheta^k\omega(B(x,r))$. If $4^kr\le R<4^{k+1}r$ and $\sigma=\log_4\vartheta$, then
\[
 \omega(B(x,R))
 \ge \vartheta^{-1}\left(\frac Rr\right)^\sigma\omega(B(x,r)),
\]
which is equivalent to \eqref{eq:metric-reverse-doubling}.
\end{proof}

We will also use a zero-boundary Sobolev inequality. We state the dilation explicitly because later arguments apply the estimate to zero extensions on a larger fixed ball.

\begin{lemma}
\label{lem:zero-boundary-sobolev}
There exist structural constants $\kappa>1$ and $\lambda_S\ge1$ with the following property. For every fixed pair $\Theta>1$ and $\Theta'\ge4\lambda_S\Theta$, there is a constant $C_{\Theta,\Theta'}$ such that
\begin{equation}
 \left(\avgint_{\Theta B}|v|^{2\kappa}\,d\omega\right)^{1/(2\kappa)}
 \le C_{\Theta,\Theta'}r_B
 \left(\avgint_{\Theta'B}g_v^2\,d\omega\right)^{1/2}
 \label{eq:zero-boundary-sobolev}
\end{equation}
for every ball $B$ and every $v\in N_0^{1,2}(\Theta B)$. Here $v$ is extended by zero outside $\Theta B$, and $g_v$ denotes the minimal $2$-weak upper gradient of that zero extension; by locality, $g_v=0$ almost everywhere outside $\Theta B$.
\end{lemma}

\begin{proof}
The Sobolev--Poincar\'e improvement of the $2$-PI inequality gives structural $\kappa>1$ and $\lambda_S\ge1$ such that, on every ball $D$,
\[
 \left(\avgint_D|f-f_{D,\omega}|^{2\kappa}\,d\omega\right)^{1/(2\kappa)}
 \le Cr_D
 \left(\avgint_{\lambda_SD}g_f^2\,d\omega\right)^{1/2}.
\]
Apply this with $D=4\Theta B$ to the zero extension of $v$. As in the proof of \cref{lem:reverse-doubling}, with
$\vartheta:=1+C_D^{-3}>1$, one has
\[
 \omega(4\Theta B)\ge\vartheta\,\omega(\Theta B).
\]
so the zero set
\[
 Z:=4\Theta B\setminus\Theta B
\]
has measure at least $(1-\vartheta^{-1})\omega(4\Theta B)$. Since $v=0$ on $Z$, for $a:=v_{4\Theta B,\omega}$ we have
\[
 \omega(Z)|a|^{2\kappa}
 \le \int_{4\Theta B}|v-a|^{2\kappa}\,d\omega.
\]
Thus the average term $|a|$ is controlled by the same oscillation norm, with a structural constant. The triangle inequality then yields
\[
 \left(\avgint_{4\Theta B}|v|^{2\kappa}\,d\omega\right)^{1/(2\kappa)}
 \le Cr_B
 \left(\avgint_{4\lambda_S\Theta B}g_v^2\,d\omega\right)^{1/2}.
\]
Restricting the left-hand side to $\Theta B$, enlarging the gradient ball to the fixed ball $\Theta'B$, and using doubling to compare the normalized averages prove \eqref{eq:zero-boundary-sobolev}. Notice that the gradient integral does not acquire new terms under this enlargement because the minimal weak upper gradient of the zero extension vanishes outside $\Theta B$. This is the usual Friedrichs mechanism for Newtonian functions with a positive-measure zero set; see also \cite{HajlaszKoskela,HKST}.
\end{proof}

\subsection{Relative capacity and quasicontinuous representatives}

Let $\Omega\subset X$ be open. The relative capacity below measures the least quadratic upper-gradient energy needed to raise a zero-boundary Sobolev function to level one on a set. As usual, ``quasi-everywhere'' means outside a set of Sobolev $2$-capacity zero. For $E\subset\Omega$, define
\begin{equation*}
 \capw(E;\Omega)
 :=\inf_v\int_\Omega g_v^2\,d\omega,
\end{equation*}
where the infimum is over $v\in N_0^{1,2}(\Omega)$ such that $\widetilde v\ge1$ quasi-everywhere on $E$. Under the standing assumptions, Newtonian functions have quasicontinuous representatives, and relative capacity is Choquet on open reference sets; see \cite{BjornBjornBook,BjornBjornCapacity}.

We also use the global Sobolev capacity fixed in the opening conventions. Under the standing complete doubling PI assumptions, it has the equivalent quasicontinuous formulation
\begin{equation*}
 \operatorname{Cap}_{2,\omega}(E)
 =\inf_v\int_X\bigl(|v|^2+g_v^2\bigr)\,d\omega,
\end{equation*}
where the infimum is over $v\in N^{1,2}(X)$ with $\widetilde v\ge1$ quasi-everywhere on $E$. We write $K\Subset\Omega$ when the closure of $K$ is a compact subset of the open set $\Omega$. If $\operatorname{Cap}_{2,\omega}(E)=0$, then for every compact $K\subset E$ and every open $\Omega$ with $K\Subset\Omega$, cutoff localization gives
\begin{equation}
 \capw(K;\Omega)=0.
 \label{eq:global-to-relative-capacity}
\end{equation}
Indeed, multiply global admissible functions of arbitrarily small $N^{1,2}$ norm by a fixed Lipschitz cutoff that equals one near $K$ and is compactly supported in $\Omega$. We will use \eqref{eq:global-to-relative-capacity} when passing from local capacitary domination to global representative-independence.

The boxing argument in the next section needs continuous admissible functions that equal one on a neighbourhood of the compact set. We write $\LipLip(\Omega)$ for the Lipschitz functions on $\Omega$ with support compactly contained in $\Omega$; their zero extensions are globally Lipschitz on $X$. The following approximation provides exactly that form without changing the capacity.

\begin{lemma}
\label{lem:lipschitz-capacity-approx}
Let $K\Subset\Omega$ be compact. For every $\varepsilon>0$, there exists $v\in\LipLip(\Omega)$ such that
\begin{equation*}
 0\le v\le1,
 \qquad v=1\ \text{in an open neighbourhood of }K,
\end{equation*}
and
\begin{equation*}
 \int_\Omega g_v^2\,d\omega
 \le \capw(K;\Omega)+\varepsilon.
\end{equation*}
\end{lemma}

\begin{proof}
The condenser-capacity approximation theorem for compact sets allows the variational capacity to be computed using compactly supported Lipschitz functions. More precisely, \cite[Proposition~7.8]{BjornBjorn2024Condenser} gives, under hypotheses weaker than those imposed here,
\[
 \capw(K;\Omega)
 =\inf\left\{\int_\Omega g_h^2\,d\omega:
 h\in\LipLip(\Omega),\ h\ge1\ \text{on }K\right\}.
\]
Hence, after replacing $\varepsilon$ by a smaller number, there is $h\in\LipLip(\Omega)$ such that $h\ge1$ on $K$ and
\[
 \int_\Omega g_h^2\,d\omega\le\capw(K;\Omega)+\varepsilon/4.
\]
After truncation we may assume $0\le h\le1$, so $h=1$ on $K$. Fix $\theta\in(0,1)$ sufficiently close to $1$ and set
\[
 v:=\min\{1,h/\theta\}.
\]
Since $h$ is continuous and equals $1$ on the compact set $K$, the set $\{h>\theta\}$ is an open neighbourhood of $K$, and $v=1$ there. The upper-gradient chain rule gives $g_v\le\theta^{-1}g_h$ almost everywhere. Choosing $\theta$ sufficiently close to $1$ yields
\[
 \int_\Omega g_v^2\,d\omega
 \le\theta^{-2}\bigl(\capw(K;\Omega)+\varepsilon/4\bigr)
 \le\capw(K;\Omega)+\varepsilon.
\]
This is the required estimate.
\end{proof}

\subsection{A capacitary trace theorem}

We next record the trace principle that converts capacity domination into an $L^2(d\pi)$ estimate. This step has two roles. It bounds the singular potential term by the ambient first-order energy, and it shows that the target measure $\pi$ does not see the exceptional sets on which a quasicontinuous representative is undetermined.

The proof uses the following dyadic strong capacitary estimate.

\begin{lemma}
\label{lem:strong-capacitary}
Let $\Omega\subset X$ be open and $v\in N_0^{1,2}(\Omega)$. Then
\begin{equation}
 \sum_{k\in\mathbb Z}2^{2k}
 \capw\bigl(\{x\in\Omega:|\widetilde v(x)|>2^k\};\Omega\bigr)
 \le C\int_\Omega g_v^2\,d\omega.
 \label{eq:strong-capacitary}
\end{equation}
\end{lemma}

\begin{proof}
For a real number $t$, write $t_+:=\max\{t,0\}$. For $k\in\mathbb Z$, define
\[
 v_k(x):=\min\left\{1,2^{-k}(|v(x)|-2^k)_+\right\}.
\]
Then $v_k\in N_0^{1,2}(\Omega)$ and $v_k=1$ quasi-everywhere on $\{|\widetilde v|>2^{k+1}\}$. The upper-gradient chain rule gives
\[
 g_{v_k}\le 2^{-k}g_v\one_{\{2^k<|v|<2^{k+1}\}}
 \quad\text{a.e.}
\]
Hence
\[
 2^{2k}\capw(\{|\widetilde v|>2^{k+1}\};\Omega)
 \le\int_{\{2^k<|v|<2^{k+1}\}}g_v^2\,d\omega.
\]
Summing over $k$ proves \eqref{eq:strong-capacitary}, after an index shift.
\end{proof}

The next theorem converts compact-set capacity domination into a trace inequality and also guarantees that quasicontinuous representatives are determined by the target measure.

\begin{theorem}
\label{thm:trace-capacity}
Let $\Omega\subset X$ be open and let $\pi$ be a positive Radon measure on $\Omega$. Assume that there exists $C_{\mathrm{cap}}>0$ such that
\begin{equation}
 \pi(K)\le C_{\mathrm{cap}}\,\capw(K;\Omega)
 \label{eq:capacity-domination-general}
\end{equation}
for every compact $K\subset\Omega$. Then $\pi$ vanishes on Borel sets of relative capacity zero and
\begin{equation}
 \int_\Omega|\widetilde v|^2\,d\pi
 \le C C_{\mathrm{cap}}\int_\Omega g_v^2\,d\omega
 \label{eq:trace-capacity}
\end{equation}
for every $v\in N_0^{1,2}(\Omega)$.
\end{theorem}

\begin{proof}
Since $\pi$ is Radon and relative capacity is monotone, \eqref{eq:capacity-domination-general} extends from compact sets to Borel sets by inner regularity. In particular, $\pi$ does not charge capacity-zero sets, and the integral of $\widetilde v$ is representative-independent.

Let $E_k=\{|\widetilde v|>2^k\}$. By the dyadic layer-cake estimate,
\[
 \int_\Omega|\widetilde v|^2\,d\pi
 \le C\sum_{k\in\mathbb Z}2^{2k}\pi(E_k).
\]
Applying \eqref{eq:capacity-domination-general} and \cref{lem:strong-capacitary} proves \eqref{eq:trace-capacity}.
\end{proof}

\subsection{Mean-zero traces and a general mixed estimate}

We now pass from a zero-boundary trace estimate to a mean-zero estimate on a ball. To keep every later capacity relative to one common outer domain, fix first a structural dilation
\begin{equation}
 \Lambda_0>2,
 \label{eq:Lambda0-choice}
\end{equation}
and then choose
\begin{equation}
 1<\tau<\Lambda_0,
 \qquad \Lambda\ge\max\{\Lambda_0,\lambda_2\tau\}.
 \label{eq:dilation-package}
\end{equation}
Here $\lambda_2$ is the dilation in the quadratic Poincar\'e estimate \eqref{eq:quadratic-pi}. The choice $\Lambda_0>2$ is sufficient for all later uses: it contains the support of the scale-$r$ cutoff used on $B$, and it also leaves room for the fine-scale balls used in the content--capacity argument, where only the inclusion $(21/20)B\subset\Lambda_0B$ is needed. The role of $\tau$ is to support a cutoff that equals one on $B$, while $\Lambda$ contains both the capacity domain and the Poincar\'e dilation of $\tau B$. The constants are fixed once and for all in this order.

The capacity condition on a ball gives the following mean-zero trace estimate.

\begin{proposition}
\label{prop:mean-zero-trace}
Let $B=B(x,r)$ and suppose there exists $C_{\mathrm{cap},B}>0$ such that
\begin{equation}
 \pi(K)\le C_{\mathrm{cap},B}\,\capw(K;\Lambda_0B)
 \label{eq:local-cap-domination}
\end{equation}
for every compact $K\subset B$. Then
\begin{equation*}
 \int_B|\widetilde u-u_{B,\omega}|^2\,d\pi
 \le C C_{\mathrm{cap},B}\int_{\Lambda B}g_u^2\,d\omega
\end{equation*}
for every $u\in N^{1,2}_{\loc}(X)$.
\end{proposition}

\begin{proof}
Choose a Lipschitz cutoff $\eta$ with $\eta=1$ on $B$, $\supp\eta\subset\tau B$, and
$g_\eta\le C_\tau/r$, where $C_\tau\simeq(\tau-1)^{-1}$ and $\tau$ is fixed structurally. Set
\[
 v:=\eta(u-u_{B,\omega}).
\]
Then $v\in N_0^{1,2}(\Lambda_0B)$ and
\[
 g_v\le \eta g_u+|u-u_{B,\omega}|g_\eta.
\]
Let $\pi_B:=\pi\lfloor_B$, where $\pi\lfloor_B$ denotes the restriction of $\pi$ to $B$. The restricted measure $\pi_B$ is Radon on $\Lambda_0B$. For every compact $F\subset\Lambda_0B$, inner regularity and \eqref{eq:local-cap-domination} give
\[
 \begin{aligned}
 \pi_B(F)
 &=\pi(F\cap B)
 =\sup_{\substack{K\subset F\cap B\\K\text{ compact}}}\pi(K)\\
 &\le C_{\mathrm{cap},B}\sup_{\substack{K\subset F\cap B\\K\text{ compact}}}
 \capw(K;\Lambda_0B)
 \le C_{\mathrm{cap},B}\capw(F;\Lambda_0B).
 \end{aligned}
\]
Thus \cref{thm:trace-capacity} applies to $\pi_B$. Moreover, the quasicontinuous representative of $v$ equals $\eta(\widetilde u-u_{B,\omega})$ quasi-everywhere. Since $\pi_B$ does not charge sets of relative capacity zero, this identity holds $\pi_B$-almost everywhere. Since $v=u-u_{B,\omega}$ on $B$,
\begin{align*}
 \int_B|\widetilde u-u_{B,\omega}|^2\,d\pi
 &=\int_{\Lambda_0B}|\widetilde v|^2\,d\pi_B\\
 &\le CC_{\mathrm{cap},B}\int_{\tau B}g_u^2\,d\omega
 +\frac{CC_{\mathrm{cap},B}}{r^2}\int_{\tau B}|u-u_{B,\omega}|^2\,d\omega.
\end{align*}
To control the last term, first write
\[
 \begin{aligned}
 \int_{\tau B}|u-u_{B,\omega}|^2\,d\omega
 &\le 2\int_{\tau B}|u-u_{\tau B,\omega}|^2\,d\omega\\
 &\quad+2\omega(\tau B)|u_{\tau B,\omega}-u_{B,\omega}|^2.
 \end{aligned}
\]
Moreover,
\[
 |u_{\tau B,\omega}-u_{B,\omega}|^2
 \le \frac{1}{\omega(B)}\int_B|u-u_{\tau B,\omega}|^2\,d\omega.
\]
Since $\tau$ is fixed, doubling bounds $\omega(\tau B)/\omega(B)$ by a structural constant. Hence
\[
 \int_{\tau B}|u-u_{B,\omega}|^2\,d\omega
 \le C\int_{\tau B}|u-u_{\tau B,\omega}|^2\,d\omega.
\]
Applying \eqref{eq:quadratic-pi} to $\tau B$ and using $\lambda_2\tau B\subset\Lambda B$ gives
\[
 \frac1{r^2}\int_{\tau B}|u-u_{B,\omega}|^2\,d\omega
 \le C\int_{\Lambda B}g_u^2\,d\omega.
\]
Substitution in the preceding trace bound proves the proposition.
\end{proof}

The mean-zero trace estimate now combines with the ordinary $d\omega$ Poincar\'e estimate to control the mixed $d\omega(x)d\pi(y)$ oscillation. The two terms below have different origins: $r^2\pi(B)$ comes from ambient oscillation in the $x$ variable, whereas $C_{\mathrm{cap},B}\omega(B)$ comes from the capacitary trace in the $y$ variable.

\begin{theorem}
\label{thm:general-mixed-estimate}
Under the hypotheses of \cref{prop:mean-zero-trace}, for every $u\in N^{1,2}_{\loc}(X)$,
\begin{equation*}
 \begin{aligned}
 &\int_B\int_B|u(x)-\widetilde u(y)|^2\,d\omega(x)\,d\pi(y)\\
 &\qquad\le C\bigl[r^2\pi(B)+C_{\mathrm{cap},B}\omega(B)\bigr]
 \int_{\Lambda B}g_u^2\,d\omega.
 \end{aligned}
\end{equation*}
In particular, if for some $C_0>0$,
\begin{equation}
 C_{\mathrm{cap},B}\omega(B)\le C_0r^2\pi(B),
 \label{eq:normalized-cap-coefficient}
\end{equation}
then
\begin{equation*}
 \int_B\int_B|u(x)-\widetilde u(y)|^2\,d\omega(x)\,d\pi(y)
 \le Cr^2\pi(B)\int_{\Lambda B}g_u^2\,d\omega.
\end{equation*}
\end{theorem}

\begin{proof}
The pointwise estimate
\[
 |u(x)-\widetilde u(y)|^2
 \le2|u(x)-u_{B,\omega}|^2+2|\widetilde u(y)-u_{B,\omega}|^2
\]
gives
\begin{align*}
 \mathcal O_{\omega,\pi}(u;B)
 &\le2\pi(B)\int_B|u-u_{B,\omega}|^2\,d\omega
 +2\omega(B)\int_B|\widetilde u-u_{B,\omega}|^2\,d\pi.
\end{align*}
By the quadratic Poincar\'e inequality,
\[
 \pi(B)\int_B|u-u_{B,\omega}|^2\,d\omega
 \le C r^2\pi(B)\int_{\Lambda B}g_u^2\,d\omega,
\]
after enlarging the structural Poincar\'e dilation to the fixed ball $\Lambda B$. On the other hand, \cref{prop:mean-zero-trace} gives
\[
 \omega(B)\int_B|\widetilde u-u_{B,\omega}|^2\,d\pi
 \le C C_{\mathrm{cap},B}\omega(B)
 \int_{\Lambda B}g_u^2\,d\omega.
\]
Adding these two estimates proves the first assertion. If \eqref{eq:normalized-cap-coefficient} holds, then
$C_{\mathrm{cap},B}\omega(B)\le C_0r^2\pi(B)$, so the two coefficients are both bounded by a constant multiple of $r^2\pi(B)$, which gives the normalized conclusion.
\end{proof}
\section{Codimensional content and variational capacity}
\label{sec:content-capacity}

Throughout this section, $(X,d,\omega)$ is an unbounded complete $2$-PI space. Recall the fixed structural dilations $\Lambda_0$ and $\Lambda$ from \eqref{eq:Lambda0-choice}--\eqref{eq:dilation-package}, and, for a ball $B_0=B(x_0,R)$, write
\begin{equation*}
 \Omega_0:=\Lambda_0B_0.
\end{equation*}

The purpose of this section is to turn an upper growth estimate for the potential measure into the compact-set capacity domination needed by the trace theorem. The intermediate object is a restricted codimensional content. Its covering cost is chosen to match exactly the scale appearing in the growth condition for $\pi$.

\subsection{Restricted codimensional content}

Let $\delta>0$ and $K\subset B_0$. Define
\begin{equation}
 \cH_{\delta,\omega}(K;B_0)
 :=\inf_{\mathcal B}
 \sum_i\left(\frac{r_i}{R}\right)^\delta
 \frac{\omega(B_i)}{r_i^2},
 \label{eq:content-definition}
\end{equation}
where the infimum is taken over countable families $\mathcal B=\{B_i=B(x_i,r_i)\}$ such that
\begin{equation*}
 K\subset\bigcup_iB_i,
 \qquad x_i\in B_0,
 \qquad 0<r_i\le R.
\end{equation*}
Thus each covering ball is charged the codimension-two cost $\omega(B_i)/r_i^2$, multiplied by the positive small-scale gain $(r_i/R)^\delta$. The restriction on centers and radii keeps the content tied to the reference ball $B_0$.

In the Ahlfors-regular case, the definition reduces to a familiar Hausdorff-content quantity. If for some $Q_{\mathrm{reg}}>0$ one has $\omega(B(x,r))\simeq r^{Q_{\mathrm{reg}}}$ and $Q_{\mathrm{reg}}-2+\delta>0$, then
\[
 \cH_{\delta,\omega}(K;B_0)
 \simeq R^{-\delta}\mathcal H^{Q_{\mathrm{reg}}-2+\delta}_\infty(K),
\]
where $\mathcal H^s_\infty$ is the $s$-dimensional Hausdorff content, defined as the infimum of $\sum_i r_i^s$ over countable ball coverings. Hence $\delta>0$ moves the covering exponent strictly above the capacity-critical dimension $Q_{\mathrm{reg}}-2$.

Still in this Ahlfors-regular setting, restricting the covering balls to have centers in $B_0$ and radii at most $R$ changes the usual Hausdorff content only by a multiplicative constant. Indeed, a covering ball $B(z,r)$ with $r\le R$ that meets $K$ can be recentered at a point $z'\in K\cap B(z,r)\subset B_0$. If $2r\le R$, replace it by $B(z',2r)$; Ahlfors regularity changes the cost only by a fixed factor. If $2r>R$, then $r>R/2$, so the single ball $B_0$ covers all of $K$ and has Hausdorff-content cost comparable with that of $B(z,r)$. The same replacement by $B_0$ applies to any covering ball with radius larger than $R$.

\subsection{The content--capacity estimate}

The next theorem is the central local result. The important point is that the capacity of every compact $K\subset B_0$ is computed relative to the same outer domain $\Omega_0=\Lambda_0B_0$.

\begin{theorem}[Content--capacity embedding]
\label{thm:content-capacity}
For every $\delta>0$ and the fixed choice $\Lambda_0>2$ from \eqref{eq:Lambda0-choice}, there is $C_{\delta,\Lambda_0}>0$, depending only on $\delta$, $\Lambda_0$, and the $2$-PI data, such that
\begin{equation}
 \cH_{\delta,\omega}(K;B_0)
 \le C_{\delta,\Lambda_0}\capw(K;\Omega_0)
 \label{eq:content-capacity-main}
\end{equation}
for every ball $B_0=B(x_0,R)$ and every compact $K\subset B_0$.
\end{theorem}

\begin{proof}
Fix $\varepsilon>0$. By \cref{lem:lipschitz-capacity-approx}, choose $u\in\LipLip(\Omega_0)$ satisfying
\begin{equation*}
 0\le u\le1,
 \qquad u=1\ \text{in a neighbourhood of }K
\end{equation*}
and
\begin{equation}
 \int_{\Omega_0}g_u^2\,d\omega
 \le\capw(K;\Omega_0)+\varepsilon.
 \label{eq:boxing-near-minimizer}
\end{equation}
Extend $u$ by zero to $X$. Since $u\in\LipLip(\Omega_0)$, this zero extension is globally Lipschitz. Let $g$ denote its minimal $2$-weak upper gradient; by locality, $g=0$ almost everywhere outside $\Omega_0$.

The proof compares the value $u(x)=1$ at a point $x\in K$ with the small average of $u$ on a sufficiently large ball. A dyadic telescoping argument then produces one of two alternatives. Either a scale comparable with $R$ already carries enough energy, or every point of $K$ selects a smaller ball whose codimensional covering cost is controlled by the energy on a slightly enlarged ball. The first alternative is handled by the single cover $\{B_0\}$; the second by a disjoint $5$-covering selection.

\smallskip
\noindent\emph{Step 1: a uniformly small large-scale average.}
Fix $x\in K$. Since $x\in B_0$,
\[
 \Omega_0\subset B(x,(\Lambda_0+1)R).
\]
By reverse doubling, once $L>2(\Lambda_0+1)$ is chosen sufficiently large in terms of the structural data,
\begin{equation}
 \frac{\omega(\Omega_0)}{\omega(B(x,LR))}
 \le \frac{\omega(B(x,(\Lambda_0+1)R))}{\omega(B(x,LR))}
 \le\frac14
 \label{eq:large-average-ratio}
\end{equation}
for every $x\in B_0$. Since $0\le u\le1$ and the zero extension of $u$ vanishes outside $\Omega_0$,
\begin{equation*}
 u_{B(x,LR),\omega}\le\frac14.
\end{equation*}

\smallskip
\noindent\emph{Step 2: telescoping the dyadic averages.}
Set
\begin{equation*}
 r_k:=2^{-k}LR,
 \qquad D_k:=B(x,r_k),
 \qquad k\ge0.
\end{equation*}
Because $u$ is continuous and equals $1$ in a neighbourhood of $K$,
\[
 \lim_{k\to\infty}u_{D_k,\omega}=u(x)=1.
\]
Together with Step~1 this gives
\begin{equation*}
 \frac34
 \le\sum_{k=0}^\infty
 |u_{D_{k+1},\omega}-u_{D_k,\omega}|.
\end{equation*}
For clarity, one may estimate each difference by
\begin{align*}
 |u_{D_{k+1},\omega}-u_{D_k,\omega}|
 &\le \frac1{\omega(D_{k+1})}
       \int_{D_{k+1}}|u-u_{D_k,\omega}|\,d\omega\\
 &\le C\left(\frac1{\omega(D_k)}
       \int_{D_k}|u-u_{D_k,\omega}|^2\,d\omega\right)^{1/2},
\end{align*}
where doubling was used in the second line. The quadratic Poincar\'e inequality and another doubling comparison therefore give
\begin{equation}
 |u_{D_{k+1},\omega}-u_{D_k,\omega}|
 \le Cr_k
 \left(\frac1{\omega(\lambda_2D_k)}
 \int_{\lambda_2D_k}g^2\,d\omega\right)^{1/2}.
 \label{eq:average-difference}
\end{equation}
Denote the right-hand side by $a_k(x)$, with the structural constant chosen so that
\begin{equation*}
 \frac34\le\sum_{k=0}^\infty a_k(x).
\end{equation*}

\smallskip
\noindent\emph{Step 3: split the scales into coarse and fine parts.}
Let $k_0$ be the smallest integer such that
\begin{equation*}
 r_{k_0}\le\frac{R}{20\lambda_2}.
\end{equation*}
Because $r_k=2^{-k}LR$ and $L$ is fixed, the number of indices $k<k_0$ is bounded by a structural constant.

\smallskip
\noindent\emph{Case 1: some point has a large coarse contribution.}
Suppose that
\begin{equation*}
 \sum_{k<k_0}a_k(x)\ge\frac38
\end{equation*}
for at least one $x\in K$. Since there are only structurally many coarse indices, $a_k(x)\ge c>0$ for some $k<k_0$. By \eqref{eq:average-difference},
\begin{equation}
 \int_{\lambda_2D_k}g^2\,d\omega
 \ge c\frac{\omega(\lambda_2D_k)}{r_k^2}.
 \label{eq:coarse-energy}
\end{equation}
For $k<k_0$ one has $r_k>R/(20\lambda_2)$ and $r_k\le LR$. Since $x\in B_0$,
\[
 B_0\subset B(x,2R)\subset B(x,40\lambda_2r_k).
\]
Repeated doubling therefore yields $\omega(\lambda_2D_k)\gtrsim\omega(B_0)$, while $r_k\lesssim R$. Since $g=0$ outside $\Omega_0$, \eqref{eq:coarse-energy} gives
\begin{equation*}
 \int_{\Omega_0}g^2\,d\omega
 \gtrsim\frac{\omega(B_0)}{R^2}.
\end{equation*}
The single ball $B_0$ is an admissible cover of $K$, and hence
\begin{equation*}
 \cH_{\delta,\omega}(K;B_0)
 \le\frac{\omega(B_0)}{R^2}
 \lesssim\int_{\Omega_0}g^2\,d\omega.
\end{equation*}
Combining this estimate with \eqref{eq:boxing-near-minimizer} and then letting $\varepsilon\downarrow0$ proves \eqref{eq:content-capacity-main} in Case~1.

\smallskip
\noindent\emph{Case 2: every point has a large fine contribution.}
Assume now that Case~1 fails for every $x\in K$. Then
\begin{equation*}
 \sum_{k\ge k_0}a_k(x)\ge\frac38,
 \qquad x\in K.
\end{equation*}
The geometric series
\[
 S_\delta:=\sum_{k\ge k_0}\left(\frac{r_k}{R}\right)^{\delta/2}
\]
is finite. Therefore, for each $x\in K$, at least one index $k=k(x)\ge k_0$ satisfies
\begin{equation}
 a_k(x)
 \ge c_\delta\left(\frac{r_k}{R}\right)^{\delta/2},
 \qquad c_\delta:=\frac{3}{8S_\delta}.
 \label{eq:selected-scale}
\end{equation}
Let $r_x=r_{k(x)}$ and $B_x=B(x,r_x)$. Combining \eqref{eq:average-difference} and \eqref{eq:selected-scale}, and using doubling to compare $\omega(B_x)$ and $\omega(\lambda_2B_x)$, gives
\begin{equation}
 \int_{\lambda_2B_x}g^2\,d\omega
 \ge c_\delta
 \left(\frac{r_x}{R}\right)^\delta
 \frac{\omega(B_x)}{r_x^2}.
 \label{eq:fine-energy-density}
\end{equation}
Moreover $r_x\le R/(20\lambda_2)$, so
\begin{equation*}
 \lambda_2B_x
 \subset B\left(x_0,R+\lambda_2r_x\right)
 \subset \frac{21}{20}B_0
 \subset\Omega_0.
\end{equation*}
The last inclusion follows from the fixed choice $\Lambda_0>2$ in \eqref{eq:Lambda0-choice}.

Apply the $5$-covering lemma to the family $\{\lambda_2B_x:x\in K\}$. There is a pairwise disjoint countable subfamily $\{\lambda_2B_i\}$ such that
\begin{equation*}
 K\subset\bigcup_i5\lambda_2B_i.
\end{equation*}
Each $5\lambda_2B_i$ is centered in $B_0$ and has radius at most
$5\lambda_2R/(20\lambda_2)=R/4$, hence this is an admissible cover in \eqref{eq:content-definition}. By doubling and \eqref{eq:fine-energy-density},
\begin{align*}
 \cH_{\delta,\omega}(K;B_0)
 &\le\sum_i
 \left(\frac{5\lambda_2r_i}{R}\right)^\delta
 \frac{\omega(5\lambda_2B_i)}{(5\lambda_2r_i)^2}\\
 &\le C_\delta\sum_i\int_{\lambda_2B_i}g^2\,d\omega
 \le C_\delta\int_{\Omega_0}g^2\,d\omega.
\end{align*}
The last inequality uses both pairwise disjointness and $\lambda_2B_i\subset\Omega_0$.

Combining the two alternatives with \eqref{eq:boxing-near-minimizer},
\[
 \cH_{\delta,\omega}(K;B_0)
 \le C_\delta\bigl(\capw(K;\Omega_0)+\varepsilon\bigr).
\]
Letting $\varepsilon\downarrow0$ proves \eqref{eq:content-capacity-main}.
\end{proof}

The proof of \cref{thm:content-capacity} in isolation only uses $\Lambda_0>21/20$ in the fine-scale inclusion. Throughout the paper we keep the stronger fixed choice $\Lambda_0>2$ from \eqref{eq:Lambda0-choice}, because the trace and covering arguments use cutoffs supported in $2B$.

Only a positive reverse-doubling exponent is used to choose the large scale in \eqref{eq:large-average-ratio}; no condition such as $\sigma>2$ is needed for \cref{thm:content-capacity}. This is why the theorem applies to low-dimensional and non-Ahlfors-regular $2$-PI spaces. Stronger lower-volume exponents enter only in particular critical-radius constructions later in the paper.

\subsection{Measure growth implies capacity domination}

We now impose the local upper growth condition that is naturally matched to \eqref{eq:content-definition}. Let $B_0=B(x_0,R)$ and $\delta>0$. We say that a positive Radon measure $\pi$ satisfies the $\delta$-subcritical growth condition on $B_0$ if there exists $C_{\mathrm{gr}}>0$ such that
\begin{equation}
 \pi(B(y,r))
 \le C_{\mathrm{gr}}
 \left(\frac rR\right)^\delta
 \frac{\omega(B(y,r))}{r^2}
 \label{eq:subcritical-growth}
\end{equation}
for every $y\in B_0$ and $0<r\le R$.

The terminology ``subcritical'' refers to the positive factor $(r/R)^\delta$: at smaller scales the potential mass decays faster than the codimension-two quantity $\omega(B(y,r))/r^2$.

For comparison only, suppose that $Q_A>2$, the ambient measure is $Q_A$-Ahlfors regular on the relevant scales, and that the target measure $\pi$ satisfies the local doubling and dimension hypotheses of \cite[Theorem~4.1]{BjornKalamajska2022}. Then, for $2<q<\infty$, $y\in B_0$, $0<\rho\le R$, and the fixed dilation $\lambda$ in that theorem, \eqref{eq:subcritical-growth} and Ahlfors regularity give
\[
 \frac{\rho\,\pi(B(y,\rho))^{1/q}}
 {\omega(B(y,\lambda\rho))^{1/2}}
 \lesssim
 R^{-\delta/q}\,
 \rho^{\,1-Q_A/2+(Q_A-2+\delta)/q}.
\]
Hence the corresponding local balance quantity is finite whenever
\[
 2<q\le \frac{2(Q_A-2+\delta)}{Q_A-2},
\]
and such a $q$ exists exactly when $\delta>0$. Thus their theorem recovers a local trace inequality under these additional target-measure and two-sided Ahlfors hypotheses.

The next theorem gives a different output directly: it requires neither an Ahlfors condition on $\omega$ nor doubling or a lower-dimension condition on $\pi$, and it controls every compact subset by capacity relative to the same outer ball. In particular, it yields the capacity absolute continuity needed to interpret quasicontinuous representatives $\pi$-almost everywhere. These are the features used later in the form theory.

\begin{theorem}
\label{thm:growth-capacity}
If \eqref{eq:subcritical-growth} holds, then
\begin{equation*}
 \pi(K)
 \le C\capw(K;\Lambda_0B_0)
\end{equation*}
for every compact $K\subset B_0$. The constant depends only on $C_{\mathrm{gr}}$, $\delta$, and the structural data.
\end{theorem}

\begin{proof}
Let $K\subset\bigcup_iB_i$ be any admissible cover in \eqref{eq:content-definition}. By definition, every $B_i=B(x_i,r_i)$ has $x_i\in B_0$ and $0<r_i\le R$, exactly the range on which \eqref{eq:subcritical-growth} is assumed. Countable subadditivity and that growth bound give
\[
 \pi(K)
 \le \sum_i\pi(B_i)
 \le C_{\mathrm{gr}}
 \sum_i\left(\frac{r_i}{R}\right)^\delta
 \frac{\omega(B_i)}{r_i^2}.
\]
Taking the infimum over all admissible covers yields
$\pi(K)\le C_{\mathrm{gr}}\cH_{\delta,\omega}(K;B_0)$, and \cref{thm:content-capacity} finishes the proof.
\end{proof}

On a ball where the potential mass is comparable with the codimension-two ambient scale, the capacity estimate immediately gives the mean-zero trace and mixed Poincar\'e inequalities.

\begin{corollary}
\label{cor:growth-mixed}
Assume \eqref{eq:subcritical-growth} and, for some $c_0>0$, the lower critical normalization
\begin{equation}
 \pi(B_0)\ge c_0\frac{\omega(B_0)}{R^2}.
 \label{eq:normalized-ball-lower}
\end{equation}
Then, for every $u\in N^{1,2}_{\loc}(X)$,
\begin{equation*}
 \int_{B_0}|\widetilde u-u_{B_0,\omega}|^2\,d\pi
 \le C\int_{\Lambda B_0}g_u^2\,d\omega
\end{equation*}
and
\begin{equation}
 \begin{aligned}
 &\int_{B_0}\int_{B_0}|u(x)-\widetilde u(y)|^2\,d\omega(x)\,d\pi(y)\\
 &\qquad\le CR^2\pi(B_0)\int_{\Lambda B_0}g_u^2\,d\omega.
 \end{aligned}
 \label{eq:growth-mixed-cor}
\end{equation}
\end{corollary}

\begin{proof}
By \cref{thm:growth-capacity}, there is a constant $C_{\mathrm{cap},0}$, depending only on
$C_{\mathrm{gr}}$, $\delta$, and the $2$-PI data, such that
\[
 \pi(K)\le C_{\mathrm{cap},0}\capw(K;\Lambda_0B_0)
\]
for every compact $K\subset B_0$. Thus \cref{prop:mean-zero-trace}, with
$C_{\mathrm{cap},B}=C_{\mathrm{cap},0}$, gives
\[
 \int_{B_0}|\widetilde u-u_{B_0,\omega}|^2\,d\pi
 \le CC_{\mathrm{cap},0}\int_{\Lambda B_0}g_u^2\,d\omega.
\]
For the mixed oscillation, \cref{thm:general-mixed-estimate} gives
\[
 \mathcal O_{\omega,\pi}(u;B_0)
 \le C\bigl(R^2\pi(B_0)+C_{\mathrm{cap},0}\omega(B_0)\bigr)
 \int_{\Lambda B_0}g_u^2\,d\omega.
\]
The lower normalization \eqref{eq:normalized-ball-lower} implies
$C_{\mathrm{cap},0}\omega(B_0)\le (C_{\mathrm{cap},0}/c_0)R^2\pi(B_0)$, which yields
\eqref{eq:growth-mixed-cor}. Finally, the matching upper normalization follows
directly from \eqref{eq:subcritical-growth} by taking $y=x_0$ and $r=R$.
\end{proof}

\subsection{Verification from scale decay}

Define the dimensionless potential strength at center $x$ and radius $r$ by
\begin{equation*}
 \PhiPi(x,r):=\frac{r^2\pi(B(x,r))}{\omega(B(x,r))}.
\end{equation*}
Then the growth hypothesis above can be checked from a scale-decay property of $\PhiPi$. Indeed, let $B_0=B(x_0,R)$ and suppose that, for some $C_{\mathrm{sc}}>0$ and $M_0>0$,
\[
 \PhiPi(y,r)
 \le C_{\mathrm{sc}}
 \left(\frac rR\right)^\delta
 \PhiPi(y,R),
 \qquad y\in B_0,\quad0<r\le R,
\]
and
\[
 \sup_{y\in B_0}\PhiPi(y,R)\le M_0.
\]
Then
\[
 \PhiPi(y,r)\le C_{\mathrm{sc}}M_0\left(\frac rR\right)^\delta,
\]
which is exactly \eqref{eq:subcritical-growth} after multiplying by $\omega(B(y,r))/r^2$. Thus one may take $C_{\mathrm{gr}}=C_{\mathrm{sc}}M_0$.

\subsection{Endpoint failure}

The positive codimensional gain cannot in general be removed. The next example shows that at the critical value $\delta=0$, positive content need not force positive relative $2$-capacity even in Euclidean space.

\begin{proposition}
\label{prop:endpoint-failure}
There exist an unbounded complete $2$-PI space, a ball $B_0$, and a compact set $K\subset B_0$ such that
\begin{equation*}
 \cH_{0,\omega}(K;B_0)>0,
 \qquad
 \capw(K;\Lambda_0B_0)=0.
\end{equation*}
\end{proposition}

\begin{proof}
Take $X=\R^3$, $d\omega=dx$, set $B_0=B(0,2)$, and let $K=[-1,1]\times\{0\}\times\{0\}\subset B_0$. Here $|E|$ denotes the Lebesgue measure of a measurable set $E$. Since $|B_i|/r_i^2\simeq r_i$, projecting any admissible covering of $K$ onto the first coordinate gives $\sum_i r_i\gtrsim1$. Consequently,
\[
 \cH_{0,\omega}(K;B_0)\gtrsim1.
\]

For a set $F\subset\R^3$, write $\dist(x,F):=\inf_{y\in F}|x-y|$. Fix $r_0>0$ small enough that the closed $r_0$-neighbourhood of $K$ is contained in $\Lambda_0B_0$. For $0<\varepsilon<r_0/2$, define a cutoff which equals $1$ on $\{\dist(x,K)\le\varepsilon\}$, equals $0$ on $\{\dist(x,K)\ge r_0\}$, and on the intermediate region is given by
\[
 u_\varepsilon(x)
 =\frac{\log(r_0/\dist(x,K))}{\log(r_0/\varepsilon)}.
\]
After an arbitrarily small Lipschitz smoothing at the two transition levels, this is an admissible relative-capacity test function. The level surface at distance $s$ from the segment has area $\lesssim s$ for $0<s<r_0$; equivalently, the tubular-neighbourhood volume derivative is $\lesssim s$. The coarea formula therefore gives
\begin{align*}
 \int_{\R^3}|\nabla u_\varepsilon|^2\,dx
 &\lesssim \frac1{\log^2(r_0/\varepsilon)}
 \int_\varepsilon^{r_0}\frac1{s}\,ds\\
 &\lesssim\frac1{\log(r_0/\varepsilon)}\longrightarrow0.
\end{align*}
Hence $\capw(K;\Lambda_0B_0)=0$.
\end{proof}

The endpoint example shows why the factor $(r/R)^\delta$ is needed. At the critical codimension, logarithmic cutoffs can have arbitrarily small energy; any positive power gain excludes this degeneracy.

\section{Normalized capacitary covers and globalization}
\label{sec:normalized-covers}

Throughout this section, $(X,d,\omega)$ is an unbounded complete $2$-PI space, $\pi$ is a positive Radon measure, and the fixed dilations satisfy \eqref{eq:dilation-package} and \eqref{eq:Lambda0-choice}. All balls are open and carry their designated centers and radii. Functions may be real- or complex-valued; absolute values are understood accordingly.

The purpose of the section is to pass from the local trace estimate of \cref{sec:preliminaries} to a global comparison of the two potential energies. The covering balls $B_j=B(x_j,r_j)$ encode a local length scale $r_j$. On each such ball, the multiplier $m$ is comparable with the reciprocal scale $r_j^{-1}$, while the potential measure has the critical size $\omega(B_j)/r_j^2$ and is controlled on compact subsets by relative capacity. Bounded overlap then allows the local inequalities to be summed.

\subsection{Normalized capacitary covers}

We collect these requirements in one definition. For a set $E\subset X$, $\one_E$ denotes its indicator function. Conditions \textup{(C1)}--\textup{(C2)} describe the geometry of the cover, \textup{(C3)} identifies $m$ with the reciprocal local scale, and \textup{(C4)} fixes the lower size of the potential mass at that scale. Condition \textup{(C5)} controls the $\pi$-mass of every compact subset by relative capacity; this is what makes the singular trace term independent of the chosen quasicontinuous representative.

\begin{definition}
\label{def:normalized-capacitary-cover}
A pair $(\mathscr B,m)$, where
\begin{equation*}
 \mathscr B=\{B_j=B(x_j,r_j)\}_{j\ge1}
\end{equation*}
is a countable family of open balls and $m:X\to(0,\infty)$ is Borel, is a \emph{normalized capacitary cover} for $(\omega,\pi)$ if the following conditions hold uniformly in $j$:
\begin{enumerate}[label=\textup{(C\arabic*)},leftmargin=2.4em]
\item the balls cover $X$ pointwise,
\begin{equation*}
 X=\bigcup_{j\ge1}B_j.
\end{equation*}
\item for the fixed final dilation $\Lambda$, there exists $N_\Lambda<\infty$ such that
\begin{equation*}
 \sum_j\one_{\Lambda B_j}(x)\le N_\Lambda,
 \qquad x\in X;
\end{equation*}
\item there is $C_m\ge1$ such that
\begin{equation}
 C_m^{-1}r_j^{-1}\le m(x)\le C_mr_j^{-1},
 \qquad x\in B_j;
 \label{eq:cover-m-comparison}
\end{equation}
\item there is $c_\pi>0$ such that
\begin{equation}
 \pi(B_j)\ge c_\pi\frac{\omega(B_j)}{r_j^2};
 \label{eq:cover-normalization}
\end{equation}
\item there is $C_{\mathrm{cap}}>0$ such that
\begin{equation*}
 \pi(K)\le C_{\mathrm{cap}}\capw(K;\Lambda_0B_j)
\end{equation*}
for every compact $K\subset B_j$.
\end{enumerate}
\end{definition}

The covering and overlap assumptions play different roles for singular target measures. We impose \textup{(C1)} pointwise, rather than merely $\omega$-almost everywhere, because an $\omega$-null uncovered set may still carry $\pi$-mass. If one used an almost-everywhere version instead, the exceptional set would have to be null for both measures. By contrast, for the open balls occurring in \textup{(C2)}, pointwise and $\omega$-almost-everywhere bounded multiplicity are equivalent because $\omega$ has full support: if the multiplicity exceeded $N_\Lambda$ at one point, the same would hold on a nonempty open set and hence on a set of positive $\omega$-measure.

Before proving any analytic estimate, we record two elementary consequences of the cover. First, the scale function $m$ is locally bounded above and below, and the cover is locally finite. Second, the capacitary condition ensures that $\pi$ ignores the exceptional sets on which quasicontinuous representatives may differ.

\begin{lemma}
\label{lem:cover-local-bounds}
Assume \textup{(C1)}--\textup{(C3)}. For every compact $K\subset X$ there exist constants
\[
 0<c_K\le C_K<\infty
\]
such that $c_K\le m\le C_K$ on $K$. Moreover, only finitely many balls $B_j$ meet $K$.
\end{lemma}

\begin{proof}
For each $x\in K$, choose $j(x)$ with $x\in B_{j(x)}$. Since $B_{j(x)}$ is open, \eqref{eq:cover-m-comparison} gives two-sided bounds for $m$ on a neighborhood of $x$. A finite subcover of $K$ yields the first assertion.

If $B_j\cap K\ne\varnothing$, choose $y\in B_j\cap K$. The compact-set bounds for $m(y)$ and \eqref{eq:cover-m-comparison} imply that all such radii $r_j$ lie in one fixed interval $[a_K,b_K]\subset(0,\infty)$. Their centers lie in the bounded neighborhood $\{x:d(x,K)<b_K\}$. This neighborhood has compact closure because a complete doubling space is proper. Cover it by finitely many balls of radius $a_K/4$. If the center $x_j$ belongs to one such net ball centered at $z$, then $d(z,x_j)<a_K/4\le r_j$, so $z\in B_j\subset\Lambda B_j$. The overlap bound \textup{(C2)} therefore permits at most $N_\Lambda$ such indices for each net ball. Hence only finitely many $B_j$ meet $K$.
\end{proof}

The local capacity condition also determines quasicontinuous representatives globally.

\begin{lemma}
\label{lem:global-representative-independence}
Assume \textup{(C1)} and \textup{(C5)}. If $E\subset X$ is Borel and has global Sobolev $2$-capacity zero, then $\pi(E)=0$. Consequently, for every $u\in N^{1,2}_{\loc}(X)$, the quantity
\[
 \int_X|\widetilde u|^2\,d\pi\in[0,\infty]
\]
is independent of the chosen quasicontinuous representative.
\end{lemma}

\begin{proof}
Let $K\subset E\cap B_j$ be compact. By \eqref{eq:global-to-relative-capacity},
\[
 \capw(K;\Lambda_0B_j)=0.
\]
Hence \textup{(C5)} implies $\pi(K)=0$. Inner regularity gives $\pi(E\cap B_j)=0$, and \textup{(C1)} together with countable subadditivity yields $\pi(E)=0$. For the representative statement, two quasicontinuous representatives of $u\in N^{1,2}_{\loc}(X)$ agree outside a set of global Sobolev $2$-capacity zero. For every compact subset of this exceptional set contained in $\Lambda_0B_j$, \eqref{eq:global-to-relative-capacity} gives zero relative capacity in $\Lambda_0B_j$. Applying \textup{(C5)} on $B_j$, then using inner regularity and the pointwise cover, shows that the two representatives agree $\pi$-almost everywhere on $X$.
\end{proof}

The lower normalization in \textup{(C4)} is an explicit hypothesis. The matching upper bound is automatic from \textup{(C5)}, because a cutoff of scale $r_j$ gives a test function for every compact subset of $B_j$.

\begin{lemma}
\label{lem:cover-upper-normalization}
Assume \textup{(C5)}. Then
\begin{equation}
 \pi(B_j)\le C_D C_{\mathrm{cap}}\frac{\omega(B_j)}{r_j^2}
 \label{eq:cover-upper-normalization}
\end{equation}
for every $j$. Thus \textup{(C4)} and \textup{(C5)} imply the two-sided normalization
\[
 \pi(B_j)\simeq\frac{\omega(B_j)}{r_j^2}.
\]
\end{lemma}

\begin{proof}
Let
\[
 \eta_j(x):=\left(1-\frac{\dist(x,B_j)}{r_j}\right)_+.
\]
Then $\eta_j=1$ on $B_j$, $\supp\eta_j\subset\overline{2B_j}\Subset\Lambda_0B_j$, and $g_{\eta_j}\le r_j^{-1}\one_{2B_j}$. Hence, for every compact $K\subset B_j$,
\[
 \capw(K;\Lambda_0B_j)
 \le\int_{2B_j}g_{\eta_j}^2\,d\omega
 \le\frac{\omega(2B_j)}{r_j^2}.
\]
Apply \textup{(C5)} and take the supremum over compact $K\subset B_j$ using inner regularity of $\pi$.
\end{proof}

The outer dilation in the capacity condition can be enlarged without changing the argument. The structural value of $\Lambda_0$ need not be explicit. Since relative capacity decreases as the outer domain increases, verifying \textup{(C5)} with any fixed $\Lambda'\ge\Lambda_0$ is sufficient. Thus direct applications may work with a convenient larger dilation.

\subsection{Local estimates}

At this point the abstract local theory of \cref{sec:preliminaries} can be applied ball by ball. Condition \textup{(C5)} supplies the trace estimate, while \textup{(C4)} supplies the lower normalization needed to convert that trace estimate into the mixed Poincar\'e bound.

\begin{proposition}
\label{prop:cover-local-estimates}
Let $(\mathscr B,m)$ be a normalized capacitary cover and let $u\in N^{1,2}_{\loc}(X)$. Then, uniformly in $j$,
\begin{align}
 \int_{B_j}|\widetilde u-u_{B_j,\omega}|^2\,d\pi
 &\le C\int_{\Lambda B_j}g_u^2\,d\omega,
 \label{eq:cover-local-trace}\\
 \int_{B_j}\int_{B_j}|u(x)-\widetilde u(y)|^2\,d\omega(x)\,d\pi(y)
 &\le Cr_j^2\pi(B_j)\int_{\Lambda B_j}g_u^2\,d\omega.
 \notag
\end{align}
\end{proposition}

\begin{proof}
Condition \textup{(C5)} is precisely the hypothesis of
\cref{prop:mean-zero-trace} on the ball $B_j$, with
$C_{\mathrm{cap},B_j}=C_{\mathrm{cap}}$. Hence
\[
 \int_{B_j}|\widetilde u-u_{B_j,\omega}|^2\,d\pi
 \le C C_{\mathrm{cap}}\int_{\Lambda B_j}g_u^2\,d\omega,
\]
which is the first estimate after absorbing the fixed constant
$C_{\mathrm{cap}}$ into $C$.

For the mixed estimate, \cref{thm:general-mixed-estimate} gives the coefficient
\[
 r_j^2\pi(B_j)+C_{\mathrm{cap}}\omega(B_j).
\]
By the lower normalization \eqref{eq:cover-normalization},
\[
 C_{\mathrm{cap}}\omega(B_j)
 \le \frac{C_{\mathrm{cap}}}{c_\pi}\,r_j^2\pi(B_j).
\]
Thus the two terms have the same scale, and the second asserted inequality
follows.
\end{proof}

The mixed oscillation and the mean-zero trace term are equivalent modulo ambient energy.

\begin{proposition}
Let $B=B(x,r)$ satisfy $0<\omega(B),\pi(B)<\infty$ and $r^2\pi(B)\simeq\omega(B)$. Then, for every $u\in N^{1,2}_{\loc}(X)$,
\begin{equation*}
 \begin{aligned}
 \omega(B)\int_B|\widetilde u-u_{B,\omega}|^2\,d\pi
 &\le\mathcal O_{\omega,\pi}(u;B)\\
 &\le2\omega(B)\int_B|\widetilde u-u_{B,\omega}|^2\,d\pi
 +C\omega(B)\int_{\lambda_2B}g_u^2\,d\omega.
 \end{aligned}
\end{equation*}
\end{proposition}

\begin{proof}
For fixed $y$, the variance identity gives
\[
 \int_B|u(x)-\widetilde u(y)|^2\,d\omega(x)
 \ge\omega(B)|u_{B,\omega}-\widetilde u(y)|^2.
\]
Integrating in $d\pi(y)$ proves the lower bound. The upper bound follows from
\[
 |u(x)-\widetilde u(y)|^2
 \le2|u(x)-u_{B,\omega}|^2+2|\widetilde u(y)-u_{B,\omega}|^2,
\]
followed by Lemma~\ref{lem:quadratic-pi} and $r^2\pi(B)\simeq\omega(B)$.
\end{proof}

For globalization we need the two local energies to control one another at the scale $r_j$. The next proposition is the precise local Fefferman--Phong pair that will be summed over the cover.

\begin{proposition}
\label{prop:local-fp-pair}
For every $j$ and every $u\in N^{1,2}_{\loc}(X)$,
\begin{align}
 \frac1{r_j^2}\int_{B_j}|u|^2\,d\omega
 &\le C\left[
 \int_{\Lambda B_j}g_u^2\,d\omega+
 \int_{B_j}|\widetilde u|^2\,d\pi
 \right],
 \label{eq:local-fp-forward}\\
 \int_{B_j}|\widetilde u|^2\,d\pi
 &\le C\left[
 \int_{\Lambda B_j}g_u^2\,d\omega+
 \frac1{r_j^2}\int_{B_j}|u|^2\,d\omega
 \right].
 \label{eq:local-fp-reverse}
\end{align}
\end{proposition}

\begin{proof}
Fix $B=B_j$ and $r=r_j$. The elementary decomposition
\[
 |u|^2\le2|u-u_{B,\omega}|^2+2|u_{B,\omega}|^2
\]
gives
\[
 \frac1{r^2}\int_B|u|^2\,d\omega
 \le \frac2{r^2}\int_B|u-u_{B,\omega}|^2\,d\omega
 +\frac{2\omega(B)}{r^2}|u_{B,\omega}|^2.
\]
The quadratic Poincar\'e inequality bounds the first term by
$C\int_{\Lambda B}g_u^2\,d\omega$. For the second term, integrate
\[
 |u_{B,\omega}|^2
 \le2|\widetilde u(y)|^2+2|\widetilde u(y)-u_{B,\omega}|^2
\]
with respect to $d\pi(y)$. Since the lower normalization \eqref{eq:cover-normalization} gives
$\omega(B)/r^2\le c_\pi^{-1}\pi(B)$, we obtain
\[
 \frac{\omega(B)}{r^2}|u_{B,\omega}|^2
 \le C\int_B|\widetilde u|^2\,d\pi
 +C\int_B|\widetilde u-u_{B,\omega}|^2\,d\pi.
\]
The last term is controlled by \eqref{eq:cover-local-trace}. Combining the three bounds proves \eqref{eq:local-fp-forward}.

For the reverse estimate, \eqref{eq:cover-local-trace} and \eqref{eq:cover-upper-normalization} give
\[
 \begin{aligned}
 \int_B|\widetilde u|^2\,d\pi
 &\le2\int_B|\widetilde u-u_{B,\omega}|^2\,d\pi
 +2\pi(B)|u_{B,\omega}|^2\\
 &\le C\int_{\Lambda B}g_u^2\,d\omega
 +\frac C{r^2}\int_B|u|^2\,d\omega.
 \end{aligned}
\]
This completes the proof.
\end{proof}

\subsection{Global equivalence}

The passage from the local pair in \cref{prop:local-fp-pair} to the global comparison uses only two facts: the balls cover $X$, and their fixed dilates have bounded overlap. No further regularity of $\pi$ is needed at this stage.

\begin{theorem}[Globalization over normalized capacitary covers]
\label{thm:global-fp}
Let $(\mathscr B,m)$ be a normalized capacitary cover and let $u\in N^{1,2}_{\loc}(X)$. Then
\begin{align}
 \int_Xm^2|u|^2\,d\omega
 &\le C\left[
 \int_Xg_u^2\,d\omega+
 \int_X|\widetilde u|^2\,d\pi
 \right],
 \label{eq:global-fp-forward}\\
 \int_X|\widetilde u|^2\,d\pi
 &\le C\left[
 \int_Xg_u^2\,d\omega+
 \int_Xm^2|u|^2\,d\omega
 \right],
 \label{eq:global-fp-reverse}
\end{align}
with all integrals interpreted in $[0,\infty]$. In particular, finiteness of either right-hand side implies finiteness of the corresponding left-hand side.

The constant depends only on the doubling and $2$-Poincar\'e data, $\lambda_2$, $\Lambda_0$, $\Lambda$, $N_\Lambda$, $C_m$, $c_\pi$, and $C_{\mathrm{cap}}$.
\end{theorem}

\begin{proof}
By \textup{(C1)} and \textup{(C3)},
\[
 \int_Xm^2|u|^2\,d\omega
 \le C\sum_jr_j^{-2}\int_{B_j}|u|^2\,d\omega.
\]
Using \eqref{eq:local-fp-forward} and summing gives
\[
 \int_Xm^2|u|^2\,d\omega
 \le C\sum_j\int_{\Lambda B_j}g_u^2\,d\omega
   +C\sum_j\int_{B_j}|\widetilde u|^2\,d\pi.
\]
Tonelli's theorem and \textup{(C2)} bound the first sum by $N_\Lambda\int_Xg_u^2\,d\omega$. The same bounded-overlap estimate, or simply $B_j\subset\Lambda B_j$, bounds the second sum by $N_\Lambda\int_X|\widetilde u|^2\,d\pi$. This proves \eqref{eq:global-fp-forward} with all quantities allowed to be infinite.

Similarly, \textup{(C1)} gives
\[
 \int_X|\widetilde u|^2\,d\pi
 \le\sum_j\int_{B_j}|\widetilde u|^2\,d\pi.
\]
After applying \eqref{eq:local-fp-reverse}, \textup{(C2)} controls the gradient sum. For the reciprocal-radius terms, \textup{(C3)} gives $r_j^{-2}|u|^2\lesssim m^2|u|^2$ on $B_j$, and the bounded multiplicity of the $B_j$'s gives
\[
 \sum_j r_j^{-2}\int_{B_j}|u|^2\,d\omega
 \lesssim\int_Xm^2|u|^2\,d\omega.
\]
This proves \eqref{eq:global-fp-reverse}.
\end{proof}

The normalization $\omega(B_j)/r_j^2$ in \textup{(C4)} is consistent with the quadratic scaling of the first-order energy. Indeed, under the metric rescaling $d_t=t d$, the same set $B_j$ has radius $tr_j$, the minimal upper gradient becomes $g_{u,t}=t^{-1}g_u$, and we set $m_t=t^{-1}m$ and $\pi_t=t^{-2}\pi$, while leaving $\omega$ unchanged. Then \textup{(C4)} is preserved and
\[
 \int_Xg_{u,t}^2\,d\omega,\qquad
 \int_Xm_t^2|u|^2\,d\omega,\qquad
 \int_X|\widetilde u|^2\,d\pi_t
\]
all equal $t^{-2}$ times their original values. This explains the factor $\omega(B_j)/r_j^2$.
The resulting global energy is independent, up to equivalence, of the particular normalized capacitary cover. Indeed, let $(\mathscr B,m)$ and $(\mathscr B',m')$ be normalized capacitary covers for the same pair $(\omega,\pi)$. Then
\[
 \int_Xg_u^2\,d\omega+\int_Xm^2|u|^2\,d\omega
 \simeq
 \int_Xg_u^2\,d\omega+\int_Xm'^2|u|^2\,d\omega
\]
for every $u\in N^{1,2}_{\loc}(X)$, with both sides interpreted in $[0,\infty]$. In particular, finiteness for one cover is equivalent to finiteness for the other. No pointwise comparison $m\simeq m'$ is asserted.

Indeed, apply \eqref{eq:global-fp-forward} to $(\mathscr B,m)$ and \eqref{eq:global-fp-reverse} to $(\mathscr B',m')$, and then exchange the roles of the two covers.

Pointwise comparability of multipliers from different covers requires additional scale information; in the Euclidean critical-radius construction it follows from the scale-decay hypothesis.
\subsection{Energy completions and their realization}

Let $\mathscr D=\LipLip(X)$. For $u\in\mathscr D$, define
\begin{align}
 \|u\|_{\Hpi}^2
 &:=\int_Xg_u^2\,d\omega+\int_X|\widetilde u|^2\,d\pi,
 \label{eq:Hpi-norm}\\
 \|u\|_{\Hm}^2
 &:=\int_Xg_u^2\,d\omega+\int_Xm^2|u|^2\,d\omega.
 \label{eq:Hm-norm}
\end{align}
The potential integral is finite because $u$ is bounded with compact support and $\pi$ is finite on compact sets. Theorem~\ref{thm:global-fp} guarantees finiteness of the multiplier integral.

We now identify the two homogeneous energy completions.

\begin{corollary}
\label{cor:energy-equivalence}
The norms \eqref{eq:Hpi-norm} and \eqref{eq:Hm-norm} are equivalent on $\mathscr D$. We denote the corresponding completions by $\Hpi$ and $\Hm$, respectively. They are canonically identified by the extension of the identity map.
\end{corollary}

The abstract completion can be represented by an actual local Newtonian function. The statement below is a realization of completion elements; no converse characterization of all local Newtonian functions with the displayed integrability properties is needed here.

\begin{proposition}
\label{prop:abstract-energy-realization}
Every element of the completion $\Hm$ determines a unique $\omega$-almost-everywhere class
\[
 u\in N^{1,2}_{\loc}(X),
 \qquad g_u\in L^2(X,d\omega),
 \qquad mu\in L^2(X,d\omega).
\]
Under the canonical identification of \cref{cor:energy-equivalence}, the quasicontinuous representative of this local Newtonian function satisfies $\widetilde u\in L^2(X,d\pi)$. Moreover, if $(u_n)\subset\mathscr D$ represents the completion element, then $\widetilde u_n\to\widetilde u$ in $L^2(X,d\pi)$.
\end{proposition}

\begin{proof}
Let $(u_n)\subset\mathscr D$ be Cauchy in $\Hm$. Choose an exhaustion by relatively compact open sets
\[
 U_1\Subset U_2\Subset\cdots\Subset X,
 \qquad \bigcup_{\ell=1}^{\infty}U_\ell=X.
\]
By \cref{lem:cover-local-bounds}, applied to $\overline{U_\ell}$, for every $\ell$ there is $c_\ell>0$ such that $m\ge c_\ell$ on $U_\ell$. Hence
\[
 \|u_n-u_k\|_{L^2(U_\ell,d\omega)}^2
 \le c_\ell^{-2}\|m(u_n-u_k)\|_{L^2(X,d\omega)}^2.
\]
Together with the upper-gradient term, this shows that $(u_n)$ is Cauchy in $N^{1,2}(U_\ell)$. Completeness gives a limit $u^{(\ell)}\in N^{1,2}(U_\ell)$. Uniqueness of the $L^2(U_\ell,d\omega)$ limit implies
\[
 u^{(\ell+1)}=u^{(\ell)}
 \quad\omega\text{-a.e. on }U_\ell,
\]
so the local limits patch to a unique function $u\in N^{1,2}_{\loc}(X)$.

Since $(u_n)$ is Cauchy in $\Hm$, it is bounded in that norm. Lower semicontinuity of the minimal weak upper gradient under the local Newtonian convergence gives, for every $\ell$,
\[
 \int_{U_\ell} g_u^2\,d\omega
 \le \liminf_{n\to\infty}\int_{U_\ell} g_{u_n}^2\,d\omega
 \le \sup_n\int_X g_{u_n}^2\,d\omega.
\]
Letting $\ell\to\infty$ and using monotone convergence yields
\[
 \int_X g_u^2\,d\omega
 \le \sup_n\int_X g_{u_n}^2\,d\omega<\infty.
\]
Moreover, $(mu_n)$ is Cauchy in $L^2(X,d\omega)$; let $h$ be its global limit. On each $U_\ell$, the function $m$ is bounded above and $u_n\to u$ in $L^2(U_\ell,d\omega)$, whence
\[
 mu_n\longrightarrow mu
 \quad\text{in }L^2(U_\ell,d\omega).
\]
Therefore $h=mu$ almost everywhere on every $U_\ell$, and hence $mu\in L^2(X,d\omega)$. If $(v_n)\subset\mathscr D$ is another sequence representing the same completion element, then $\|u_n-v_n\|_{\Hm}\to0$; the preceding local estimate yields $u_n-v_n\to0$ in $N^{1,2}(U_\ell)$ for every $\ell$, proving independence of the representative.

By \cref{cor:energy-equivalence}, $(u_n)$ is also Cauchy in $L^2(X,d\pi)$; denote its limit by $q$. Fix $j$. Since $u_n\to u$ in $N^{1,2}(\Lambda B_j)$, the reverse local Fefferman--Phong estimate \eqref{eq:local-fp-reverse}, applied to $u_n-u$, gives
\[
 \int_{B_j}|\widetilde u_n-\widetilde u|^2\,d\pi
 \le C\left[
 \int_{\Lambda B_j}g_{u_n-u}^2\,d\omega
 +\frac1{r_j^2}\int_{B_j}|u_n-u|^2\,d\omega
 \right]\longrightarrow0.
\]
Thus $\widetilde u_n\to\widetilde u$ in $L^2(B_j,d\pi)$.
The restriction of the global $L^2(d\pi)$ limit $q$ to $B_j$ must therefore equal $\widetilde u$. Since the countable family $\{B_j\}$ covers $X$, $q=\widetilde u$ $\pi$-almost everywhere on $X$. This proves the asserted concrete realization.
\end{proof}

In a general $2$-PI space, $u\mapsto\int g_u^2\,d\omega$ need not be quadratic. The energy completions are Hilbert spaces precisely in the infinitesimally Hilbertian case, i.e. when the Cheeger energy is quadratic; see \cite{AmbrosioGigliSavare}.
Adding the ambient $L^2(d\omega)$ term gives the corresponding inhomogeneous equivalence.

The same equivalence holds for the corresponding inhomogeneous completions. Adding the common term $\int_X|u|^2\,d\omega$ to \eqref{eq:Hpi-norm} and \eqref{eq:Hm-norm} yields equivalent norms and canonically identified completions. In this inhomogeneous case the function realization is immediate from the common $L^2(d\omega)$ term.

\subsection{Sharpness and stability}

The following examples show that the two essential analytic hypotheses, \textup{(C4)} and \textup{(C5)}, cannot simply be omitted.

\begin{proposition}
The following statements hold.
\begin{enumerate}[label=\textup{(\alph*)}]
\item The lower bound \textup{(C4)} cannot be omitted. Indeed, for any $d\ge1$, on $\R^d$ with $d\omega=dx$, take a unit-scale bounded-overlap cover, $m\equiv1$, and $\pi\equiv0$. Conditions \textup{(C1)}--\textup{(C3)} and \textup{(C5)} hold, but \eqref{eq:global-fp-forward} would reduce to the false inequality $\|u\|_{L^2(\R^d)}^2\lesssim\|\nabla u\|_{L^2(\R^d)}^2$ on $\R^d$.
\item Condition \textup{(C5)} does not follow from ball normalization. In $\R^3$ with $d\omega=dx$, let $B_j=B(j,1)$, $j\in\mathbb Z^3$, let $m\equiv1$, and let $\pi=\sum_{j\in\mathbb Z^3}\delta_j$, where $\delta_j$ is the Dirac mass at $j$. The unit lattice balls cover $\R^3$ and have bounded overlap, and $\pi(B_j)\simeq |B_j|$. Nevertheless,
\[
 \capw(\{j\};\Lambda_0B_j)=0.
\]
Thus \textup{(C5)} fails, and $\int|\widetilde u|^2\,d\pi$ is not representative-independent.
\end{enumerate}
\end{proposition}

The capacitary cover is also stable under adding a measure whose density is of the same reciprocal-scale size as the multiplier. The only non-immediate point is to show that the added measure still satisfies \textup{(C5)}.

\begin{lemma}
\label{lem:cover-augmentation}
If $(\mathscr B,m)$ is a normalized capacitary cover for $(\omega,\pi)$ and $c>0$, then it is a normalized capacitary cover for
\[
 d\pi_c:=d\pi+c m^2\,d\omega,
\]
with constants depending additionally on $c$.
\end{lemma}

\begin{proof}
Conditions \textup{(C1)}--\textup{(C3)} are unchanged, and the lower normalization \textup{(C4)} remains true because $\pi_c\ge\pi$. It remains to verify \textup{(C5)} for the added measure.

Fix a compact set $K\subset B_j$ and let $v\in N_0^{1,2}(\Lambda_0B_j)$ be admissible for $K$. Extend $v$ by zero outside $\Lambda_0B_j$. Apply \cref{lem:zero-boundary-sobolev} with $\Theta=\Lambda_0$ and with any fixed $\Theta'\ge4\lambda_S\Lambda_0$. Normalized $L^p$ monotonicity then gives
\[
 \int_{\Lambda_0B_j}|v|^2\,d\omega
 \le C r_j^2\int_{\Theta'B_j}g_v^2\,d\omega
 = C r_j^2\int_{\Lambda_0B_j}g_v^2\,d\omega,
\]
where the last equality uses that the minimal upper gradient of the zero extension vanishes almost everywhere outside $\Lambda_0B_j$. Since $\widetilde v\ge1$ quasi-everywhere on $K$ and capacity-zero sets are $\omega$-null, this also implies
\[
 \omega(K)\le C r_j^2\int_{\Lambda_0B_j}g_v^2\,d\omega.
\]
On $B_j$, condition \textup{(C3)} gives $m^2\le C_m^2r_j^{-2}$. Therefore
\[
 \int_Km^2\,d\omega
 \le C\int_{\Lambda_0B_j}g_v^2\,d\omega.
\]
Taking the infimum over all admissible $v$ yields
\[
 \int_Km^2\,d\omega\le C\capw(K;\Lambda_0B_j).
\]
Combining this estimate with \textup{(C5)} for $\pi$ proves \textup{(C5)} for $\pi_c$.
\end{proof}

The stability statement in \cref{lem:cover-augmentation} will also be useful when background multiplier measures are added. The same globalization argument applies to covers constructed only at a prescribed finite scale.

A finite-scale normalized capacitary cover is simply a normalized cover whose radii are uniformly bounded above. The proof of \cref{thm:global-fp} is unchanged and does not use the magnitude of the radii. A capped construction of this type is given in the Euclidean verification section for the augmented measure used in the finite-scale application.
\subsection{Covers obtained from local growth}

The normalized-cover definition in \cref{def:normalized-capacitary-cover} becomes particularly easy to verify when the potential measure decays faster on smaller balls by a fixed positive power. Once \textup{(C1)}--\textup{(C4)} are known, the growth-to-capacity theorem of \cref{sec:content-capacity} supplies the only remaining condition, namely \textup{(C5)}. For brevity, we call a normalized capacitary cover obtained in this way a \emph{growth-generated cover}.

\begin{corollary}
\label{cor:growth-generated-cover}
Assume \textup{(C1)}--\textup{(C3)}. Suppose that there exist $\delta>0$ and $C_{\mathrm{gr}}>0$ such that, uniformly in $j$,
\begin{equation}
 \pi(B(y,r))
 \le C_{\mathrm{gr}}
 \left(\frac r{r_j}\right)^\delta
 \frac{\omega(B(y,r))}{r^2},
 \qquad y\in B_j,\quad0<r\le r_j,
 \label{eq:cover-growth}
\end{equation}
and that the lower normalization \eqref{eq:cover-normalization} holds. Then $(\mathscr B,m)$ is a normalized capacitary cover, and all conclusions of \cref{thm:global-fp,cor:energy-equivalence,prop:abstract-energy-realization} follow.
\end{corollary}

\begin{proof}
Conditions \textup{(C1)}--\textup{(C3)} are assumed, and the lower normalization
\textup{(C4)} is exactly \eqref{eq:cover-normalization}. Fix $j$. Applying
\cref{thm:growth-capacity} to $B_0=B_j$ with $R=r_j$ and the growth bound
\eqref{eq:cover-growth} gives
\[
 \pi(K)\le C\capw(K;\Lambda_0B_j),
 \qquad K\subset B_j\text{ compact},
\]
where $C$ is independent of $j$. This is condition \textup{(C5)} with a uniform
constant. Hence $(\mathscr B,m)$ is a normalized capacitary cover.

The global Fefferman--Phong inequalities now follow from \cref{thm:global-fp},
the equivalence of the energy completions from \cref{cor:energy-equivalence},
and the concrete realization from \cref{prop:abstract-energy-realization}. The
matching upper normalization may also be read directly from
\eqref{eq:cover-growth} with $y=x_j$ and $r=r_j$, or from
\cref{lem:cover-upper-normalization}.
\end{proof}
\section{Verification classes and optimality}
\label{sec:verification}

This section shows how the abstract theory can be checked in concrete settings. The main case is a weighted Euclidean measure potential. There the dimensionless quantity $\PhiPi(x,r)$ measures the size of the potential on the ball $B(x,r)$ relative to the first-order energy scale $w(B(x,r))/r^2$. A critical radius is then chosen at the scale where this dimensionless potential has fixed size. We prove that these radii vary slowly enough to generate the normalized cover required in \cref{sec:normalized-covers}.

We then record a finite-scale version needed later, give a direct argument for reverse-H\"older function potentials, and specialize the metric theorem to Carnot groups and genuinely singular measures. The final examples explain why the structural assumption in the abstract theory is the quadratic PI property rather than membership in a particular Muckenhoupt class.

\subsection{Euclidean measure potentials and the critical radius}

Throughout this subsection,
\begin{equation*}
 d\omega=dw=w(x)\,dx
\end{equation*}
on $\R^d$, $d\ge3$, and assume
\begin{equation}
 w\in A_2(\R^d).
 \label{eq:A2-assumption}
\end{equation}
An $A_2$ weight is doubling, and the weighted Poincar\'e--Sobolev inequalities of Fabes--Kenig--Serapioni \cite[Theorems~1.3 and~1.5]{FabesKenigSerapioni} imply in particular the weak quadratic Poincar\'e inequality required in \cref{def:2pi}. Since the underlying Euclidean metric is complete and unbounded, $(\R^d,|\cdot|,w\,dx)$ is therefore an unbounded complete $2$-PI space. No lower reverse-doubling exponent is assumed in the critical-radius and energy arguments below; only the automatic reverse-volume exponent supplied by \cref{lem:reverse-doubling} is used. This is weaker than the ambient-weight hypothesis in \cite{BuiDoTrong2021}, where a polynomial reverse-doubling condition with exponent $\beta>2$ is imposed for the broader kernel and Hardy-space theory.

We write $W_{w,\loc}^{1,2}(\R^d)$ for the local weighted Sobolev space consisting of functions with locally square-integrable value and Euclidean gradient with respect to $dw$. For $A_2$ weights, the Newtonian and classical weighted Sobolev descriptions agree; in particular the minimal $2$-weak upper gradient agrees almost everywhere with the Euclidean weighted Sobolev gradient, see \cite[Proposition~A.12]{BjornBjornBook}. Hence, for $u\in W_{w,\loc}^{1,2}(\R^d)$, the abstract energy in \cref{sec:preliminaries,sec:content-capacity,sec:normalized-covers} is represented by $|\nabla u|^2\,dw$. When $d\omega=dw$, we abbreviate the relative variational capacity $\operatorname{cap}_{2,\omega}$ by $\operatorname{cap}_{2,w}$.

Fix a doubling upper exponent $Q_w>2$ such that
\begin{equation}
 w(B(x,tr))\le C t^{Q_w}w(B(x,r)),
 \qquad t\ge1.
 \label{eq:weighted-upper-dimension}
\end{equation}

Let $\pi\not\equiv0$ be a positive Radon measure and define
\begin{equation*}
 \PhiPi(x,r):=\frac{r^2\pi(B(x,r))}{w(B(x,r))}.
\end{equation*}
Thus $\PhiPi(x,r)$ compares the potential mass $\pi(B(x,r))$ with the quantity $w(B(x,r))/r^2$ having the same scaling as the quadratic first-order energy on the ball. The critical radius is the supremal scale at which this dimensionless ratio remains below a fixed threshold.

We assume, for some $\delta>0$ and constants $C_0,C_M>0$, the scale-decay condition
\begin{equation}
 \PhiPi(x,r)
 \le C_0\left(\frac rR\right)^\delta\PhiPi(x,R),
 \qquad x\in\R^d,\quad0<r<R,
 \label{eq:global-scale-decay}
\end{equation}
and the controlled-enlargement condition
\begin{equation}
 \pi(B(x,2r))
 \le C_M\left[\pi(B(x,r))+\frac{w(B(x,r))}{r^2}\right]
 \label{eq:controlled-enlargement}
\end{equation}
for every $x\in\R^d$ and $r>0$. These are the measure hypotheses used in the generalized degenerate Schr\"odinger framework of \cite{BuiDoTrong2021}. For later use set
\begin{equation*}
 C_*:=4C_M,
 \qquad A_*:=2C_*=8C_M.
\end{equation*}

The first consequence is a polynomial upper bound for how fast $\PhiPi$ can grow when the radius is enlarged.

\begin{lemma}
\label{lem:Phi-polynomial-growth}
There exist $C\ge1$ and $\gamma_\Phi>0$ such that, for every $x\in\R^d$, $r>0$, and $t\ge1$,
\begin{equation*}
 \PhiPi(x,tr)+1
 \le C t^{\gamma_\Phi}\bigl(\PhiPi(x,r)+1\bigr).
\end{equation*}
\end{lemma}

\begin{proof}
By \eqref{eq:controlled-enlargement} and the monotonicity $w(B(x,r))\le w(B(x,2r))$,
\[
 \begin{aligned}
 \PhiPi(x,2r)
 &\le 4C_M\left[
 \frac{r^2\pi(B(x,r))}{w(B(x,2r))}
 +\frac{w(B(x,r))}{w(B(x,2r))}
 \right]\\
 &\le 4C_M\bigl(\PhiPi(x,r)+1\bigr).
 \end{aligned}
\]
Iteration gives the claim for dyadic $t$. For general $t\ge1$, choose $n\ge0$ with $2^n\le t<2^{n+1}$. Since $B(x,tr)\subset B(x,2^{n+1}r)$ and $w$ is doubling,
\[
 \begin{aligned}
 \PhiPi(x,tr)
 &=\frac{(tr)^2\pi(B(x,tr))}{w(B(x,tr))}\\
 &\le C\frac{(2^{n+1}r)^2\pi(B(x,2^{n+1}r))}{w(B(x,2^{n+1}r))}
 =C\PhiPi(x,2^{n+1}r),
 \end{aligned}
\]
where the fixed constant comes from comparing $w(B(x,2^{n+1}r))$ with $w(B(x,tr))$. Applying the dyadic estimate at $2^{n+1}$ and using $2^{n+1}\le2t$ proves the asserted bound for arbitrary $t$.
\end{proof}

For any threshold $A>0$, define
\begin{equation}
 \rho_A(x)
 :=\sup\{r>0:\PhiPi(x,r)\le A\},
 \qquad m_A(x):=\rho_A(x)^{-1}.
 \label{eq:critical-radius-definition}
\end{equation}
The number $\rho_A(x)$ is therefore a local length scale and $m_A(x)$ its reciprocal multiplier. Importantly, the admissible set in \eqref{eq:critical-radius-definition} need not be an interval: $r\mapsto\PhiPi(x,r)$ is not assumed to be monotone. The proofs use only the scale-decay relation and the definition of the supremum.

We first separate the properties that hold for every positive threshold from the lower normalization needed for the critical cover.

\begin{lemma}
\label{lem:critical-normalization}
For every $A>0$ and every $x\in\R^d$,
\begin{equation}
 0<\rho_A(x)<\infty,
 \qquad
 \PhiPi(x,\rho_A(x))\le A.
 \label{eq:rho-positive-finite}
\end{equation}
If $A\ge A_*$, then there is $c_A>0$ with
\begin{equation}
 c_A\le \PhiPi(x,\rho_A(x))\le A
 \label{eq:critical-Phi-normalization}
\end{equation}
uniformly in $x$. In particular, for $A\ge A_*$,
\begin{equation}
 \pi(B(x,\rho_A(x)))
 \simeq \frac{w(B(x,\rho_A(x)))}{\rho_A(x)^2}.
 \label{eq:critical-ball-normalization}
\end{equation}
\end{lemma}

\begin{proof}
Fix $x$. Since $\pi\not\equiv0$, a sufficiently large ball centered at $x$ has positive $\pi$-measure. Applying \eqref{eq:global-scale-decay} below such a scale shows that $\PhiPi(x,r)\to0$ as $r\downarrow0$. Reversing \eqref{eq:global-scale-decay} above a scale at which $\PhiPi$ is positive shows that $\PhiPi(x,R)\to\infty$ as $R\to\infty$. This proves \eqref{eq:rho-positive-finite}.

Let $\rho=\rho_A(x)$. For each $k$ choose an admissible radius $s_k>\rho-1/k$, and set
\[
 r_k:=\max\{s_1,\ldots,s_k\}.
\]
Each $r_k$ is one of the admissible radii $s_1,\ldots,s_k$, so $\PhiPi(x,r_k)\le A$, while $r_k\uparrow\rho$. Since $B(x,r_k)\uparrow B(x,\rho)$, continuity from below gives
\[
 \pi(B(x,r_k))\to\pi(B(x,\rho)),
 \qquad
 w(B(x,r_k))\to w(B(x,\rho)),
\]
and therefore $\PhiPi(x,\rho)\le A$. On the other hand, $2\rho>\rho_A(x)$, hence $2\rho$ is not admissible and $\PhiPi(x,2\rho)>A$. The controlled-enlargement estimate used in the proof of \cref{lem:Phi-polynomial-growth} gives
\[
 \PhiPi(x,2\rho)\le C_*\bigl(\PhiPi(x,\rho)+1\bigr).
\]
Thus
\[
 \PhiPi(x,\rho)>A/C_*-1\ge1,
\]
because $A\ge A_*=2C_*$. Hence \eqref{eq:critical-Phi-normalization} holds, for example with $c_A:=A/C_*-1>0$.
\end{proof}

From now on in the critical-cover construction we fix a threshold $A\ge A_*$. The critical radius then satisfies a global admissibility estimate.

\begin{lemma}
\label{lem:global-critical-comparison}
There exist $a,b>0$ and $C_\rho\ge1$, depending only on the fixed threshold $A$ and the $A_2$ data and the constants in \eqref{eq:global-scale-decay}--\eqref{eq:controlled-enlargement}, such that, for all $x,y\in\R^d$,
\begin{equation}
 C_\rho^{-1}\rho_A(x)
 \left(1+\frac{|x-y|}{\rho_A(x)}\right)^{-a}
 \le \rho_A(y)
 \le C_\rho\rho_A(x)
 \left(1+\frac{|x-y|}{\rho_A(x)}\right)^b.
 \label{eq:global-rho-comparison}
\end{equation}
Equivalently,
\begin{equation}
 c\,m_A(x)\bigl(1+|x-y|m_A(x)\bigr)^{-b}
 \le m_A(y)
 \le C\,m_A(x)\bigl(1+|x-y|m_A(x)\bigr)^a.
 \label{eq:global-m-comparison}
\end{equation}
\end{lemma}

\begin{proof}
Write $r=\rho_A(x)$ and
\[
 T:=1+\frac{|x-y|}{r}\ge1,
 \qquad S:=2Tr.
\]
Then
\[
 B(x,r)\subset B(y,S)\subset B(x,3Tr).
\]
By \cref{lem:Phi-polynomial-growth}, doubling, and \eqref{eq:critical-Phi-normalization}, there is a structural exponent $\Gamma_\Phi>0$ such that
\begin{equation}
 \PhiPi(y,S)\le CT^{\Gamma_\Phi}.
 \label{eq:cross-centre-Phi-upper}
\end{equation}
Indeed,
\[
 \PhiPi(y,S)
 \le \frac{S^2\pi(B(x,3Tr))}{w(B(y,S))}
 \le C T^{2+Q_w}\PhiPi(x,3Tr)
 \le CT^{\Gamma_\Phi},
\]
where doubling compares $w(B(y,S))$ with $w(B(x,r))$ and \cref{lem:Phi-polynomial-growth} controls $\PhiPi(x,3Tr)$.

The constants in the two comparison estimates are chosen sequentially. First choose $a>0$ so that $(a+1)\delta>\Gamma_\Phi$; next choose the small constant $c$ in the lower-bound argument; independently, the large constant $L$ will be chosen in the upper-bound argument. Put $s=crT^{-a}$. Since $s<S$, \eqref{eq:global-scale-decay} and \eqref{eq:cross-centre-Phi-upper} give
\[
 \PhiPi(y,s)
 \le C_0(s/S)^\delta\PhiPi(y,S)
 \le Cc^\delta T^{\Gamma_\Phi-(a+1)\delta}.
\]
Choosing $c$ sufficiently small makes the last expression at most $A$. Thus $s$ is admissible in \eqref{eq:critical-radius-definition}, and
\[
 \rho_A(y)\ge crT^{-a}.
\]

For the reverse estimate, $B(x,r)\subset B(y,S)$ gives a lower bound for the numerator, while
$B(y,S)\subset B(x,3Tr)$ and \eqref{eq:weighted-upper-dimension} give
$w(B(y,S))\le CT^{Q_w}w(B(x,r))$. Since $S\simeq Tr$,
\begin{equation}
 \PhiPi(y,S)
 \ge cT^{2-Q_w}\PhiPi(x,r)
 \ge c_AT^{2-Q_w}.
 \label{eq:cross-centre-Phi-lower}
\end{equation}
Let
\[
 R_y:=LST^{(Q_w-2)/\delta}.
\]
Reversing \eqref{eq:global-scale-decay} between $S$ and $R_y$ and using \eqref{eq:cross-centre-Phi-lower}, we obtain
\[
 \PhiPi(y,R_y)
 \ge C_0^{-1}(R_y/S)^\delta\PhiPi(y,S)
 \ge c_AC_0^{-1}L^\delta.
\]
Choose $L$ so large that the right-hand side is greater than $C_0A$. For every $R>R_y$, another reversed use of \eqref{eq:global-scale-decay} gives $\PhiPi(y,R)>A$. Hence no radius greater than $R_y$ is admissible, and
\[
 \rho_A(y)\le R_y
 \le C rT^{1+(Q_w-2)/\delta}.
\]
This proves \eqref{eq:global-rho-comparison} with $b=1+(Q_w-2)/\delta$. Taking reciprocals gives \eqref{eq:global-m-comparison}.
\end{proof}

The global comparison in \cref{lem:global-critical-comparison} has the following local consequence.

\begin{corollary}
\label{cor:local-critical-comparison}
For every fixed $L\ge1$ there exists $C_L\ge1$ such that
\begin{equation*}
 C_L^{-1}\rho_A(x)\le\rho_A(y)\le C_L\rho_A(x)
\end{equation*}
whenever $|x-y|\le L\rho_A(x)$.
\end{corollary}

Different positive thresholds generate equivalent critical radii; no lower normalization at the smaller threshold is needed for this comparison.

\begin{lemma}
\label{lem:threshold-equivalence}
Let $A^{(1)},A^{(2)}>0$. Then
\begin{equation*}
 \rho_{A^{(1)}}(x)\simeq\rho_{A^{(2)}}(x)
\end{equation*}
uniformly in $x$, with constants depending only on $A^{(1)},A^{(2)},C_0,\delta$ and the structural data.
\end{lemma}

\begin{proof}
Assume $A^{(1)}\le A^{(2)}$. Then the admissible set for $A^{(1)}$ is contained in the admissible set for $A^{(2)}$, so
\[
 \rho_{A^{(1)}}(x)\le \rho_{A^{(2)}}(x).
\]
Fix $x$ and write $r_i=\rho_{A^{(i)}}(x)$ for $i=1,2$. The same continuity-from-below argument used in \cref{lem:critical-normalization} gives
\[
 \PhiPi(x,r_2)\le A^{(2)};
\]
this upper bound does not require $A^{(2)}\ge A_*$. Let $\varepsilon>0$ and set $s=(1+\varepsilon)r_1$. Since $s>r_1$, the definition of the supremum gives
\[
 \PhiPi(x,s)>A^{(1)}.
\]
If $s\ge r_2$, then $r_1\ge r_2/(1+\varepsilon)$. If $s<r_2$, the scale-decay inequality yields
\[
 A^{(1)}<\PhiPi(x,s)
 \le C_0\left(\frac{s}{r_2}\right)^\delta\PhiPi(x,r_2)
 \le C_0A^{(2)}\left(\frac{s}{r_2}\right)^\delta.
\]
Hence
\[
 \frac{r_1}{r_2}
 >\frac1{1+\varepsilon}
 \left(\frac{A^{(1)}}{C_0A^{(2)}}\right)^{1/\delta}.
\]
Letting $\varepsilon\downarrow0$ gives the required lower comparison. Exchanging the thresholds when necessary proves the claim.
\end{proof}

In particular, equation~(7) of \cite{BuiDoTrong2021} defines the historical critical radius as $\rho_{C_1}$ for the same functional $\PhiPi$, where $C_1>0$ denotes the fixed threshold constant used in that paper. Thus \cref{lem:threshold-equivalence} compares $\rho_{C_1}$ directly with every $\rho_A$, $A\ge A_*$, without any reverse-doubling assumption on $w$.
\subsection{Critical covers and the Euclidean verification theorem}

Throughout this subsection, we retain the Euclidean hypotheses and notation of the preceding subsection: $d\omega=dw=w\,dx$ on $\R^d$, $d\ge3$, $w\in A_2(\R^d)$, $\pi\not\equiv0$ is a positive Radon measure satisfying \eqref{eq:global-scale-decay} and \eqref{eq:controlled-enlargement}, and the threshold $A\ge A_*$ fixed above remains in force.

A maximal disjoint-family argument gives the required critical covering.

\begin{proposition}
\label{prop:critical-covering}
There is a sequence $\{x_j\}_{j\ge1}\subset\R^d$ such that, with
\begin{equation*}
 r_j:=\rho_A(x_j),
 \qquad B_j:=B(x_j,r_j),
\end{equation*}
one has
\begin{equation}
 \R^d=\bigcup_jB_j
 \label{eq:critical-covering}
\end{equation}
and, for every fixed $\Lambda\ge1$, there exists $N_\Lambda<\infty$ such that
\begin{equation}
 \sum_j\one_{\Lambda B_j}\le N_\Lambda
 \quad\hbox{pointwise on }\R^d.
 \label{eq:critical-bounded-overlap}
\end{equation}
Furthermore,
\begin{equation}
 m_A(x)\simeq r_j^{-1},
 \qquad x\in B_j.
 \label{eq:critical-m-comparison}
\end{equation}
\end{proposition}

\begin{proof}
Let $C_{\mathrm{loc}}$ be the local comparison constant from \cref{cor:local-critical-comparison} with $L=1$. Choose
\[
 0<\eta<\min\left\{\frac1{10},\frac1{2(C_{\mathrm{loc}}+1)}\right\}
\]
and select a maximal pairwise disjoint family
\[
 \{B(x_j,\eta\rho_A(x_j))\}_j.
\]
The family is countable. If $x$ is not one of the selected centers, maximality gives an index $j$ for which
\[
 B(x,\eta\rho_A(x))\cap B(x_j,\eta r_j)\ne\varnothing.
\]
If $\rho_A(x)\le r_j$, then $|x-x_j|<2\eta r_j<r_j$. If $\rho_A(x)>r_j$, then $|x-x_j|<2\eta\rho_A(x)<\rho_A(x)$. Applying \cref{cor:local-critical-comparison} with center $x$ and $L=1$ gives $\rho_A(x)\le C_{\mathrm{loc}}r_j$, whence
\[
 |x-x_j|<\eta(C_{\mathrm{loc}}+1)r_j<r_j.
\]
Thus $x\in B_j$, proving \eqref{eq:critical-covering}.

Suppose $z\in\Lambda B_i\cap\Lambda B_j$ and $r_i\ge r_j$. Then
\[
 |x_i-x_j|\le\Lambda(r_i+r_j)\le2\Lambda r_i.
\]
By \cref{cor:local-critical-comparison}, $r_j\simeq_{\Lambda}r_i$, where the subscript indicates that the comparison constant may depend on the fixed dilation $\Lambda$. Hence all reduced balls $B(x_j,\eta r_j)$ whose $\Lambda$-dilates contain a fixed point have comparable radii and lie in a Euclidean ball of comparable radius. Recall that a metric space is called geometrically doubling if every ball contains only uniformly many pairwise disjoint balls of comparable radius; Euclidean space has this property. Hence the number of such reduced balls is uniformly bounded. This proves \eqref{eq:critical-bounded-overlap}. Finally, \eqref{eq:critical-m-comparison} follows from local comparability.
\end{proof}

The scale-decay hypothesis yields uniform growth on the selected balls.

\begin{lemma}
\label{lem:growth-critical-balls}
Uniformly in $j$,
\begin{equation}
 \pi(B(y,s))
 \le C\left(\frac{s}{r_j}\right)^\delta
 \frac{w(B(y,s))}{s^2},
 \qquad y\in B_j,\quad0<s\le r_j.
 \label{eq:growth-on-critical-balls}
\end{equation}
\end{lemma}

\begin{proof}
Let $r=r_j$ and $y\in B_j$. Since $|y-x_j|<r$,
\[
 B(y,r)\subset B(x_j,2r),
 \qquad
 B(x_j,r)\subset B(y,2r).
\]
The controlled-enlargement condition and the critical normalization give
\[
 \pi(B(y,r))
 \le \pi(B(x_j,2r))
 \le C_M\left[\pi(B(x_j,r))+\frac{w(B(x_j,r))}{r^2}\right]
 \le C\frac{w(B(x_j,r))}{r^2}.
\]
Because $B(x_j,r)\subset B(y,2r)$, doubling implies
$w(B(x_j,r))\le w(B(y,2r))\le Cw(B(y,r))$. Therefore
\[
 \PhiPi(y,r)=\frac{r^2\pi(B(y,r))}{w(B(y,r))}\le C.
\]
For $0<s\le r$, the scale-decay condition at the center $y$ yields
\[
 \PhiPi(y,s)
 \le C_0\left(\frac{s}{r}\right)^\delta\PhiPi(y,r)
 \le C\left(\frac{s}{r}\right)^\delta.
\]
Rewriting this inequality in terms of $\pi(B(y,s))$ gives \eqref{eq:growth-on-critical-balls}.
\end{proof}

\begin{theorem}[Euclidean verification]
\label{thm:euclidean-verification}
Fix $A\ge A_*$ as above, and let $\Lambda$ be the fixed structural dilation from \eqref{eq:dilation-package}. Under \eqref{eq:A2-assumption}, \eqref{eq:global-scale-decay}, and \eqref{eq:controlled-enlargement}, the critical balls in \cref{prop:critical-covering}, together with $m_A$, form a normalized growth-generated cover in the sense of \cref{cor:growth-generated-cover}. The local and global estimates below hold for every $u\in W_{w,\loc}^{1,2}(\R^d)$, with all global terms interpreted in $[0,\infty]$. Consequently,
\begin{equation*}
 \pi(K)\le C\capwt(K;\Lambda_0B_j),
 \qquad K\subset B_j\text{ compact},
\end{equation*}
\begin{equation*}
 \int_{B_j}\int_{B_j}|u(x)-\widetilde u(y)|^2\,dw(x)\,d\pi(y)
 \le Cr_j^2\pi(B_j)\int_{\Lambda B_j}|\nabla u|^2\,dw,
\end{equation*}
and globally
\begin{equation}
 \int_{\R^d}m_A^2|u|^2\,dw
 \le C\left(\int_{\R^d}|\nabla u|^2\,dw+
 \int_{\R^d}|\widetilde u|^2\,d\pi\right),
 \label{eq:euclidean-fp-forward}
\end{equation}
\begin{equation}
 \int_{\R^d}|\widetilde u|^2\,d\pi
 \le C\left(\int_{\R^d}|\nabla u|^2\,dw+
 \int_{\R^d}m_A^2|u|^2\,dw\right).
 \label{eq:euclidean-fp-reverse}
\end{equation}
In particular, finiteness of the gradient term and either potential term implies finiteness of the other potential term.
\end{theorem}

\begin{proof}
By \cref{prop:critical-covering}, the critical balls cover $\R^d$ pointwise,
their fixed dilates have uniformly bounded overlap, and
$m_A(x)\simeq r_j^{-1}$ on $B_j$. These are conditions
\textup{(C1)}--\textup{(C3)} of a normalized capacitary cover. The normalization
\eqref{eq:critical-ball-normalization} gives
\[
 \pi(B_j)\simeq \frac{w(B_j)}{r_j^2},
\]
so condition \textup{(C4)} also holds uniformly. Finally,
\cref{lem:growth-critical-balls} gives the positive-power growth estimate
\eqref{eq:cover-growth} on every $B_j$. Therefore
\cref{cor:growth-generated-cover} supplies condition \textup{(C5)} and shows
that the critical balls together with $m_A$ form a normalized capacitary cover.

The compact-set capacity estimate is precisely the resulting condition
\textup{(C5)}, and the mixed estimate follows from \cref{prop:cover-local-estimates}. The two global inequalities are the conclusions of
\cref{thm:global-fp}. In the weighted Euclidean setting $|\nabla u|$ is an
admissible $2$-weak upper gradient, so the abstract gradient terms are bounded
by the displayed Euclidean gradient integrals.
\end{proof}

Thus the Euclidean energy-level verification requires only $w\in A_2$, together with the scale-decay and controlled-enlargement hypotheses on the potential measure. The additional polynomial reverse-doubling condition with exponent $\beta>2$ used in \cite{BuiDoTrong2021} is not needed for \cref{thm:euclidean-verification}.
\subsection{A finite-scale critical-cover construction}

For a later local comparison, it is useful to record a finite-scale version. The proposition below fixes a largest scale $R_*>0$ on which the assumptions are known and an augmentation parameter $M>0$. Two modifications are necessary. First, the admissible radii are capped strictly below $R_*$. Second, we add the harmless background measure $MR_*^{-2}dw$. This augmentation prevents the capped critical radius from losing its lower normalization when it reaches the hard cap. The explicit margin between the working scale and $R_*$ ensures that every enlargement and every reversed scale comparison remains inside the assumed range.

\begin{proposition}
\label{prop:finite-scale-cover}
Let $d\ge3$. Fix $R_*>0$ and a prescribed dilation $\Lambda\ge1$. Let $w\in A_2(\R^d)$ and let $\pi$ be a positive Radon measure. Let $C_D\ge1$ be a fixed doubling constant for the measure $w\,dx$. Set
\[
 \PhiPi(x,r):=\frac{r^2\pi(B(x,r))}{w(B(x,r))}.
\]
Assume that there are constants $\delta>0$ and $C_0,C_E>0$ such that
\begin{equation}
 \PhiPi(x,r)
 \le C_0\left(\frac rR\right)^\delta\PhiPi(x,R),
 \qquad x\in\R^d,\quad 0<r<R\le R_*,
 \label{eq:finite-scale-decay-base}
\end{equation}
and
\begin{equation}
 \pi(B(x,2r))
 \le C_E\left[\pi(B(x,r))+\frac{w(B(x,r))}{r^2}\right],
 \qquad x\in\R^d,\quad 0<r\le R_*/2.
 \label{eq:finite-scale-enlargement-base}
\end{equation}
Set
\begin{equation*}
 \delta_0:=\min\{\delta,2\}.
\end{equation*}
For $M>0$, define
\begin{equation}
 d\pi_M:=d\pi+MR_*^{-2}\,dw,
 \qquad
 \Phi_M(x,r):=\frac{r^2\pi_M(B(x,r))}{w(B(x,r))}
 =\PhiPi(x,r)+M\left(\frac r{R_*}\right)^2.
 \label{eq:finite-scale-augmented-measure}
\end{equation}
Set
\begin{equation*}
 C_{\mathrm{aug}}:=4\max\{C_E,C_D\}.
\end{equation*}
Choose $A_M\ge2C_{\mathrm{aug}}$.
Put $L_0:=2\Lambda+2$. Choose $c_*\in(0,1)$ sufficiently small that
\begin{equation}
 2(2L_0+2)c_*\le1,
 \qquad 2c_*\le1,
 \label{eq:finite-margin-choice}
\end{equation}
and set $r_*:=c_*R_*$. Define
\begin{equation*}
 \rho_{M,*}(x)
 :=\sup\{0<r\le r_*:\Phi_M(x,r)\le A_M\},
 \qquad m_{M,*}(x):=\rho_{M,*}(x)^{-1}.
\end{equation*}
Then the following conclusions hold.

\begin{enumerate}[label=\textup{(\roman*)}]
\item There exists $c_{M,*}>0$, independent of $x$, such that for every $x\in\R^d$,
\begin{equation}
 0<\rho_{M,*}(x)\le r_*,
 \qquad
 c_{M,*}\le\Phi_M(x,\rho_{M,*}(x))\le A_M.
 \label{eq:finite-critical-normalization}
\end{equation}
In particular,
\begin{equation*}
 \pi_M(B(x,\rho_{M,*}(x)))
 \simeq\frac{w(B(x,\rho_{M,*}(x)))}{\rho_{M,*}(x)^2}.
\end{equation*}

\item For every fixed $L\in[1,L_0]$, there exists $C_L\ge1$ such that
\begin{equation}
 C_L^{-1}\rho_{M,*}(x)\le\rho_{M,*}(y)\le C_L\rho_{M,*}(x)
 \label{eq:finite-local-comparison}
\end{equation}
whenever $|x-y|\le L\rho_{M,*}(x)$.

\item There are a constant $N_\Lambda<\infty$ and a sequence $\{x_j\}_{j\ge1}$ such that, writing
\begin{equation*}
 r_j:=\rho_{M,*}(x_j),
 \qquad B_j:=B(x_j,r_j),
\end{equation*}
one has
\begin{equation}
 \R^d=\bigcup_jB_j,
 \qquad
 \sum_j\one_{\Lambda B_j}\le N_\Lambda
 \quad\text{pointwise},
 \label{eq:finite-cover-overlap}
\end{equation}
and
\begin{equation}
 m_{M,*}(x)\simeq r_j^{-1},
 \qquad x\in B_j.
 \label{eq:finite-m-comparison}
\end{equation}

\item Uniformly in $j$,
\begin{equation}
 \Phi_M(y,s)
 \le C\left(\frac{s}{r_j}\right)^{\delta_0},
 \qquad y\in B_j,\quad0<s\le r_j,
 \label{eq:finite-positive-gain-Phi}
\end{equation}
or equivalently,
\begin{equation}
 \pi_M(B(y,s))
 \le C\left(\frac{s}{r_j}\right)^{\delta_0}
 \frac{w(B(y,s))}{s^2}.
 \label{eq:finite-positive-gain-measure}
\end{equation}
Thus the selected balls, together with $m_{M,*}$, form a normalized growth-generated cover for $\pi_M$.
\end{enumerate}
All constants may depend on the fixed parameters $M$, $A_M$, $c_*$, and $\Lambda$, as well as on the structural data, but not on $x$, $y$, or $j$.
\end{proposition}

\begin{proof}
\emph{Step 1: finite-scale estimates for the augmented measure.}
Set $C_S:=\max\{C_0,1\}$. For $0<r<R\le R_*$, \eqref{eq:finite-scale-decay-base} gives
\begin{equation}
 \Phi_M(x,r)
 \le C_S\left(\frac rR\right)^{\delta_0}\Phi_M(x,R).
 \label{eq:finite-scale-decay-augmented}
\end{equation}
Indeed, the first term in \eqref{eq:finite-scale-augmented-measure} has exponent $\delta$, while the added term has exponent $2$.

Set $C_E':=\max\{C_E,C_D\}$. By \eqref{eq:finite-scale-enlargement-base} and doubling, whenever $2r\le R_*$,
\begin{equation}
 \pi_M(B(x,2r))
 \le C_E'\left[\pi_M(B(x,r))+\frac{w(B(x,r))}{r^2}\right].
 \label{eq:finite-enlargement-augmented}
\end{equation}
Consequently,
\begin{align*}
 \Phi_M(x,2r)
 &=\frac{4r^2\pi_M(B(x,2r))}{w(B(x,2r))}\\
 &\le4C_E'\bigl(\Phi_M(x,r)+1\bigr)
 \le C_{\mathrm{aug}}\bigl(\Phi_M(x,r)+1\bigr).
\end{align*}
Thus
\begin{equation}
 \Phi_M(x,2r)\le C_{\mathrm{aug}}\bigl(\Phi_M(x,r)+1\bigr),
 \qquad 2r\le R_*.
 \label{eq:finite-Phi-double}
\end{equation}
Iteration and comparison with the next dyadic scale yield an exponent $\gamma_M>0$ such that, for $x\in\R^d$, $r>0$, and $t\ge1$ with $tr\le R_*/2$,
\begin{equation}
 \Phi_M(x,tr)+1
 \le Ct^{\gamma_M}\bigl(\Phi_M(x,r)+1\bigr).
 \label{eq:finite-Phi-polynomial}
\end{equation}
The restriction $tr\le R_*/2$ is essential: the next dyadic radius used in the comparison is then at most $R_*$, so every invocation of \eqref{eq:finite-Phi-double} remains inside the assumed enlargement range.

\emph{Step 2: positivity and normalization.}
For fixed $x$, \eqref{eq:finite-scale-decay-base} implies $\PhiPi(x,r)\to0$ as $r\downarrow0$, and the added term also tends to zero. Thus $\rho_{M,*}(x)>0$. The upper bound $\rho_{M,*}(x)\le r_*$ is part of the definition.

Let $\rho=\rho_{M,*}(x)$. By the definition of the supremum, choose admissible $s_k>\rho-1/k$ and set $r_k:=\max\{s_1,\ldots,s_k\}$. Then each $r_k$ is admissible and $r_k\uparrow\rho$. Continuity from below of $\pi_M$ and $w\,dx$ then gives
\begin{equation}
 \Phi_M(x,\rho)\le A_M.
 \label{eq:finite-upper-normalization}
\end{equation}
If $\rho<r_*/2$, then $2\rho\le r_*$ and $2\rho$ cannot be admissible. Thus $\Phi_M(x,2\rho)>A_M$, and \eqref{eq:finite-Phi-double} gives
\begin{equation}
 \Phi_M(x,\rho)>A_M/C_{\mathrm{aug}}-1\ge1.
 \label{eq:finite-interior-lower}
\end{equation}
If $\rho\ge r_*/2$, then
\begin{equation*}
 \Phi_M(x,\rho)
 \ge M(\rho/R_*)^2\ge Mc_*^2/4.
\end{equation*}
Consequently \eqref{eq:finite-critical-normalization} holds with
\begin{equation*}
 c_{M,*}:=\min\{A_M/C_{\mathrm{aug}}-1,Mc_*^2/4\}>0.
\end{equation*}

\emph{Step 3: lower comparison of critical radii.}
Fix $L\in[1,L_0]$, write $r=\rho_{M,*}(x)$, and assume $|x-y|\le Lr$. Set
\[
 S:=(L+2)r,
 \qquad U:=(2L+2)r.
\]
Then $B(y,S)\subset B(x,U)$, and the strengthened margin \eqref{eq:finite-margin-choice} ensures $U\le R_*/2$. Thus \eqref{eq:finite-Phi-polynomial} is applicable between $r$ and $U$. Also
\[
 B(x,U)\subset B(y,(3L+2)r),
 \qquad \frac{3L+2}{L+2}\le3.
\]
Doubling and \eqref{eq:finite-Phi-polynomial} therefore give
\begin{equation*}
 \Phi_M(y,S)
 \le C_L\Phi_M(x,U)
 \le C_L(A_M+1).
\end{equation*}
Choose $\theta_L\in(0,1)$ so small that
\[
 C_S\left(\frac{\theta_L}{L+2}\right)^{\delta_0}C_L(A_M+1)\le A_M.
\]
Applying \eqref{eq:finite-scale-decay-augmented} at the center $y$ between $\theta_Lr$ and $S$ gives $\Phi_M(y,\theta_Lr)\le A_M$. Hence
\begin{equation}
 \rho_{M,*}(y)\ge\theta_Lr.
 \label{eq:finite-lower-comparison}
\end{equation}

\emph{Step 4: upper comparison of critical radii.}
Set $S_0:=(L+1)r$. The inclusions
\[
 B(x,r)\subset B(y,S_0)\subset B(x,(2L+1)r)
\]
and doubling give
\begin{equation}
 \Phi_M(y,S_0)\ge c_L\Phi_M(x,r).
 \label{eq:finite-cross-lower}
\end{equation}
Choose $D_L>\max\{2,L+1\}$ sufficiently large that
\begin{equation}
 C_S^{-1}\left(\frac{D_L}{L+1}\right)^{\delta_0}
 c_L(A_M/C_{\mathrm{aug}}-1)>C_SA_M.
 \label{eq:finite-D-choice}
\end{equation}
If $r\ge r_*/D_L$, then the cap immediately gives
\[
 \rho_{M,*}(y)\le r_*\le D_Lr.
\]
Suppose $r<r_*/D_L$. Then $r<r_*/2$, so \eqref{eq:finite-interior-lower} applies. Reversing \eqref{eq:finite-scale-decay-augmented} between $S_0$ and $D_Lr$, and using \eqref{eq:finite-cross-lower} and \eqref{eq:finite-D-choice}, gives
\begin{equation*}
 \Phi_M(y,D_Lr)>C_SA_M.
\end{equation*}
For every $R$ with $D_Lr<R\le r_*$, another reversed use of \eqref{eq:finite-scale-decay-augmented} yields $\Phi_M(y,R)>A_M$. Thus no radius larger than $D_Lr$ is admissible, and
\begin{equation*}
 \rho_{M,*}(y)\le D_Lr.
\end{equation*}
Together with \eqref{eq:finite-lower-comparison}, this proves \eqref{eq:finite-local-comparison}.

\emph{Step 5: construction and overlap of the cover.}
Let $C_{\mathrm{loc},*}$ be the comparison constant in \eqref{eq:finite-local-comparison} for $L=1$, and choose
\begin{equation*}
 0<\eta<\min\left\{\frac14,\frac1{2(C_{\mathrm{loc},*}+1)}\right\}.
\end{equation*}
Select a maximal pairwise disjoint family
\[
 \{B(x_j,\eta\rho_{M,*}(x_j))\}_{j\ge1}.
\]
Write $r_j=\rho_{M,*}(x_j)$. For arbitrary $x$, maximality gives an index $j$ such that
\[
 |x-x_j|<\eta(\rho_{M,*}(x)+r_j).
\]
If $\rho_{M,*}(x)\le r_j$, then $|x-x_j|<2\eta r_j<r_j$. If $\rho_{M,*}(x)>r_j$, then $|x-x_j|<2\eta\rho_{M,*}(x)<\rho_{M,*}(x)$, and \eqref{eq:finite-local-comparison}, centered at $x$, gives $\rho_{M,*}(x)\le C_{\mathrm{loc},*}r_j$. Hence
\[
 |x-x_j|<\eta(C_{\mathrm{loc},*}+1)r_j<r_j.
\]
Thus $x\in B_j$, proving the covering.

If $z\in\Lambda B_i\cap\Lambda B_j$ and $r_i\ge r_j$, then
\[
 |x_i-x_j|\le\Lambda(r_i+r_j)\le2\Lambda r_i.
\]
Since $2\Lambda\le L_0$, \eqref{eq:finite-local-comparison} gives $r_i\simeq_\Lambda r_j$. The reduced balls are pairwise disjoint, have comparable radii, and lie in a ball of comparable radius. The geometric-doubling property just used above therefore proves the pointwise overlap bound in \eqref{eq:finite-cover-overlap}. Finally, \eqref{eq:finite-m-comparison} follows from \eqref{eq:finite-local-comparison} with $L=1$.

\emph{Step 6: positive gain on the selected balls.}
Fix $j$, put $r=r_j$, and let $y\in B_j$. By \eqref{eq:finite-margin-choice}, $2r\le R_*$. Since $B(y,r)\subset B(x_j,2r)$, \eqref{eq:finite-enlargement-augmented}, doubling, and \eqref{eq:finite-upper-normalization} give
\[
 \pi_M(B(y,r))
 \le C\left[\pi_M(B(x_j,r))+\frac{w(B(x_j,r))}{r^2}\right]
 \le C\frac{w(B(y,r))}{r^2}.
\]
Thus $\Phi_M(y,r)\le C$. Applying \eqref{eq:finite-scale-decay-augmented} at the center $y$ between $s$ and $r$ proves \eqref{eq:finite-positive-gain-Phi}, and hence \eqref{eq:finite-positive-gain-measure}.

For clarity, the constants can be chosen in the following non-circular order:
\[
 C_S,\ C_E',\ C_{\mathrm{aug}},\ A_M,\ L_0,\ c_*,\ \theta_L,\ D_L,\ \eta.
\]
The hard cap enters Step~4 only in the alternative $r\ge r_*/D_L$. In the complementary alternative, $D_Lr<r_*\le R_*$ and all reversed scale-decay comparisons stay strictly within the prescribed range. In Step~3, the stronger choice $2(2L_0+2)c_*\le1$ keeps the terminal radius below $R_*/2$, precisely the range in which the dyadic polynomial estimate has been proved.
\end{proof}

The augmentation parameter $M$ is fixed throughout the finite-scale construction, and the constants are allowed to depend on it. The proposition must be applied to the unaugmented base measure; the measure $\pi_M$ in \eqref{eq:finite-scale-augmented-measure} is the output of the augmentation step. In particular, the construction does not iterate the augmentation.

\subsection{Reverse-H\"older function potentials}

Let $(X,d,\omega)$ be an unbounded complete $2$-PI space. Fix a Sobolev--Poincar\'e exponent $\kappa>1$ for which
\begin{equation}
 \left(\fint_B|u-u_{B,\omega}|^{2\kappa}\,d\omega\right)^{1/(2\kappa)}
 \le Cr_B\left(\fint_{\lambda_SB}g_u^2\,d\omega\right)^{1/2}.
 \label{eq:Sobolev-Poincare-RH}
\end{equation}
Let $d\pi=V\,d\omega$, where $q>1$, $q'=q/(q-1)$, and $V\in RH_q(d\omega)$ in the sense that there exists $C_{RH}\ge1$ such that
\begin{equation*}
 \left(\fint_B V^q\,d\omega\right)^{1/q}
 \le C_{RH}\fint_B V\,d\omega
\end{equation*}
for every ball $B$. Assume $q'\le\kappa$. This reverse-H\"older condition on the potential density is unrelated to the polynomial reverse-doubling condition sometimes imposed on the ambient weight; the present Euclidean energy verification does not require that additional ambient-weight assumption. The restriction $q'\le\kappa$ is exactly what allows the $L^{2q'}$ factor produced by H\"older's inequality to be controlled by the available $L^{2\kappa}$ Sobolev--Poincar\'e estimate. We use the zero-boundary Sobolev inequality from \cref{lem:zero-boundary-sobolev}.

For absolutely continuous potentials, the trace estimate can be verified directly, without first passing through the growth-to-capacity theorem. The capacity estimate below uses the same H\"older argument together with the zero-boundary Sobolev inequality; we spell out the auxiliary dilation because this is where a hidden outer ball can otherwise enter the proof.

\begin{proposition}
\label{prop:reverse-holder-direct}
Fix a dilation $\Lambda>1$. For every ball $B$ and every $u\in N^{1,2}_{\loc}(X)$,
\begin{equation}
 \int_B|\widetilde u-u_{B,\omega}|^2\,d\pi
 \le C\frac{r_B^2\pi(B)}{\omega(B)}
 \int_{\lambda_SB}g_u^2\,d\omega.
 \label{eq:RH-trace}
\end{equation}
Consequently, on balls satisfying $\pi(B)\simeq\omega(B)/r_B^2$, the mixed Poincar\'e estimate follows directly.

For compact $K\subset B$, one also has
\begin{equation}
 \pi(K)
 \le C\frac{r_B^2\pi(\Lambda B)}{\omega(\Lambda B)}
 \capw(K;\Lambda B).
 \label{eq:RH-capacity}
\end{equation}
Thus a uniform normalized capacity estimate additionally requires uniform control of the fixed-dilate ratio in the displayed capacity bound; it is not inferred from normalization on $B$ alone. This is precisely where the abstract growth-to-capacity theorem is stronger: its ball-growth hypothesis controls all smaller scales and produces a uniform coefficient without separately comparing the potential mass of $B$ and $\Lambda B$.
\end{proposition}

\begin{proof}
H\"older's inequality gives
\[
 \int_B|u-u_{B,\omega}|^2V\,d\omega
 \le \left(\int_BV^q\,d\omega\right)^{1/q}
 \left(\int_B|u-u_{B,\omega}|^{2q'}\,d\omega\right)^{1/q'}.
\]
By the reverse-H\"older condition,
\[
 \left(\int_BV^q\,d\omega\right)^{1/q}
 \le C\,\pi(B)\,\omega(B)^{1/q-1}.
\]
Since $q'\le\kappa$, normalized $L^p$ monotonicity followed by \eqref{eq:Sobolev-Poincare-RH} gives
\begin{align*}
 \left(\int_B|u-u_{B,\omega}|^{2q'}\,d\omega\right)^{1/q'}
 &\le \omega(B)^{1/q'}
 \left(\fint_B|u-u_{B,\omega}|^{2\kappa}\,d\omega\right)^{1/\kappa}\\
 &\le C r_B^2\omega(B)^{1/q'}
 \fint_{\lambda_SB}g_u^2\,d\omega.
\end{align*}
Because $1/q+1/q'=1$, the powers of $\omega(B)$ cancel, and since $\omega(B)\le\omega(\lambda_SB)$ we obtain
\[
 \int_B|u-u_{B,\omega}|^2V\,d\omega
 \le C\frac{r_B^2\pi(B)}{\omega(B)}
 \int_{\lambda_SB}g_u^2\,d\omega.
\]
This proves \eqref{eq:RH-trace}. For \eqref{eq:RH-capacity}, let $v\in N_0^{1,2}(\Lambda B)$ be admissible for $K$ and extend $v$ and $g_v$ by zero outside $\Lambda B$. Since admissibility gives $\widetilde v\ge1$ quasi-everywhere on $K$, while Sobolev-capacity-zero sets are $\omega$-null and hence also $\pi$-null because $d\pi=V\,d\omega$, we may use $|\widetilde v|^2$ as a pointwise majorant of $\one_K$ in the $d\pi$ integral. Choose once and for all
\[
 \Lambda'\ge4\lambda_S\Lambda.
\]
H\"older's inequality and the reverse-H\"older condition on $\Lambda B$ give
\begin{align*}
 \pi(K)
 &\le\int_{\Lambda B}|\widetilde v|^2V\,d\omega\\
 &\le C\pi(\Lambda B)
 \left(\fint_{\Lambda B}|v|^{2q'}\,d\omega\right)^{1/q'}.
\end{align*}
Since $q'\le\kappa$, normalized $L^p$ monotonicity and \cref{lem:zero-boundary-sobolev}, applied with $\Theta=\Lambda$ and $\Theta'=\Lambda'$, yield
\[
 \left(\fint_{\Lambda B}|v|^{2q'}\,d\omega\right)^{1/q'}
 \le C r_B^2\fint_{\Lambda'B}g_v^2\,d\omega.
\]
The zero extension has $g_v=0$ outside $\Lambda B$, and doubling compares $\omega(\Lambda'B)$ and $\omega(\Lambda B)$. Therefore
\[
 \pi(K)
 \le C\frac{r_B^2\pi(\Lambda B)}{\omega(\Lambda B)}
 \int_{\Lambda B}g_v^2\,d\omega.
\]
Taking the infimum over admissible $v$ proves \eqref{eq:RH-capacity}.
\end{proof}

In unweighted $\R^d$, one may take $\kappa=d/(d-2)$, and $q'\le\kappa$ is equivalent to $q\ge d/2$, the familiar reverse-H\"older threshold in Schr\"odinger analysis.
\subsection{Carnot groups}

Let $G$ be a Carnot group of homogeneous dimension $Q_G>2$, equipped with a homogeneous Carnot--Carath\'eodory distance and Haar measure $dg$, and let $\pi$ be a positive Radon measure on $G$. Let $\Lambda$ denote the fixed structural dilation from \eqref{eq:dilation-package}. For measurable $E\subset G$, write $|E|:=\int_E dg$. We write $B_G(x,r)$ for the corresponding metric ball and $\nabla_Hu$ for the horizontal gradient. Then $|B_G(x,r)|\simeq r^{Q_G}$, and the Newtonian Sobolev space agrees with the horizontal Sobolev space, with comparable energy represented by $|\nabla_Hu|$; see \cite{HKST}. 
The metric theorem immediately specializes to Carnot groups.

\begin{corollary}
\label{cor:Carnot-mixed}
Let $B_0=B_G(x_0,R)$. If there exist $\delta>0$ and $C_{G,\pi}>0$ such that
\begin{equation}
 \pi(B_G(y,r))\le C_{G,\pi} R^{-\delta}r^{Q_G-2+\delta},
 \qquad y\in B_0,\quad0<r\le R,
 \label{eq:Carnot-growth}
\end{equation}
and $\pi(B_0)\simeq R^{Q_G-2}$, then, for every $u\in N^{1,2}_{\loc}(G)$,
\begin{equation*}
 \int_{B_0}\int_{B_0}|u(x)-\widetilde u(y)|^2\,dg(x)\,d\pi(y)
 \le CR^2\pi(B_0)\int_{\Lambda B_0}|\nabla_Hu|^2\,dg.
\end{equation*}
\end{corollary}

\begin{proof}
Haar measure on $G$ is $Q_G$-Ahlfors regular, so
\[
 \frac{|B_G(y,r)|}{r^2}\simeq r^{Q_G-2}.
\]
Consequently \eqref{eq:Carnot-growth} can be rewritten as
\[
 \pi(B_G(y,r))
 \le C\left(\frac rR\right)^\delta
 \frac{|B_G(y,r)|}{r^2},
\]
which is precisely the subcritical growth condition
\eqref{eq:subcritical-growth}. The assumption
$\pi(B_0)\simeq R^{Q_G-2}$ gives the required lower normalization because
$|B_0|/R^2\simeq R^{Q_G-2}$. Thus \cref{cor:growth-mixed} applies. The Newtonian
quadratic energy on a Carnot group is comparable with the horizontal-gradient
energy, so the gradient term may be written using $|\nabla_Hu|$.
\end{proof}

\subsection{Quadratically admissible weights outside \texorpdfstring{$A_2$}{A2}}

Let $d\ge3$ and fix $\gamma>2$. Write $x=(x',x'')\in\R^2\times\R^{d-2}$ and set
\begin{equation*}
 w_\gamma(x):=(1+|x'|)^\gamma.
\end{equation*}

The measure $w_\gamma\,dx$ is doubling and supports a quadratic Poincar\'e inequality, but $w_\gamma\notin A_2(\R^d)$.

Indeed, the $2$-admissibility follows from \cite[Proposition~6]{Bjorn2001}, applied with the distinguished variable $x'\in\R^2$, $p=2$, and background weight equal to one. To see failure of $A_2$, let $B_R=B(0,R)$ with $R$ sufficiently large. Direct integration gives
\[
 \fint_{B_R}w_\gamma\,dx\simeq R^\gamma,
 \qquad
 \fint_{B_R}w_\gamma^{-1}\,dx\simeq R^{-2}.
\]
Hence the product of these two averages is comparable to $R^{\gamma-2}$ and tends to infinity.

\subsection{Why bare \texorpdfstring{$A_p$, $p>2$}{Ap, p greater than 2}, is insufficient}

Let $d\ge3$. Fix $p>2$ and choose
\begin{equation*}
 1<\alpha<p-1,
 \qquad w_\alpha(x)=|x_1|^\alpha.
\end{equation*}
Then $w_\alpha\in A_p(\R^d)$, but the quadratic Poincar\'e inequality fails.

Let $d\omega_\alpha=w_\alpha\,dx$ and $d\pi=d\omega_\alpha$. The scale-decay condition holds with $\delta=2$, and the controlled-enlargement condition holds, but the quadratic mixed Poincar\'e inequality fails.

Indeed, the classical power-weight criterion gives $w_\alpha\in A_p$ because $-1<\alpha<p-1$. Since $\pi=\omega_\alpha$, the dimensionless scale functional equals $r^2$, so the scale-decay condition holds with exponent two; controlled enlargement follows from doubling.

Let $B=B(0,1)$ and define
\[
 u_\varepsilon(x)=
 \begin{cases}
 -1,&x_1\le-\varepsilon,\\
 x_1/\varepsilon,&|x_1|<\varepsilon,\\
 1,&x_1\ge\varepsilon.
 \end{cases}
\]
Both $B$ and the weight $w_\alpha$ are symmetric under $x_1\mapsto-x_1$, while $u_\varepsilon$ is odd in $x_1$. Hence
\[
 (u_\varepsilon)_{B,\omega_\alpha}=0.
\]
Moreover, $|u_\varepsilon|\to1$ away from the hyperplane $\{x_1=0\}$ and $|u_\varepsilon|\le1$, so dominated convergence gives
\[
 \int_B|u_\varepsilon|^2\,d\omega_\alpha
 \longrightarrow\omega_\alpha(B)>0.
\]
On the transition strip $|x_1|<\varepsilon$ one has $|\nabla u_\varepsilon|=\varepsilon^{-1}$. Since the $(d-1)$-dimensional cross-sections of $B(0,1)$ have uniformly comparable size for $|x_1|<\varepsilon\le1/2$,
\[
 \int_B|\nabla u_\varepsilon|^2\,d\omega_\alpha
 \simeq \varepsilon^{-2}\int_{-\varepsilon}^{\varepsilon}|t|^\alpha\,dt
 \simeq\varepsilon^{\alpha-1}\longrightarrow0.
\]
The decisive point is $\alpha>1$, which makes the transition energy vanish. Thus the quadratic Poincar\'e inequality fails already on the ball $B(0,1)$. Since $\pi=\omega_\alpha$ and $r_B=1$, this ball is also critically normalized:
\[
 \pi(B)=\omega_\alpha(B)=\frac{\omega_\alpha(B)}{r_B^2}.
\]
Finally, the variance identity gives
\[
 \int_B\int_B|u_\varepsilon(x)-u_\varepsilon(y)|^2\,
 d\omega_\alpha(x)d\omega_\alpha(y)
 =2\omega_\alpha(B)
 \int_B|u_\varepsilon-(u_\varepsilon)_{B,\omega_\alpha}|^2\,d\omega_\alpha,
\]
whose left-hand side stays bounded below while the gradient energy tends to zero. Hence the quadratic mixed Poincar\'e inequality fails as well.

\subsection{Genuinely singular measure potentials}

For measures $\mu$ and $\nu$, we write $\mu\perp\nu$ when they are mutually singular. Ahlfors-regular lower-dimensional sets give genuinely singular local examples. Here $\mathcal H^s$ denotes the $s$-dimensional Hausdorff measure.

\begin{proposition}
\label{prop:singular-local-example}
Let $Q_X>2$ and let $(X,d,\omega)$ be $Q_X$-Ahlfors regular. Fix $0<\delta<2$, set
\[
 s:=Q_X-2+\delta,
\]
and let $E\subset X$ be $s$-Ahlfors regular. Fix $x_0\in E$, $R>0$, and set
\begin{equation*}
 d\pi_R:=R^{-\delta}\,d\mathcal H^s\lfloor_E.
\end{equation*}
Then, on $B_0=B(x_0,R)$,
\begin{equation*}
 \pi_R(B(y,r))
 \le C\left(\frac rR\right)^\delta
 \frac{\omega(B(y,r))}{r^2},
 \qquad0<r\le R,
\end{equation*}
and
\begin{equation}
 \pi_R(B_0)\simeq\frac{\omega(B_0)}{R^2}.
 \label{eq:singular-normalization}
\end{equation}
Hence the capacity and mixed Poincar\'e conclusions of \cref{thm:growth-capacity,cor:growth-mixed} hold. If $s<Q_X$, then $\pi_R\perp\omega$.
\end{proposition}

\begin{proof}
Fix $y\in B_0$ and $0<r\le R$. If $E\cap B(y,r)=\varnothing$, then $\pi_R(B(y,r))=0$ and the growth estimate is immediate. Otherwise choose $z\in E\cap B(y,r)$. Then
\[
 E\cap B(y,r)\subset E\cap B(z,2r),
\]
so the upper $s$-Ahlfors regularity of $E$ gives
\[
 \pi_R(B(y,r))
 =R^{-\delta}\mathcal H^s(E\cap B(y,r))
 \le C R^{-\delta}r^s
 =C R^{-\delta}r^{Q_X-2+\delta}.
\]
Since $X$ is $Q_X$-Ahlfors regular,
$\omega(B(y,r))\simeq r^{Q_X}$, and hence
\[
 \pi_R(B(y,r))
 \le C\left(\frac rR\right)^\delta
 \frac{\omega(B(y,r))}{r^2}.
\]
At the distinguished center $x_0\in E$, the lower and upper regularity estimates for both $E$ and $X$ give
\[
 \pi_R(B_0)
 =R^{-\delta}\mathcal H^s(E\cap B(x_0,R))
 \simeq R^{s-\delta}
 =R^{Q_X-2}
 \simeq \frac{\omega(B_0)}{R^2},
\]
which is \eqref{eq:singular-normalization}. Finally, if $s<Q_X$, then $Q_X$-Ahlfors regularity implies $\omega(E)=0$, while $\pi_R$ is supported on $E$; therefore $\pi_R\perp\omega$.
\end{proof}

The following construction gives a global normalized cover for a purely singular measure.

\begin{proposition}
\label{prop:singular-global-example}
Let $d\ge3$, let $d\omega=dx$ on $\R^d$, and define the locally finite, purely singular Radon measure
\begin{equation}
 d\pi:=\sum_{k\in\mathbb Z}d\mathcal H^{d-1}\lfloor_{\{x_1=k\}}.
 \label{eq:periodic-hyperplane-measure}
\end{equation}
There exists a fixed-scale lattice cover $\{B_j\}$ and $m\equiv1$ such that $(\{B_j\},m)$ is a normalized growth-generated cover. In particular, the global conclusions of \cref{thm:global-fp,cor:energy-equivalence} apply to a measure $\pi\perp dx$.
\end{proposition}

\begin{proof}
Set $h:=d^{-1/2}$, $r_j:=1$, and take the unit-scale lattice balls
\[
 B_j:=B(hj,r_j),\qquad j\in\mathbb Z^d.
\]
Every point of $\R^d$ lies within distance at most $h\sqrt d/2=1/2$ of some point of $h\mathbb Z^d$, so the balls cover $\R^d$ pointwise; their fixed dilates have bounded overlap. For each center $hj$, its first coordinate lies within distance at most $1/2$ of some integer. Hence $B_j$ meets one of the hyperplanes in \eqref{eq:periodic-hyperplane-measure} in a $(d-1)$-disk of radius at least $\sqrt3/2$, while a unit ball meets at most three integer hyperplanes. Consequently, with constants depending only on $d$,
\[
 \pi(B_j)\simeq 1\simeq \frac{|B_j|}{r_j^2}.
\]
Moreover, for $y\in B_j$ and $0<r\le1$, at most three integer hyperplanes meet $B(y,r)$ with positive $(d-1)$-measure, and each cross-section has $(d-1)$-measure at most $Cr^{d-1}$. Therefore
\[
 \pi(B(y,r))\le C r^{d-1}
 =C r\frac{|B(y,r)|}{r^2}.
\]
Thus \eqref{eq:cover-growth} holds with $\delta=1$ at the unit scale, and \cref{cor:growth-generated-cover} applies. Since the support of \eqref{eq:periodic-hyperplane-measure} has zero Lebesgue measure, $\pi\perp dx$.
\end{proof}

The preceding examples establish that the natural abstract hypothesis is the $2$-PI property. The Euclidean $A_2$ class is an important verification class, not the maximal structural setting, while bare $A_p$ membership for $p>2$ does not control the quadratic geometry needed by the mixed theory.
\section{Quadratic forms and generalized degenerate Schr\"odinger operators}
\label{sec:operators}

We now pass from the energy inequalities of \cref{sec:verification} to operator theory.  The point of this section is that the two-sided Fefferman--Phong comparison identifies the natural form domain in two equivalent ways: either through the singular measure term $\int|\widetilde u|^2\,d\pi$ or through the ordinary weighted multiplier term $\int m^2|u|^2\,dw$.  This identification makes the closed-form realization and the local energy estimates transparent even when $\pi$ is singular with respect to $dw$.

Throughout this section, the hypotheses of \cref{thm:euclidean-verification} are in force.  We write $\rho=\rho_A$ for the critical radius and
\begin{equation*}
 m=\rho^{-1}
\end{equation*}
for its reciprocal critical-scale multiplier.

By \cref{thm:euclidean-verification,lem:global-representative-independence}, the measure $\pi$ does not charge weighted Sobolev $2$-capacity-zero sets. Hence every occurrence of $\widetilde u$ in an $L^2(d\pi)$ integral below is independent of the chosen weighted quasicontinuous representative.

Let $\mathbf A(x)$ be a real symmetric measurable $d\times d$ matrix satisfying
\begin{equation*}
 \alpha_{\mathbf A} w(x)|\xi|^2
 \le \mathbf A(x)\xi\cdot\xi
 \le \beta_{\mathbf A} w(x)|\xi|^2
\end{equation*}
for almost every $x\in\R^d$ and every $\xi\in\R^d$, where $0<\alpha_{\mathbf A}\le\beta_{\mathbf A}<\infty$. The underlying Hilbert space is
\begin{equation*}
 L_w^2:=L^2(\R^d,dw),
 \qquad (f,g)_{L_w^2}:=\int_{\R^d}f\,\overline g\,dw.
\end{equation*}
For $1\le p\le\infty$, we also write $L_w^p:=L^p(\R^d,dw)$.

The Radon measure $\pi$ is kept as a form measure. We do not write it formally as an additive differential expression. When the form measure is of the form $d\pi=w\,d\mu$ for a positive Radon measure $\mu$, the associated operator has the formal expression
\begin{equation*}
 -\frac1w\operatorname{div}(\mathbf A\nabla)+\mu.
\end{equation*}

\subsection{Concrete energy spaces}

Let
\begin{equation*}
 W_w^{1,2}(\R^d)
 :=\{u\in L_w^2:\nabla u\in L_w^2(\R^d;\C^d)\}.
\end{equation*}
Every element has a weighted $2$-quasicontinuous representative, denoted by $\widetilde u$. Define
\begin{align*}
 \Vpi
 &:=\{u\in W_w^{1,2}:\widetilde u\in L^2(d\pi)\},\\
 \Vm
 &:=\{u\in W_w^{1,2}:mu\in L_w^2\}.
\end{align*}
Their natural norms are
\begin{align}
 \|u\|_{\Vpi}^2
 &:=\|u\|_{L_w^2}^2+\|\nabla u\|_{L_w^2}^2+
 \|\widetilde u\|_{L^2(d\pi)}^2,
 \notag\\
 \|u\|_{\Vm}^2
 &:=\|u\|_{L_w^2}^2+\|\nabla u\|_{L_w^2}^2+
 \|mu\|_{L_w^2}^2.
 \label{eq:Vm-norm}
\end{align}

The critical multiplier is locally bounded above and below.

\begin{lemma}
\label{lem:m-local-bounds}
For every compact $K\subset\R^d$, there are constants $0<c_K\le C_K<\infty$ such that
\begin{equation*}
 c_K\le m(x)\le C_K,
 \qquad x\in K.
\end{equation*}
\end{lemma}

\begin{proof}
The critical balls from \cref{prop:critical-covering} cover $\R^d$ pointwise.  Choose finitely many of them, say $B_{j_1},\dots,B_{j_N}$, that cover $K$.  The pointwise comparison \eqref{eq:critical-m-comparison} gives
\[
 C^{-1}r_{j_\ell}^{-1}\le m(x)\le Cr_{j_\ell}^{-1},
 \qquad x\in B_{j_\ell}.
\]
Taking the minimum of the lower constants and the maximum of the upper constants over $\ell=1,\dots,N$ proves the claim. In particular, the infimum of $m$ over every compact ball is strictly positive.
\end{proof}

The two concrete form domains therefore coincide.

\begin{proposition}
\label{prop:concrete-domain-equivalence}
One has
\begin{equation*}
 \Vpi=\Vm
\end{equation*}
as sets, with equivalent norms.
\end{proposition}

\begin{proof}
Let $u\in\Vpi$. Then $u\in W_w^{1,2}$ and
$\widetilde u\in L^2(d\pi)$. The forward Fefferman--Phong estimate
\eqref{eq:euclidean-fp-forward} gives $mu\in L_w^2$, hence $u\in\Vm$.
Conversely, if $u\in\Vm$, then \eqref{eq:euclidean-fp-reverse} gives
$\widetilde u\in L^2(d\pi)$, so $u\in\Vpi$. Adding the common $L_w^2$ term to
the two Fefferman--Phong inequalities gives both norm comparisons and therefore
equivalence of the norms.

For completeness, $\Vm$ is complete. If $(u_n)$ is Cauchy in the norm
\eqref{eq:Vm-norm}, then $(u_n)$ is Cauchy in $W_w^{1,2}$ and $(mu_n)$ is
Cauchy in $L_w^2$. Write
\[
 u_n\to u\quad\text{in }W_w^{1,2},
 \qquad mu_n\to v\quad\text{in }L_w^2.
\]
After passing to a subsequence, $u_n\to u$ and $mu_n\to v$ almost everywhere.
Since $m$ is finite and positive everywhere, $v=mu$ almost everywhere. Hence
$u\in\Vm$ and $u_n\to u$ in the full $\Vm$ norm. Thus $\Vm$ is Hilbert, and
the norm equivalence transfers completeness to $\Vpi$.
\end{proof}

\subsection{Smooth density}

Smooth compactly supported functions are dense in the common energy space.

\begin{proposition}
\label{prop:smooth-core-density}
The space $C_c^\infty(\R^d)$ is dense in $\Vpi=\Vm$ in either of the equivalent norms.  More precisely, if $u\in\Vm$ has compact support and $U\Subset\R^d$ is open with $\supp u\Subset U$, then the approximating smooth functions may be chosen in $C_c^\infty(U)$.
\end{proposition}

\begin{proof}
We work first with $\Vm$. Let $u\in\Vm$, and choose standard cutoffs $\eta_R\in C_c^\infty(\R^d)$ satisfying
\[
 0\le\eta_R\le1,
 \quad \eta_R=1\text{ on }B(0,R),
 \quad \supp\eta_R\subset B(0,2R),
 \quad |\nabla\eta_R|\le C/R.
\]
The $L_w^2$ and multiplier terms converge by dominated convergence:
\[
 \|(1-\eta_R)u\|_{L_w^2}\longrightarrow0,
 \qquad
 \|m(1-\eta_R)u\|_{L_w^2}\longrightarrow0.
\]
For the gradient, the product rule gives
\[
 \nabla(\eta_Ru-u)=(\eta_R-1)\nabla u+u\nabla\eta_R.
\]
The first term tends to zero in $L_w^2$ by dominated convergence, while
\[
 \|u\nabla\eta_R\|_{L_w^2}^2
 \le \frac{C}{R^2}
 \int_{B(0,2R)\setminus B(0,R)}|u|^2\,dw
 \longrightarrow0.
\]
Thus $\eta_Ru\to u$ in $\Vm$, and compactly supported elements are dense.

Now let $u\in\Vm$ have compact support and fix an open set $U$ with $\supp u\Subset U$. Choose $\chi\in C_c^\infty(U)$ such that $\chi=1$ on a neighborhood of $\supp u$. By the weighted Sobolev density theorem for $A_2$ weights, see \cite[Theorem~1.1]{NakaiTomitaYabuta}, there are $v_n\in C_c^\infty(\R^d)$ with
\[
 v_n\longrightarrow u\quad\text{in }W_w^{1,2}(\R^d).
\]
Set $u_n:=\chi v_n\in C_c^\infty(U)$. Since $u=\chi u$,
\[
 u_n-u=\chi(v_n-u),
 \qquad
 \nabla(u_n-u)=\chi\nabla(v_n-u)+(v_n-u)\nabla\chi,
\]
so $u_n\to u$ in $W_w^{1,2}$.  All $u_n-u$ are supported in the fixed compact set $K:=\supp\chi$.  By \cref{lem:m-local-bounds},
\[
 \|m(u_n-u)\|_{L_w^2}
 \le C_K\|u_n-u\|_{L_w^2}
 \longrightarrow0.
\]
Hence $u_n\to u$ in $\Vm$.  The norm equivalence in \cref{prop:concrete-domain-equivalence} gives the same density in $\Vpi$.
\end{proof}

\subsection{The closed form and its operator}

Define
\begin{equation*}
 \mathfrak a_\pi(u,v)
 :=\int_{\R^d}\mathbf A\nabla u\cdot\overline{\nabla v}\,dx
 +\int_{\R^d}\widetilde u\,\overline{\widetilde v}\,d\pi,
 \qquad D(\mathfrak a_\pi)=\Vpi.
\end{equation*}
We use the standard quadratic-form notation $\mathfrak a_\pi[u]:=\mathfrak a_\pi(u,u)$.

A subspace $\mathcal C\subset\Vpi$ is called a \emph{form core} for $\mathfrak a_\pi$ if it is dense in $\Vpi$ for the form norm
\[
 \|u\|_{\mathfrak a_\pi}:=\bigl(\|u\|_{L_w^2}^2+\mathfrak a_\pi[u]\bigr)^{1/2}.
\]
The preceding energy equivalence gives a closed quadratic form and its canonical self-adjoint realization.

\begin{theorem}
\label{thm:closed-form-realization}
The form $(\mathfrak a_\pi,\Vpi)$ is densely defined, symmetric, nonnegative, and closed in $L_w^2$. Moreover, for every $u\in\Vpi=\Vm$,
\begin{equation}
 \|u\|_{L_w^2}^2+\mathfrak a_\pi[u]
 \simeq \|u\|_{\Vpi}^2
 \simeq \|u\|_{\Vm}^2,
 \label{eq:form-norm-equivalence}
\end{equation}
$C_c^\infty(\R^d)$ is a form core, and there is a unique nonnegative self-adjoint operator $L_\pi$ associated with $\mathfrak a_\pi$. Its domain is
\begin{equation*}
 D(L_\pi)=\left\{u\in\Vpi:
 \begin{array}{l}
 \text{there is }f\in L_w^2\text{ such that}\\[-1mm]
 \mathfrak a_\pi(u,v)=(f,v)_{L_w^2}
 \text{ for every }v\in\Vpi
 \end{array}\right\},
\end{equation*}
and $L_\pi u=f$.
\end{theorem}

\begin{proof}
The ellipticity condition gives
\[
 \alpha_{\mathbf A}\|\nabla u\|_{L_w^2}^2+
 \|\widetilde u\|_{L^2(d\pi)}^2
 \le\mathfrak a_\pi[u]
 \le\beta_{\mathbf A}\|\nabla u\|_{L_w^2}^2+
 \|\widetilde u\|_{L^2(d\pi)}^2.
\]
Together with \cref{prop:concrete-domain-equivalence}, this proves \eqref{eq:form-norm-equivalence} and closedness. Density and the core assertion follow from \cref{prop:smooth-core-density}. The representation theorem for closed nonnegative symmetric forms gives $L_\pi$.
\end{proof}

Write $C_0(\R^d)$ for the continuous functions on $\R^d$ that vanish at infinity. A Dirichlet form on $L_w^2$ is called \emph{regular} if its compactly supported continuous form-domain functions are dense both in the form norm and, uniformly, in $C_0(\R^d)$. With this terminology, the real restriction of $\mathfrak a_\pi$ is a regular Dirichlet form, and the complex form above is its standard complexification. A positivity-preserving semigroup $(T_t)_{t\ge0}$ is called \emph{sub-Markovian} if $0\le f\le1$ implies $0\le T_tf\le1$ for every $t\ge0$. Consequently, $e^{-tL_\pi}$ is positivity preserving and sub-Markovian. It therefore extends consistently as contractions to $L_w^p$ for $1\le p\le\infty$, and these extensions form strongly continuous contraction semigroups for $1\le p<\infty$.

To verify the Markov property, let $u\in\Vpi$ be real-valued and set
\[
 T(t):=\min\{1,\max\{0,t\}\}.
\]
The standard Sobolev truncation rule gives
\[
 \nabla T(u)=\one_{\{0<u<1\}}\nabla u
 \quad\text{a.e.}
\]
and hence the matrix part contracts pointwise:
\[
 \mathbf A\nabla T(u)\cdot\nabla T(u)
 =\one_{\{0<u<1\}}\mathbf A\nabla u\cdot\nabla u
 \le \mathbf A\nabla u\cdot\nabla u.
\]
Quasicontinuity and Lipschitz composition give
\[
 \widetilde{T(u)}=T(\widetilde u)
 \quad\text{quasi-everywhere}.
\]
Since $\pi$ vanishes on capacity-null sets, this identity holds $\pi$-almost everywhere, and $|T(\widetilde u)|\le|\widetilde u|$.  Thus
\[
 \mathfrak a_\pi[T(u)]\le\mathfrak a_\pi[u],
\]
which is the Markov property.  The same conclusion for arbitrary normal contractions follows from the usual contraction principle for Dirichlet forms.  Regularity follows from \cref{prop:smooth-core-density}: $C_c^\infty(\R^d)$ is dense in the form domain in the form norm and is uniformly dense in $C_0(\R^d)$.

\subsection{Resolvent energy estimates}

The associated resolvent satisfies the following energy bounds.

\begin{proposition}
\label{prop:resolvent-energy}
Let $\operatorname{Re}z>0$ and $u=(L_\pi+z)^{-1}f$, $f\in L_w^2$. Then
\begin{equation*}
 \|u\|_{L_w^2}\le(\operatorname{Re}z)^{-1}\|f\|_{L_w^2},
\end{equation*}
\begin{equation}
 \mathfrak a_\pi[u]+
 \frac{\operatorname{Re}z}{2}\|u\|_{L_w^2}^2
 \le\frac1{2\operatorname{Re}z}\|f\|_{L_w^2}^2,
 \label{eq:energy-resolvent}
\end{equation}
and
\begin{equation}
 \|\nabla u\|_{L_w^2}+\|\widetilde u\|_{L^2(d\pi)}+
 \|mu\|_{L_w^2}
 \le C(\operatorname{Re}z)^{-1/2}\|f\|_{L_w^2}.
 \label{eq:full-resolvent-energy}
\end{equation}
\end{proposition}

\begin{proof}
The resolvent identity in weak form is
\[
 \mathfrak a_\pi(u,v)+z(u,v)_{L_w^2}=(f,v)_{L_w^2},
 \qquad v\in\Vpi.
\]
Taking $v=u$ and real parts gives
\[
 \mathfrak a_\pi[u]+(\operatorname{Re}z)\|u\|_{L_w^2}^2
 =\operatorname{Re}(f,u)_{L_w^2}
 \le \|f\|_{L_w^2}\|u\|_{L_w^2}.
\]
Dropping the nonnegative form term yields
\[
 \|u\|_{L_w^2}
 \le (\operatorname{Re}z)^{-1}\|f\|_{L_w^2}.
\]
On the other hand, Young's inequality gives
\[
 \|f\|_{L_w^2}\|u\|_{L_w^2}
 \le \frac{1}{2\operatorname{Re}z}\|f\|_{L_w^2}^2
 +\frac{\operatorname{Re}z}{2}\|u\|_{L_w^2}^2,
\]
and substitution proves \eqref{eq:energy-resolvent}.

By ellipticity,
\[
 \|\nabla u\|_{L_w^2}^2+\|\widetilde u\|_{L^2(d\pi)}^2
 \le C\mathfrak a_\pi[u]
 \le C(\operatorname{Re}z)^{-1}\|f\|_{L_w^2}^2.
\]
Finally, \eqref{eq:euclidean-fp-forward} gives
\[
 \|mu\|_{L_w^2}^2
 \le C\left(\|\nabla u\|_{L_w^2}^2+
 \|\widetilde u\|_{L^2(d\pi)}^2\right).
\]
Taking square roots and combining the three bounds gives
\eqref{eq:full-resolvent-energy}.
\end{proof}

\subsection{The homogeneous energy problem}

We retain the notation $\Hpi$ and $\Hm$ from the abstract completions in \cref{cor:energy-equivalence}. In the present Euclidean setting the same spaces are obtained if one starts from $C_c^\infty(\R^d)$ instead of $\LipLip(\R^d)$. Indeed, let $u\in\LipLip(\R^d)$ have compact support in a compact set $K$. Weighted Sobolev density, followed by multiplication by a fixed smooth cutoff equal to one near $K$, gives $u_n\in C_c^\infty(\R^d)$ with supports in one compact set and $u_n\to u$ in $W_w^{1,2}$. By the local upper bound for $m$ from \cref{lem:m-local-bounds}, $m(u_n-u)\to0$ in $L_w^2$, so $u_n\to u$ in the $\Hm$ norm. The equivalence of the two homogeneous norms then gives the same density in $\Hpi$.

Thus, equivalently, $\Hpi$ is the completion of $C_c^\infty(\R^d)$ in
\begin{equation*}
 \|u\|_{\Hpi}^2
 :=\|\nabla u\|_{L_w^2}^2+
 \|\widetilde u\|_{L^2(d\pi)}^2,
\end{equation*}
and $\Hm$ is the completion in
\begin{equation*}
 \|u\|_{\Hm}^2
 :=\|\nabla u\|_{L_w^2}^2+
 \|mu\|_{L_w^2}^2.
\end{equation*}
The canonical identification is the one already supplied by \cref{cor:energy-equivalence}.

The homogeneous completion can also be identified with concrete local Sobolev functions. For a compact set $K\subset\R^d$, we use the local norm
\[
 \|u\|_{W_w^{1,2}(K)}^2
 :=\int_K|u|^2\,dw+\int_K|\nabla u|^2\,dw.
\]

\begin{lemma}
\label{lem:homogeneous-local-realization}
Every element of $\Hm$, and hence of $\Hpi$, has a canonical representative
\[
 u\in W_{w,\loc}^{1,2}(\R^d),
 \qquad \nabla u\in L_w^2,
 \qquad mu\in L_w^2.
\]
For every compact $K$,
\begin{equation}
 \|u\|_{W_w^{1,2}(K)}^2
 \le C_K\|u\|_{\Hm}^2.
 \label{eq:homogeneous-local-embedding}
\end{equation}
Under the canonical isomorphism $\Hpi\simeq\Hm$, the $L^2(d\pi)$ component of a completion element is the quasicontinuous representative $\widetilde u$. In particular, $\widetilde u\in L^2(d\pi)$ and the form $\mathfrak a_\pi(u,v)$ is represented by the concrete gradient and potential integrals.
\end{lemma}

\begin{proof}
Let $(u_n)\subset C_c^\infty(\R^d)$ be Cauchy in $\Hm$. By \cref{lem:m-local-bounds}, $m\ge c_K>0$ on each compact $K$. Hence
\[
 \|u_n-u_k\|_{W_w^{1,2}(K)}^2
 \le c_K^{-2}\|m(u_n-u_k)\|_{L_w^2}^2
 +\|\nabla(u_n-u_k)\|_{L_w^2}^2
 \le C_K\|u_n-u_k\|_{\Hm}^2.
\]
Thus $(u_n)$ converges in $W_w^{1,2}(K)$ for every compact $K$. Compatibility on nested compact sets yields a canonical $u\in W_{w,\loc}^{1,2}$. This representative is independent of the chosen approximating sequence: if $(v_n)\subset C_c^\infty(\R^d)$ represents the same completion element, then $\|u_n-v_n\|_{\Hm}\to0$, and the preceding local estimate gives $u_n-v_n\to0$ in $W_w^{1,2}(K)$ for every compact $K$. On each compact set the local Sobolev convergence identifies the global gradient component with $\nabla u$; local boundedness of $m$ identifies the global multiplier component with $mu$. Hence $\nabla u,mu\in L_w^2$, and \eqref{eq:homogeneous-local-embedding} follows.

The norm equivalence implies that $(u_n)$ is also Cauchy in $\Hpi$, so it has a limit $h\in L^2(d\pi)$ in the potential component. On each critical ball $B_j$, the local reverse Fefferman--Phong estimate \eqref{eq:local-fp-reverse}, applied to $u_n-u$, and the local $W_w^{1,2}$ convergence give
\[
 \|\widetilde u_n-\widetilde u\|_{L^2(B_j,d\pi)}\longrightarrow0.
\]
Therefore $h=\widetilde u$ almost everywhere on every $B_j$, and hence globally because the critical balls cover $\R^d$. This identifies the potential component and proves the final assertion.
\end{proof}

Because completion elements now have concrete local Sobolev representatives, the homogeneous equation can be solved directly by Lax--Milgram at zero spectral parameter.

\begin{theorem}
\label{thm:homogeneous-lax-milgram}
Suppose that $m^{-1}f\in L_w^2$. Then there is a unique $u\in\Hpi$ such that
\begin{equation*}
 \mathfrak a_\pi(u,v)=\int_{\R^d}f\overline v\,dw,
 \qquad v\in\Hpi,
\end{equation*}
where both sides are understood by continuous extension from $C_c^\infty$. Moreover,
\begin{equation}
 \|\nabla u\|_{L_w^2}+
 \|\widetilde u\|_{L^2(d\pi)}+
 \|mu\|_{L_w^2}
 \le C\|m^{-1}f\|_{L_w^2}.
 \label{eq:homogeneous-solution-bound}
\end{equation}
The solution is represented by a local weighted $L^2$ function through \cref{lem:homogeneous-local-realization}.
\end{theorem}

\begin{proof}
For $v\in C_c^\infty$,
\[
 \left|\int f\overline v\,dw\right|
 \le\|m^{-1}f\|_{L_w^2}\|mv\|_{L_w^2}
 \lesssim\|m^{-1}f\|_{L_w^2}\|v\|_{\Hpi}.
\]
Weighted ellipticity and Cauchy--Schwarz give
\[
 |\mathfrak a_\pi(u,v)|
 \le C\|u\|_{\Hpi}\|v\|_{\Hpi},
 \qquad u,v\in\Hpi,
\]
and
\[
 \mathfrak a_\pi[u]
 \ge \min\{\alpha_{\mathbf A},1\}\|u\|_{\Hpi}^2.
\]
Thus the form is continuous and coercive on $\Hpi$. Lax--Milgram gives existence, uniqueness, and the first two terms in \eqref{eq:homogeneous-solution-bound}; the third follows from \eqref{eq:euclidean-fp-forward}.
\end{proof}

\subsection{Local weak solutions}

For an open set $\Omega\subset\R^d$, $W_{w,\loc}^{1,2}(\Omega)$ denotes the functions whose value and Euclidean gradient are locally square-integrable with respect to $dw$, and $L^2_{\loc}(\Omega,d\pi)$ denotes local square integrability with respect to $\pi$.

\begin{definition}
\label{def:weak-solution}
Let $\Omega\subset\R^d$ be open and $\lambda_0\ge0$. A function
\[
 u\in W_{w,\loc}^{1,2}(\Omega),
 \qquad\widetilde u\in L_{\loc}^2(\Omega,d\pi),
\]
is a weak solution of $(L_\pi+\lambda_0)u=0$ in $\Omega$ if
\begin{equation}
 \int_\Omega \mathbf A\nabla u\cdot\overline{\nabla\varphi}\,dx+
 \int_\Omega\widetilde u\,\overline\varphi\,d\pi+
 \lambda_0\int_\Omega u\overline\varphi\,dw=0
 \label{eq:local-weak-formulation}
\end{equation}
for every $\varphi\in C_c^\infty(\Omega)$.
\end{definition}

Weak solutions in the sense of \cref{def:weak-solution} satisfy the standard energy estimate.

\begin{proposition}
\label{prop:Caccioppoli}
Let $u$ solve $(L_\pi+\lambda_0)u=0$ in $\Omega$ and let $\eta\in\LipLip(\Omega)$, $0\le\eta\le1$. Then
\begin{equation*}
 \int\eta^2|\nabla u|^2\,dw+
 \int\eta^2|\widetilde u|^2\,d\pi+
 \lambda_0\int\eta^2|u|^2\,dw
 \le C\int|u|^2|\nabla\eta|^2\,dw.
\end{equation*}
In particular, if $B(x,r)\subset B(x,R)\Subset\Omega$, then
\begin{equation*}
 \int_{B(x,r)}|\nabla u|^2\,dw+
 \int_{B(x,r)}|\widetilde u|^2\,d\pi+
 \lambda_0\int_{B(x,r)}|u|^2\,dw
 \le\frac{C}{(R-r)^2}\int_{B(x,R)}|u|^2\,dw.
\end{equation*}
\end{proposition}

\begin{proof}
Set $v:=\eta^2u$ on $\Omega$ and extend it by zero outside $\Omega$.  Because $\eta$ has compact support in $\Omega$ and $u\in W_{w,\loc}^{1,2}(\Omega)$, the product rule gives $v\in W_w^{1,2}(\R^d)$ with compact support in $\Omega$.  Moreover,
\[
 \widetilde v=\eta^2\widetilde u
 \quad\text{quasi-everywhere},
\]
so $v\in L^2(d\pi)$ because $\widetilde u\in L^2_{\loc}(\Omega,d\pi)$.  Hence $v\in\Vpi$.

Choose an open set $U$ with $\supp\eta\Subset U\Subset\Omega$.  The localized form-core statement in \cref{prop:smooth-core-density} gives $\varphi_k\in C_c^\infty(U)$ such that
\[
 \varphi_k\longrightarrow v
 \quad\text{in }W_w^{1,2}(\R^d)
 \quad\text{and in }L^2(d\pi).
\]
We may therefore pass to the limit in \eqref{eq:local-weak-formulation} and use $v=\eta^2u$ as a test function.  Taking real parts yields
\[
 \int \eta^2 \mathbf A\nabla u\cdot\overline{\nabla u}\,dx
 +\int\eta^2|\widetilde u|^2\,d\pi
 +\lambda_0\int\eta^2|u|^2\,dw
 =-2\operatorname{Re}\int \eta\,\overline u\,\mathbf A\nabla u\cdot\nabla\eta\,dx.
\]
By ellipticity and the upper bound on $\mathbf A$,
\[
 \left|2\operatorname{Re}\int \eta\,\overline u\,\mathbf A\nabla u\cdot\nabla\eta\,dx\right|
 \le 2\beta_{\mathbf A}\int \eta|u|\,|\nabla u|\,|\nabla\eta|\,dw.
\]
Young's inequality absorbs half of the gradient term on the left and gives
\[
 \int\eta^2|\nabla u|^2\,dw
 +\int\eta^2|\widetilde u|^2\,d\pi
 +\lambda_0\int\eta^2|u|^2\,dw
 \le C\int|u|^2|\nabla\eta|^2\,dw.
\]
Choosing the standard Lipschitz cutoff that equals one on $B(x,r)$, is supported in $B(x,R)$, and satisfies $|\nabla\eta|\lesssim(R-r)^{-1}$ gives the second assertion.
\end{proof}

The Caccioppoli estimate can now be combined with the global Fefferman--Phong comparison to control the critical multiplier on a smaller ball.

\begin{theorem}
\label{thm:local-critical-multiplier}
If $u$ solves $(L_\pi+\lambda_0)u=0$ in $B(x,R)$, then for $0<r<R$,
\begin{equation}
 \int_{B(x,r)}m(y)^2|u(y)|^2\,dw(y)
 \le\frac{C}{(R-r)^2}\int_{B(x,R)}|u|^2\,dw.
 \label{eq:local-critical-multiplier}
\end{equation}
\end{theorem}

\begin{proof}
Choose a cutoff $\eta$ supported in $B(x,R)$, equal to one on $B(x,r)$, with $|\nabla\eta|\lesssim(R-r)^{-1}$.  The zero extension of $\eta u$ belongs to $W_w^{1,2}(\R^d)$ and has $\widetilde{\eta u}=\eta\widetilde u\in L^2(d\pi)$.  Applying \eqref{eq:euclidean-fp-forward} to $\eta u$ gives
\[
 \int m^2|\eta u|^2\,dw
 \lesssim
 \int |\nabla(\eta u)|^2\,dw
 +\int \eta^2|\widetilde u|^2\,d\pi.
\]
Using $|\nabla(\eta u)|^2\lesssim \eta^2|\nabla u|^2+|u|^2|\nabla\eta|^2$ and then \cref{prop:Caccioppoli}, we obtain \eqref{eq:local-critical-multiplier}.
\end{proof}

By local comparability of the critical radius, the corresponding multiplier-free local $L^2(dw)$ consequence is genuinely stronger than the trivial inclusion estimate only in the regime where the local critical scale is smaller than the annular width, quantitatively $\rho_A(x)\lesssim R-r$ (up to the fixed local-comparability constants). No such restriction is needed for the multiplier estimate \eqref{eq:local-critical-multiplier} itself.

The conclusions of this section are intentionally confined to consequences established directly from the form and energy theory. Fundamental-solution estimates, heat-kernel H\"older bounds, perturbation identities, and Hardy-space characterizations require additional arguments and are not claimed here.
\section{Relation to earlier generalized Schr\"odinger results}
\label{sec:relation-earlier}

We explain how the mixed-measure theory interacts with the generalized Schr\"odinger results of \cite{BuiDoTrong2021,DoTruong2023}. The abstract metric content--capacity and normalized-cover theory developed above is independent of these applications. The main purpose is to relate the Euclidean capacity argument to the mixed $dw(x)d\pi(y)$ mechanism naturally associated with singular form measures and to give an independent route to the Fefferman--Phong estimate of \cite[Proposition~3.2]{BuiDoTrong2021}. A short final remark records the finite-scale comparison with \cite{DoTruong2023}.

A notational point is important for the first comparison. In \cite{BuiDoTrong2021}, the potential is written as a positive Radon measure $\mu$ and the measure entering the energy estimates is written as
\[
 d\pi=w\,d\mu.
\]
Whenever we use this notation below, we assume that the resulting measure $\pi$ is a positive Radon measure. The corresponding critical functions are denoted there by $\rho_w(x,\mu)$ and $m_w(x,\mu)=\rho_w(x,\mu)^{-1}$. We retain that notation below when referring to the earlier paper, while the abstract results above continue to be formulated directly in terms of the form measure $\pi$.

\subsection{The algebraically required oscillation}

For a ball $B$, set
\begin{equation*}
 \mathcal O_{\omega,\pi}(u;B)
 :=\int_B\int_B|u(x)-\widetilde u(y)|^2\,d\omega(x)\,d\pi(y).
\end{equation*}
The following elementary inequalities explain why this mixed oscillation, rather than an ordinary $d\omega\,d\omega$ oscillation, appears whenever one compares the $d\omega$ and $d\pi$ energies on the same ball.

\begin{lemma}
If $0<\omega(B),\pi(B)<\infty$, then, for every $u$ for which the displayed quantities are defined (with nonnegative integrals interpreted in $[0,\infty]$),
\begin{equation}
 \pi(B)\int_B|u|^2\,d\omega
 \le2\mathcal O_{\omega,\pi}(u;B)+
 2\omega(B)\int_B|\widetilde u|^2\,d\pi,
 \label{eq:mixed-algebra-forward}
\end{equation}
and
\begin{equation}
 \omega(B)\int_B|\widetilde u|^2\,d\pi
 \le2\mathcal O_{\omega,\pi}(u;B)+
 2\pi(B)\int_B|u|^2\,d\omega.
 \label{eq:mixed-algebra-reverse}
\end{equation}
\end{lemma}

\begin{proof}
Integrating
\[
 |u(x)|^2\le2|u(x)-\widetilde u(y)|^2+2|\widetilde u(y)|^2
\]
with respect to $d\omega(x)d\pi(y)$ proves \eqref{eq:mixed-algebra-forward}. The same argument applied to $|\widetilde u(y)|^2$ gives \eqref{eq:mixed-algebra-reverse}.
\end{proof}

The distinction is immaterial when the potential measure is sufficiently regular with respect to $\omega$, because the mixed term may then be reduced by other means. For a singular $\pi$, however, an ordinary $d\omega(x)d\omega(y)$ oscillation does not record the values of the quasicontinuous representative on an $\omega$-null set carrying $\pi$. This is exactly the situation for which the capacity argument of \cref{sec:content-capacity} is designed.

\subsection{The 2021 Fefferman--Phong estimate}

The Euclidean verification theorem, \cref{thm:euclidean-verification}, proves the two energy comparisons under $w\in A_2(\R^d)$ together with the scale-decay and controlled-enlargement assumptions. We isolate the mixed-measure step that arises when the quasicontinuous representative is kept explicit for a singular form measure, and then show that the exact threshold function used in \cite{BuiDoTrong2021} satisfies the same energy comparison under these weaker ambient assumptions. The Fefferman--Phong estimate itself is not being questioned; the point is that the present capacity argument supplies a direct representative-aware route to the same conclusion.

In \cite{BuiDoTrong2021}, the additional notation $w\in RD_\beta$ means that there is $c_{\mathrm{RD}}>0$ such that
\begin{equation*}
 w(B(x,tr))\ge c_{\mathrm{RD}}t^\beta w(B(x,r)),
 \qquad x\in\R^d,\ r>0,\ t\ge1.
\end{equation*}
The published Proposition~3.2 is stated under $w\in A_2(\R^d)\cap RD_\beta$, $\beta>2$, together with the measure assumptions \eqref{eq:global-scale-decay} and \eqref{eq:controlled-enlargement}.

Let $x_0\in\R^d$, set $r_0=\rho_w(x_0,\mu)$, and let $B=B(x_0,r_0)$. On pp.~9--10 of the published proof of \cite[Proposition~3.2]{BuiDoTrong2021}, the quantity
\[
 I:=\int_B\min\left\{\frac{w(B)}{r_0^2},\pi(B)\right\}|u(x)|^2\,dw(x)
\]
is introduced, and the following step is used:
\[
 I\lesssim
 \frac1{r_0^2}\int_B\int_B|u(x)-u(y)|^2\,dw(x)\,dw(y)
 +w(B)\int_B|u(y)|^2\,d\pi(y),
\]
after which the first term is controlled by the weighted Poincar\'e estimate.

For a general singular form measure $\pi$, keeping the quasicontinuous representative explicit in the second variable leads naturally to the mixed oscillation
\[
 \int_B\int_B|u(x)-\widetilde u(y)|^2\,dw(x)\,d\pi(y).
\]
The ordinary $dw(x)dw(y)$ Poincar\'e term does not itself encode the values of $\widetilde u$ on a $dw$-null set carrying $\pi$, whereas the capacity estimate of the present paper controls the mixed quantity directly. In this sense the present method makes the mixed-measure representative issue explicit and gives an independent representative-aware derivation of the same Fefferman--Phong energy estimate. Proposition~3.2 of \cite{BuiDoTrong2021} remains valid under its stated assumptions.

The exact critical function used in \cite{BuiDoTrong2021} is recovered as follows.

\begin{corollary}
\label{cor:recover-2021-FP}
Let $d\ge3$, let $w\in A_2(\R^d)$, and let $\mu$ be a positive Radon measure such that $d\pi:=w\,d\mu$ is a nonzero positive Radon measure. Assume \eqref{eq:global-scale-decay} and \eqref{eq:controlled-enlargement}. Let $C_1>0$ denote the fixed threshold constant used in \cite[(7)]{BuiDoTrong2021}, and define $\rho_w(\cdot,\mu)$ by that threshold formula,
\[
 \rho_w(x,\mu):=\sup\{r>0:\PhiPi(x,r)\le C_1\},
\]
and set $m_w(\cdot,\mu)=\rho_w(\cdot,\mu)^{-1}$. Then, for every $u\in W_{w,\loc}^{1,2}(\R^d)$,
\begin{align*}
 \int_{\R^d}m_w(x,\mu)^2|u(x)|^2\,dw(x)
 &\le C\left(\int_{\R^d}|\nabla u|^2\,dw+
 \int_{\R^d}|\widetilde u|^2\,d\pi\right),\\
 \int_{\R^d}|\widetilde u|^2\,d\pi
 &\le C\left(\int_{\R^d}|\nabla u|^2\,dw+
 \int_{\R^d}m_w(x,\mu)^2|u(x)|^2\,dw(x)\right),
\end{align*}
with all terms interpreted in $[0,\infty]$.
\end{corollary}

\begin{proof}
By \cref{lem:critical-normalization}, $0<\rho_w(x,\mu)<\infty$ for all $x$, because the positivity and finiteness statement holds for every positive threshold. By \cref{lem:threshold-equivalence}, for any fixed $A\ge A_*$,
\[
 \rho_w(x,\mu)=\rho_{C_1}(x)\simeq\rho_A(x),
 \qquad
 m_w(x,\mu)\simeq m_A(x)
\]
uniformly in $x$. The two inequalities therefore follow directly from \cref{thm:euclidean-verification}.
\end{proof}

Thus the exact critical multiplier defined by the threshold in \cite[(7)]{BuiDoTrong2021} satisfies the energy comparison under the weaker ambient assumption $w\in A_2$ alone. The additional condition $RD_\beta$, $\beta>2$, belongs to the broader kernel and Hardy-space framework of \cite{BuiDoTrong2021}, but it is not needed for the Fefferman--Phong energy comparison proved here.

The purpose of the preceding discussion is to identify the mixed-measure mechanism explicitly and to relate the present capacity framework to the critical function used in \cite{BuiDoTrong2021}; it is not to alter the statement of that proposition.

\subsection{Energy and form consequences}

Under the hypotheses of \cref{cor:recover-2021-FP}, the representative-aware results of \cref{sec:normalized-covers,sec:operators} give:
\begin{enumerate}[label=\textup{(\roman*)}]
\item equivalence of the homogeneous energies defined by the $L^2(d\pi)$ potential term and by the critical multiplier $m_w(\cdot,\mu)$;
\item equality, with equivalent norms, of the corresponding inhomogeneous form domains;
\item density of $C_c^\infty(\R^d)$ in the common form domain;
\item a closed nonnegative form and its nonnegative self-adjoint realization;
\item the local critical-multiplier estimate \eqref{eq:local-critical-multiplier} for weak solutions.
\end{enumerate}
These statements follow from \cref{prop:concrete-domain-equivalence,prop:smooth-core-density,thm:closed-form-realization,thm:local-critical-multiplier}, together with the threshold comparison in the proof of \cref{cor:recover-2021-FP}. In particular, the local trace estimates used in this route are formulated for quasicontinuous representatives and therefore remain meaningful when $\pi\perp dw$.

The later fundamental-solution, heat-kernel, perturbation, and Hardy-space results of \cite{BuiDoTrong2021} involve additional local regularity and kernel arguments. They are not consequences of the energy comparison alone and are not claimed here.

\subsection{A finite-scale comparison with the 2023 mixed inequality}

The generalized Poincar\'e inequality of \cite{DoTruong2023} is developed under $w\in A_1\cap RD_\beta$, $\beta>2$, and is used there together with additional convolution and spectral arguments. The finite-scale construction of \cref{prop:finite-scale-cover} has different, non-nested hypotheses and a different normalization: for a positive Radon measure $\nu$ and fixed parameters $M>0$ and $R_*>0$, it replaces $\nu$ by the augmented measure
\[
 d\nu_M=d\nu+MR_*^{-2}\,dw
\]
and produces a capped multiplier $m_{M,*}$. Applying \cref{prop:finite-scale-cover} with base measure $\pi=\nu$ and then \cref{cor:growth-generated-cover} to the augmented measure $\nu_M$ gives a bounded-overlap family of capped critical balls $B_j=B(x_j,r_j)$ for which
\[
 \nu_M(B_j)\simeq\frac{w(B_j)}{r_j^2}
\]
and the two global Fefferman--Phong inequalities hold with $\nu_M$ and $m_{M,*}$. Since $\nu\le\nu_M$, the corresponding local mixed estimate with $d\nu$ also follows, but its normalization coefficient remains $\nu_M(B_j)$; no $\nu$-normalized generalized Poincar\'e inequality is asserted.

This comparison is informative away from the hard cap. On a selected ball, $m_{M,*}^2\simeq r_j^{-2}$. For $0<\varepsilon<1$, if $r_j\le\varepsilon r_*=\varepsilon c_*R_*$, then the multiplier scale dominates the added background density $MR_*^{-2}$ by a factor of order $(M c_*^2\varepsilon^2)^{-1}$, up to fixed structural constants. Thus the subcap regime is not merely an artifact of the augmentation. Near the cap, by contrast, the background term may dominate. This observation is only a finite-scale comparison under different hypotheses and does not extend the spectral conclusions of \cite{DoTruong2023}.

\section*{Statements and Declarations}

\noindent\textbf{Funding statement.}
Tan Duc Do is supported by the Ho Chi Minh City University of Industry and Trade (HUIT), Vietnam.

\medskip
\noindent\textbf{Author contributions.}
Tan Duc Do conceived the study, developed and verified the mathematical arguments, and prepared the manuscript.

\medskip
\noindent\textbf{Data availability.}
No datasets were generated or analysed during the current study.

\medskip
\noindent\textbf{Competing interests.}
The author has no relevant financial or non-financial interests to disclose.

\medskip
\noindent\textbf{Generative-AI disclosure.}
ChatGPT, developed by OpenAI, was used during manuscript preparation for language editing, LaTeX formatting, and preliminary bibliographic searches. No mathematical claim or prior-art statement is accepted solely on that basis; the author assumes responsibility for independent verification of the final mathematical arguments, references, and text.

\end{document}